\documentclass[11pt,a4paper]{article}
\usepackage[utf8]{inputenc}
\usepackage[T1]{fontenc}
\usepackage[normalem]{ulem}

\usepackage[t,lf]{spectral}

\usepackage{amsmath}
\usepackage{amsfonts}
\usepackage{amsthm}
\usepackage{ upgreek }
\usepackage{mathrsfs}  

\usepackage{MnSymbol}   

\usepackage{marginnote}

\usepackage{csquotes}

\usepackage{stmaryrd}

\usepackage{enumitem}
\setlist[itemize]{parsep=0.5pt,leftmargin=17pt,itemindent=0pt}
\setlist[enumerate]{parsep=1pt,leftmargin=17pt,itemindent=0pt}

\usepackage[colorlinks=true, linkcolor=teal, linktocpage=true, citecolor=teal]{hyperref} %

\usepackage{mathtools}

\usepackage{cleveref}
\usepackage{autonum} 
\usepackage{calc}

\usepackage{pgf,tikz}
\usetikzlibrary{matrix,automata,arrows,positioning,calc,patterns,fit,calc,
decorations.pathreplacing,decorations.pathmorphing,decorations.markings,matrix,
tikzmark,shadows.blur}
\usetikzlibrary{cd}
 \usepgflibrary{fpu}
 
\usepackage{bbm}
\usepackage[margin=20pt,font=small,labelfont=sc]{caption}

\crefformat{equation}{(#2#1#3)}

\newtheorem{theorem}{Theorem}[section]
\newtheorem{proposition}[theorem]{Proposition}
\newtheorem{lemma}[theorem]{Lemma}
\newtheorem{corollary}[theorem]{Corollary}

\theoremstyle{definition}

\newtheorem{remarkx}[theorem]{Remark}
\newenvironment{remark}
  {\pushQED{\qed}\remarkx}
  {\popQED\endremarkx}
\newtheorem{assumptionx}[theorem]{Assumption}
\newenvironment{assumption}
  {\pushQED{\qed}\assumptionx}
  {\popQED\endassumptionx}

\numberwithin{equation}{section}

\input{macros.sty}

\newcommand{\Mp}{m_{\text p}}

\newcommand{\ph}{\varphi}

\DeclareMathOperator{\id}{id}

\begin{document}


\title{Diffusion stabilises time-periodic solutions\\ 
in conservation laws coupled to a relaxation oscillator}
\author{Julien Barr\'e, Nils Berglund and Hiroshi Horii}
\date{September 19, 2026} 

\maketitle

\begin{abstract}
We study a viscous one-dimensional conservation law, coupled to a fast ordinary differential equation. For a vanishing viscosity, the system has an infinite-dimensional family of time-periodic solutions, corresponding to relaxation oscillations. We show that 
for symmetric initial conditions, a small positive viscosity selects a unique periodic solution, which we prove to be linearly stable. The proof exploits the slow-fast structure through an averaging strategy, as well as spectral-theoretic 
methods. The results are illustrated by numerical simulations.
\end{abstract}

\leftline{\small 2020 {\it Mathematical Subject Classification.\/} 
35Q49, 
34C26 (primary), 
34K33, 
35P15 (secondary). 
}



\section{Introduction}
\label{sec:intro} 
Dusty plasmas are systems containing ionized gases and solid particles, typically of micrometer size, which are important both for fundamental and applied reasons, see for instance~\cite{beckers2023physics} for a recent review. Even in controlled experimental conditions, they show complex dynamics: a region free of particles (\lq\lq the void\rq\rq) typically forms~\cite{vladimirov2005stability,hu2009theory}; an instability can give rise to periodic oscillations of the 
void~\cite{zhdanov2010auto,couedel2010self,mikikian2010threshold,pikalev2021dim}, or even to mixed-mode oscillations (MMOs)~\cite{mikikian2008mixed}.


The physical mechanisms behind this complex dynamics are still under discussion (see for instance the recent work~\cite{pikalev2021heartbeat}), which impedes a precise mathematical modelling. Nevertheless, it is clear that any model should involve transport and diffusion of the dust particles, long-range coupling through the Poisson equation, and multiscale phenomena in space and time. The models we shall study, although clearly unrealistic for dusty plasmas, have these ingredients. 
 
 Relaxation oscillations, introduced by van der Pol in~\cite{vanderPol26}, are one of the simplest mechanisms giving rise to sustained oscillations and, under certain circumstances, to MMOs~\cite{MMO_review}. We shall thus couple relaxation oscillations with transport and diffusion of particles. We consider the following coupled system, where $n(t,x)$ is a density of particles, and $E(t),M(t)$ are global quantities coupled with $n$:
 \begin{align}
 \partial_t n(t,x) &= -a(E,M) \partial_x\bigbrak{f(n(t,x))} + \eps\partial_{xx}n(t,x)\;,
 \qquad x\in \mathbb{R}\;,  \label{eq:transport} \\
  \eta \dot{E} &= g(E,M)\;,  \label{eq:ODE_E} \\
   M(t) &= -\int_{-\infty}^0 n(t,x) \6x\;,  \label{eq:def_M} 
 \end{align}
 with boundary conditions 
\begin{equation}
\label{eq:bc} 
 \lim_{x\to\infty} n(t,x) = 1\;, \qquad 
 \lim_{x\to-\infty} n(t,x) = 0\;, \qquad 
 \lim_{\abs{x} \to \infty} \partial_x n(t,x) = 0~,
\end{equation} 
and complemented with appropriate initial conditions.
In the case of dusty plasmas, we think of $E(t)$ as representing the 
average electric field generated by the dust particles, ions and electrons, 
though other interpretations are possible.
The exact forms of the functions $a,g$ and $f$ do not matter, as long as the following qualitative characteristics are met:
\begin{itemize}
\item $a$ and $g$ are chosen such that the dynamics of $E,M$, which becomes autonomous for $\eps=0$, see Section~\ref{ssec:epsilon0}, can sustain relaxation oscillations.  
\item The flux function $f$ is chosen such that the equation for $n$ can sustain steep stationary step profiles between $n=0$ and $n=1$; in practice, we will use $f(n)=n(1-n)$.
\end{itemize} 
The parameter $\eta$ controls the time scale separation, while $\eps$ measures the 
intensity of diffusion. The parameters $\eps$ and $\eta$ satisfy $0 < \eta, \eps \ll 1$, and we 
will need the weak condition $\e^{-\kappa/\eta} = \Order{\eps^2}$ for some $\kappa > 0$, 
which means that $\eta$ has at most order $1/\log(\eps^{-1})$.
We will thus study a regime of small diffusion and large time scale separation between the dynamics of $n$ and $E$.  Note that the coupling between the PDE and the ODE is global, through the parameter $M$: this is in line with an idealization of the long-rang coupling provided by the Poisson equation in a plasma. 
Our goal is to study the existence and stability of periodic solutions in this setting, and our result can be informally stated as follows (see Theorem \ref{thm:main} for a precise statement): without diffusion ($\eps=0$), the system \eqref{eq:transport}--\eqref{eq:ODE_E}--\eqref{eq:def_M} admits an infinite-dimensional family of periodic solutions, indexed by the initial condition for $n$; for $\eps$ and $\eta$ positive and small enough, there exists a unique periodic solution, which is asymptotically linearly stable: it is in a sense selected by the small diffusion.

Studies of traveling front profiles for scalar conservation laws with small diffusion are classical and ubiquitous (see for instance~\cite{Mastumura_Nishihara94,Pan_Warnecke04,Beck_Wayne08,Johnson_Zumbrun11}, \cite[Chapter 6]{Serre_book} for a textbook account and \cite{Sandstede02} for an overview), 
and have been applied, in the fast--slow setting, to systems ranging from the FitzHugh--Nagumo equations for neuronal dynamics~\cite{Rinzel_Terman}, 
predator-prey systems~\cite{GGKKMO07}
and chemical reactions~\cite{Li_Liu26}
to turbulent pipe flows~\cite{Engel_2022}.
However, these systems do not show oscillating solutions. 
The work~\cite{sandstede2008hopf} studies a Hopf bifurcation in a system of viscous conservation laws, leading to periodic solutions. There is a large literature on periodic traveling waves and their stability, see for instance \cite{oh2003stability,serre2005spectral,johnson2014behavior}. We also note that \cite{greenberg1991time} exhibits a  non viscous system of two conservation laws with a space-time periodic solution. By contrast, the solutions we will study are not space periodic, nor are they propagating. Coupling a conservation law with an ODE opens other possibilities to create oscillatory patterns and is a natural modelling step in different situations (see~\cite{borsche2012mixed} for applications in traffic modelling and hydrodynamics for instance). However, all in all, we have not found references directly connected with the system \eqref{eq:transport}--\eqref{eq:ODE_E}--\eqref{eq:def_M}.

The remainder of this article is organised as follows.
A precise statement of the hypotheses and the results is given Section~\ref{sec:results}. Numerical simulations are presented in Section~\ref{sec:sim}. Section~\ref{sec:averaging} explains in detail an averaging procedure. In Section~\ref{sec:stability}, we find the stationary solutions of the resulting averaged equation, which correspond to periodic solutions in the original system, and prove their linear stability. Finally, in Section~\ref{sec:periodic_unperturbed} we show that the remainders, which 
are small when $\eps,\eta$ are small and which are neglected in the averaged equation, do not modify the existence and stability of the periodic orbits, while the appendix contains the 
proof of an existence and uniqueness result for the system~\eqref{eq:transport}-\eqref{eq:ODE_E}-\eqref{eq:def_M}.

\subsection*{Acknowledgments}

The authors would like to thank 
Ga\"etan Cane,
Laurent Desvillettes, 
Luc Hillairet,
Christian Kuehn, 
and Maxime Mikikian 
for helpful discussions. 
This project was supported by the ANR project PERISTOCH, ANR--19--CE40--0023. 


\section{Model and results}
\label{sec:results} 


\subsection{A slow--fast transport--ODE system}
\label{ssec:model} 

We start from the system \eqref{eq:transport}-\eqref{eq:ODE_E}-\eqref{eq:def_M}, with boundary conditions \eqref{eq:bc}.
The flux $f$ will be taken of the form
\begin{equation}
 f(n) = n(1-n) 
 \qquad
 \text{for $0 \leqs n \leqs 1$\;,}
\end{equation} 
while its value outside the interval $[0,1]$ will not matter in the analysis, 
since we will show that under suitable assumptions, $0 < n(t,x) < 1$ for all $t\geqs0$. 

We first provide a global existence result for solutions of the system~\eqref{eq:transport}, 
for quite general functions $a$ and $g$. We will make more specific assumptions 
ensuring the existence of relaxation oscillations in Assumption~\ref{assump:ag} below. 
To deal with the behaviour for $x\to\pm\infty$, it will be convenient to fix a smooth, strictly 
increasing function $n^*:\R\to[0,1]$, satisfying the boundary condition~\eqref{eq:bc}, and 
such that 
\begin{equation}
n^* \in H^2(\R^-) \cap L^1(\R^-)\;, \qquad  (1-n^*)  \in H^2(\R^+) \cap L^1(\R^+)\;. 
\end{equation} 
An example of such a function is 
$n^*(x) = \frac12(1 + \tanh(x))$.

\begin{theorem}[Global existence and uniqueness of solutions]
\label{thm:existence} 
Assume the functions $f$, $a$ and $g$ are globally Lipshitz continuous, 
and $a$ is bounded. Let $n_0: \R\to [0,1]$ satisfy 
$n_0 - n^* \in H^2(\R) \cap L^1(\R)$. 
Then for any $E_0 \in \R$, the system~\eqref{eq:transport} with initial conditions 
$n(0,x) = n_0(x)$ and $E(0) = E_0$ admits a unique solution, which is global in 
time, and satisfies 
$n(t,\cdot) - n^* \in H^2(\R) \cap L^1(\R)$
for all $t\geqs0$. Furthermore, $n(t,\cdot)$ satisfies the boundary 
conditions~\eqref{eq:bc} for all $t\geqs0$, 
and admits the conserved quantity 
\begin{equation}
\label{eq:conserved_quantityP} 
 P(t) 
 = \int_{-\infty}^\infty \bigbrak{n(t,x) - n^*(x)} \6x\;.
\end{equation} 
\end{theorem}

We give the proof in Appendix~\ref{app:existence}. We will later be able to relax the 
global Lipshitz and boundedness assumptions, by showing that for suitable 
initial conditions, the solution remains bounded uniformly in time. 

We will focus on the case where $n_0$ has a central symmetry. The following result 
shows that this symmetry is conserved over time.

\begin{proposition}[Symmetry]
\label{prop:symmetry} 
Assume that $f(n) = f(1-n)$ for all $n\in\R$, and that 
the initial condition satisfies the central symmetry 
\begin{equation}
\label{sym:n0} 
 n_0(-x) = 1 - n_0(x) 
 \qquad \forall x\in\R\;.
\end{equation} 
Then the solution of~\eqref{eq:transport} keeps this symmetry for all time, that is
\begin{equation}
\label{eq:symmetry0} 
 n(t,-x) = 1 - n(t,x) 
 \qquad \forall x\in\R, \forall t\geqs0\;.
\end{equation} 
\end{proposition}
\begin{proof}
Consider the function $\bar n(t,x) = 1 - n(t,-x)$. It satisfies  
$\bar n(0,\cdot) = n_0$ and 
\begin{equation}
 \partial_x \bar n(t,x) = \partial_x n(t,-x)\;, \qquad 
 \partial_{xx} \bar n(t,x) = -\partial_{xx} n(t,-x)\;.
\end{equation} 
This implies
\begin{align}
 \partial_t \bar n(t,x) 
 &= - \partial_t n(t,-x) \\
 &= a(E,M) \partial_x \bigbrak{f(n(t,-x))} - \eps\partial_{xx}n(t,-x) \\
 &= -a(E,M) \partial_x \bigbrak{f(\bar n(t,x))} + \eps\partial_{xx}\bar n(t,x)\;,
\end{align}
where we have used the fact that $f'(n) = -f'(1-n)$. Now let 
\begin{equation}
 \bar M(t) = -\int_{-\infty}^0 \bar n(t,x)\6x
 = -\int_0^\infty \bigbrak{1-n(t,x)}\6x\;.
\end{equation} 
Then $\bar M(0) = M(0)$, and 
\begin{equation}
 \frac{\6}{\6t} \bigbrak{M(t)- \bar M(t)}
 = - \frac{\6}{\6t} 
 \biggbrak{\int_{-\infty}^0 n(t,x)\6x
 + \int_0^\infty \bigbrak{n(t,x)-1}\6x}
 = -\dot P(t)
 = 0\;.
\end{equation}
Therefore, $\bar M(t) = M(t)$ for all $t\geqs0$. 
It follows that $\bar n(t,x)$ satisfies the same equation as $n(t,x)$, 
with the same initial condition. By uniqueness of solutions 
of~\eqref{eq:transport}-\eqref{eq:ODE_E}-\eqref{eq:def_M}, 
one has $\bar n(t,x) = n(t,x)$ for all $x\in\R$ and all $t\geqs0$. 
This is equivalent to~\eqref{eq:symmetry0}. 
\end{proof}

The quantity $M(t)$ measures the total amount of mass in $(-\infty,0]$. 
With the symmetry~\eqref{eq:symmetry0}, this means that smaller values of 
$\abs{M(t)}$ imply a steeper profile $n(t,x)$ near $x = 0$. The limiting 
case $M(t) = 0$ corresponds to a jump at $x = 0$, that is, a shock.  
Notice that the symmetry also implies that the conserved quantity $P$ 
defined in~\eqref{eq:conserved_quantityP} is equal to $0$.

From now on, we are always going to assume that $n_0$ satisfies the following 
properties. 

\begin{assumption}[Properties of the initial condition $n_0$] \hfill
\label{assump:n0} 
\begin{itemize}
\item   \textbf{Regularity:} $n_0 - n^*\in H^2(\R)\cap L^1(\R)$. 

\item   \textbf{Symmetry:} $n_0(-x) = 1 - n_0(x)$ for all $x\in\R$. 

\item   \textbf{Monotonicity:} $\partial_x n_0(x) > 0$ for all $x\in\R$. 
\qed
\end{itemize}
\renewcommand{\qed}{}
\end{assumption}

Note that these assumptions imply in particular that $n_0$ satisfies the 
boundary conditions~\eqref{eq:bc}, cf.\ the argument around~\eqref{eq:Fourier-L1} 
in Appendix~\ref{app:existence}. We assume symmetry to simplify the computations. We do not expect non-symmetric initial conditions to present a different qualitative behavior.

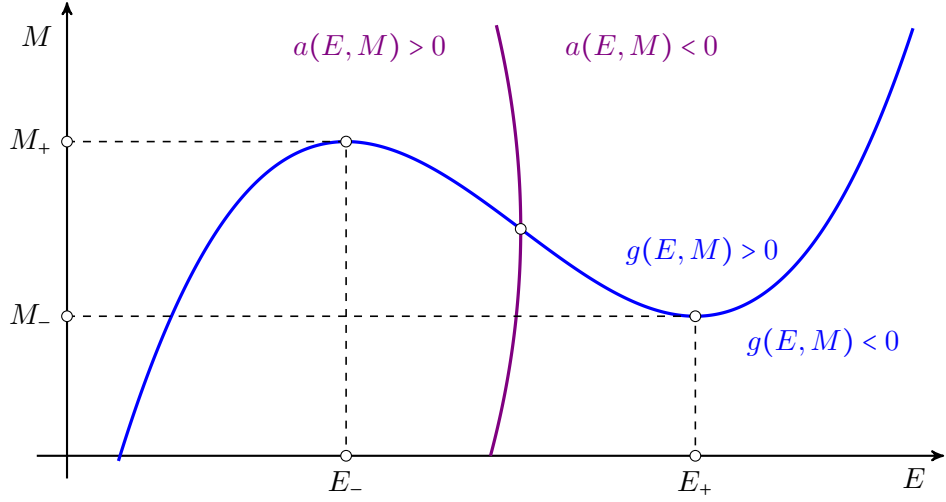
\begin{figure}
 \begin{center}
\scalebox{1.0}{
\begin{tikzpicture}[>=stealth',main node/.style={circle,minimum
size=0.01cm,inner sep=0.05cm,fill=white,draw},x=4cm,y=3cm]

\draw[->,thick] (-0.1,0) -> (2.9,0);
\draw[->,thick] (0,-0.1) -> (0,2);

%

\newcommand*{\Ez}{1.5}
\newcommand*{\Mz}{1.0}

\pgfmathsetmacro{\Em}{\Ez - 1/sqrt(3)};
\pgfmathsetmacro{\Mm}{\Mz + 2/(3*sqrt(3))};
\pgfmathsetmacro{\Ep}{\Ez + 1/sqrt(3)};
\pgfmathsetmacro{\Mp}{\Mz - 2/(3*sqrt(3))};

\draw[blue,very thick,-,smooth,domain=0.17:2.8,samples=500,/pgf/fpu,
/pgf/fpu/output format=fixed] plot (\x, {
\Mz + (\x - \Ez)^3 - (\x-\Ez)});

\draw[violet,very thick,-,smooth,domain=0.0:1.9,samples=300,/pgf/fpu,
/pgf/fpu/output format=fixed] plot ({\Ez - 0.1*(\x-\Mz)^2}, \x);

\draw[semithick, dashed] (0,{\Mm}) -- ({\Em},{\Mm}) -- ({\Em},0);
\draw[semithick, dashed] (0,{\Mp}) -- ({\Ep},{\Mp}) -- ({\Ep},0);

\node[main node] at (\Em,\Mm) {};
\node[main node] at (\Em,0) {};
\node[main node] at (0,\Mm) {};

\node[main node] at (\Ep,\Mp) {};
\node[main node] at (\Ep,0) {};
\node[main node] at (0,\Mp) {};

\node[main node] at (\Ez,\Mz) {};

\node[] at (2.8,-0.1) {$E$};
\node[] at (-0.1,1.85) {$M$};

\node[] at ({\Em},-0.12) {$E_-$};
\node[] at ({\Ep},-0.12) {$E_+$};

\node[] at (-0.12,{\Mm}) {$M_+$};
\node[] at (-0.12,{\Mp}) {$M_-$};

\node[blue] at (2.5,0.5) {$g(E,M) < 0$};
\node[blue] at (2.1,0.9) {$g(E,M) > 0$};

\node[violet] at (1.0,1.8) {$a(E,M) > 0$};
\node[violet] at (1.9,1.8) {$a(E,M) < 0$};

\end{tikzpicture}
}
\vspace{-5mm}
\end{center}
\caption[]{Geometry of the nullclines $g(E,M) = 0$ (or $M=h(E)$, critical manifold) 
and $a(E,M) = 0$ (or $E=k(M)$, slow nullcline).}
\label{fig:nullclines} 
\end{figure}

We give now more detailed assumptions on the functions $a$ and $g$ in~\eqref{eq:transport} 
and~\eqref{eq:ODE_E}:

\begin{assumption}[Properties of $a$ and $g$] \hfill
\label{assump:ag} 
\begin{itemize}
\item   \textbf{Regularity:} $a:\R\to\R$ and $g:\R^2\to\R$ are of class $\cC^2$. 

\item   \textbf{$S$-shaped critical manifold:} There exists a $\cC^2$ function $h:\R\to\R_+$ 
such that 
\begin{align}
 g(E,M) > 0 & \quad\Leftrightarrow\quad M > h(E)\;,\\
 g(E,M) < 0 & \quad\Leftrightarrow\quad M < h(E)\;.
\label{eq:def_h} 
\end{align} 
Furthermore, there exist $E_+ > E_-$ such that the function $h$ satisfies 
$h'(E) < 0$ on $(E_-,E_+)$ and 
$h'(E) > 0$ on $(-\infty,E_-) \cup (E_+,\infty)$, as well as  
\begin{equation}
 h''(E_-) < 0 
 \qquad \text{and} \qquad 
 h''(E_+) > 0\;.
\end{equation} 

\item   \textbf{Fold points:} One has 
\begin{equation}
 \partial_M g(E_\pm, h(E_\pm)) \neq 0
 \qquad \text{and} \qquad 
 \partial_{EE} g(E_\pm, h(E_\pm)) \neq 0\;.
\end{equation} 

\item   \textbf{Slow nullcline:} There exists a $\cC^2$ function $k:\R\to\R$ such that 
\begin{align}
 a(E,M) > 0 & \quad\Leftrightarrow\quad E < k(M)\;,\\
 a(E,M) < 0 & \quad\Leftrightarrow\quad E > k(M)\;.
\end{align} 
Furthermore, the nullclines $\set{M = h(E)}$ and $\set{E = k(M)}$ intersect 
exactly once, between the points of abscissa $E_-$ and $E_+$. 
\qed
\end{itemize}
\renewcommand{\qed}{}
\end{assumption}

These assumptions ensure that for $\eps = 0$, the variables $(E,M)$ display 
relaxation oscillations for small $\eta$, like those occurring for the 
Van der Pol oscillator. See Figure~\ref{fig:nullclines}, where we also 
use the notations $M_+ = h(E_-)$ and $M_- = h(E_+)$.


\subsection{The case $\eps = 0$}
\label{ssec:epsilon0} 

Substituting~\eqref{eq:transport} into~\eqref{eq:def_M} shows 
that $M(t)$ satisfies the ODE 
\begin{equation}
\label{eq:ODE_M} 
 \dot{M} = \frac14 a(E,M) - \eps \partial_x n(t,0)\;,
\end{equation} 
where the factor $\frac14$ is due to the fact that by conservation of symmetry
(Proposition~\ref{prop:symmetry}), one has $n(t,0) = \frac12$ for all $t\geqs0$, 
and $f(\frac12) = \frac14$. 
For $\eps = 0$, the variables $M$ and $E$ satisfy the closed system 
\begin{align}
\label{eq:ODE_ME_eps0} 
 \dot{M} &= \frac14 a(E,M) \\
 \eta\dot{E} &= g(E,M)\;.
\end{align}
Our assumptions on $a$ and $g$ ensure that for $\eta$ small enough, this equation 
admits a unique limit cycle $\Gamma$, attracting orbits in the whole $(E,M)$ plane 
exponentially fast. It is known that $\Gamma$ is at distance at most $\eta^{1/3}$ 
of the piecewise smooth curve consisting of the two stable parts of the 
critical manifold located between $M = M_-$ and $M = M_+$, completed by two 
horizontal lines~\cite{PontRod,Haberman}, see Figure~\ref{fig:periodic} below. 
Let $T$ denote the period of this limit cycle, and let 
$(E_p(t),M_p(t))_{t\in\R}$ be a periodic solution living on $\Gamma$. 

\begin{proposition}[Asymptotic behaviour for $\eps = 0$]
\label{prop:eps0} 
Assume $\eps = 0$. Then, there exists a constant $C$ depending on $M(0)$ and $E(0)$ 
such that, if 
\begin{equation}
\label{eq:bound_dxn} 
 \partial_x n(0,x) \leqs C \qquad \forall x\in\R\;,
\end{equation} 
then for sufficiently small $\eta$, the solution 
of~\eqref{eq:transport}--\eqref{eq:ODE_E}--\eqref{eq:def_M} with any initial condition satisfying 
Assumption~\ref{assump:n0} converges to a periodic function of time.
\end{proposition}
\begin{proof}
For $\eps = 0$, the equation~\eqref{eq:transport} reduces to the transport equation
\begin{equation}
\partial_t n(t,x) = -a(E(t), M(t)) \partial_x\bigbrak{n(t,x)(1-n(t,x))}\;.
\end{equation} 
This equation can be solved by the method of characteristics. We do this in a slightly 
unusual way, that will be particularly useful later on. Let 
\begin{equation}
\label{eq:def_tau} 
 \tau = \sup\bigsetsuch{t > 0}{0 < \partial_x n(s,x) < \infty  
 \text{ for all $x\in\R$ and all $s\in[0,t]$}}\;.
\end{equation} 
For $0\leqs t < \tau$, we can define the reciprocal function $X(t,n): [0,\tau)\times[0,1]\to\R$ 
as the unique solution of 
\begin{equation}
 X(t,n(t,x)) = x\;.
\end{equation} 
Then the relations 
\begin{equation}
 \partial_n X \partial_x n = 1\;, \qquad 
 \partial_t X + \partial_n X \partial_t n = 0
\end{equation} 
imply 
\begin{equation}
 \partial_t X(t,n) = a(E(t), M(t)) (1-2n)\;.
\end{equation} 
This equation can be integrated, leading to 
\begin{equation}
\label{eq:X_periodic} 
 X(t,n) = X(0,n) + (1-2n) \int_0^t a(E(s), M(s))\6s\;.
\end{equation} 
Since $(E(t),M(t))$ converges exponentially fast to $\Gamma$, one can 
decompose 
\begin{equation}
 a(E(t), M(t)) = a(E_p(t+\theta), M_p(t+\theta)) + R(t)\;,
\end{equation} 
where $\theta$ is a phase shift and $R(t)$ decreases exponentially fast. 
This implies that for any $t\geqs0$, 
\begin{equation}
 \int_t^{t+T} a(E(s), M(s))\6s
 = \int_t^{t+T} a(E_p(s+\theta), M_p(s+\theta))\6s + \int_t^{t+T} R(s)\6s 
 = \Order{\e^{-\gamma t}} 
\end{equation} 
for some $\gamma > 0$, where we have used the fact that 
\begin{align}
\label{eq:integral_aE} 
\int_t^{t+T} a(E_p(s+\theta), M_p(s+\theta)) \6s
&= 4\int_t^{t+T} \dot M_p(s+\theta)\6s \\
&= 4 \bigbrak{M_p(t+T+\theta) - M_p(t+\theta)}
= 0\;.
\end{align} 
Therefore, $X(t,\cdot)$ converges to a $T$-periodic 
function of time, and so does $n(t,\cdot)$. Furthermore, we have 
\begin{align}
 \partial_n X(t,n) 
 &= \partial_n X(0,n) - 2\int_0^t a(E(s), M(s))\6s \\
 &= \partial_n X(0,n) - 8\bigbrak{M(t) - M(0)}\;.
\label{eq:bound_dnX} 
\end{align} 
The assumption~\eqref{eq:bound_dxn} implies that $\partial_n X(0,n) \geqs 1/C$.  
Since the minimal value of $M(t)$ is entirely determined by the initial condition 
$(E(0),M(0))$, we can find a value of $C$ such that $\partial_n X(t,n) > 0$ for all 
$t\geqs0$. This implies that $0 < \partial_x n(t,x) < \infty$,
and thus $\tau = \infty$, i.e., there is no shock. 
\end{proof}

The expression~\eqref{eq:X_periodic} for the reciprocal function of $n(t,x)$ 
shows that $n(t,x)$ converges to a periodic modulation of $n(0,x)$. More 
precisely, we have 
\begin{equation}
 n(t,x) = n_0(Y(t,x))\;,
\end{equation} 
where $Y(t,x)$ is the unique $y$ solution of 
\begin{equation}
 y + \bigbrak{1 - 2n_0(y)}\int_0^t a(E(s), M(s))\6s = x\;,
\end{equation} 
which converges to a $T$-periodic function of time. In other words, 
the system~\eqref{eq:transport}--\eqref{eq:ODE_E}--\eqref{eq:def_M} has an uncountably infinite 
family of limit cycles, as each initial profile $n_0(x)$ leads to a different 
limit cycle. Note, however, that all these limit cycles share 
the same period.

\begin{remark}
If $M(0) \in [M_-, M_+]$, then it is known that $M(t)$ remains in the 
interval $[M_- - \Order{\eta^{2/3}}, M_+ + \Order{\eta^{2/3}}]$.
Then~\eqref{eq:bound_dnX} shows that one can take 
\begin{equation}
 C = \frac{1}{8[M_+ - M_- + \Order{\eta^{2/3}}]}
\end{equation} 
as upper bound on $\partial_x n(0,x)$. 
\end{remark}


\subsection{Main result and discussion}
\label{ssec:main_result} 

We can now state the main result of this work, which says that for sufficiently 
small, positive $\eps$, the system~\eqref{eq:transport}--\eqref{eq:ODE_E}--\eqref{eq:def_M} 
with Assumptions \ref{assump:n0} and \ref{assump:ag} has 
an isolated limit cycle, which is asymptotically stable. 
The effect of the diffusion term 
$\eps\partial_{xx}n$ is thus to almost completely lift the degeneracy, 
by collapsing the infinite-dimensional set of limit cycles to a single 
one.

\begin{theorem}[Main result]
\label{thm:main} 
Let Assumptions~\ref{assump:n0} and \ref{assump:ag} hold, and assume 
$0<\eps\ll 1$, $\eta \ll 1$. 
Then the system~\eqref{eq:transport}--\eqref{eq:ODE_E}--\eqref{eq:def_M}
admits an isolated periodic orbit $(n_p(t,x),E_p(t))$. This periodic orbit is linearly 
asymptotically stable, and therefore attracts neighboring solutions exponentially fast.
Furthermore, its period converges, as $\eps\to0$, to the common period of 
the infinite-dimensional family of limit cycles present when $\eps = 0$.
More precisely, there exist constants $\eps_0, c, C, M_0, M, \Delta_0,\Delta, \delta > 0$ such that, 
if $0 < \eps < \eps_0$ and $0\leqs \eta < c/\log(\eps^{-1})$, then for any initial condition $n_0$ satisfying, in addition to Assumption~\ref{assump:n0}, 
\begin{equation}
 \label{eq:bound_dxn2} 
 \partial_x n_0(x) \leqs C \qquad \forall x\in\R\;, 
\end{equation} 
and for some $t_0\in \R$:
\begin{equation}
\label{eq:n0_H2} 
 \bignorm{n_0 - n_p(t_0,\cdot)}_{H^2} \leqs M_0\;, 
 \qquad 
 \abs{E_0-E_p(t_0)}\leqs \Delta_0\;,
\end{equation}
the solution $n$ of~\eqref{eq:transport}--\eqref{eq:ODE_E}--\eqref{eq:def_M} 
is strictly increasing in $x$ for all $t\geqs0$ and $n, E$ satisfy 
\begin{align}
\label{eq:bound_normH2_diff} 
 \bignorm{n(t,\cdot) - n_p(t+t_0,\cdot)}_{H^2} 
 &\leqs M \e^{-\delta\eps t} \bignorm{n_0 - n_p(t_0,\cdot)}_{H^2}\;,\\
 \abs{E(t)-E_p(t+t_0)} & \leqs  \Delta  \e^{-\delta\eps t}\abs{E_0-E_p(t_0)}\;.
\end{align} 
Furthermore, $0 < n(t,x) < 1$ for all $x\in\R$ and all $t\geqs0$. 
\end{theorem}

We point out that the fact that $n$ remains bounded allows to remove the global Lipschitz hypotheses for the functions $a,f,g$ in Theorem \ref{thm:existence}, for initial conditions $n_0$ satisfying~\eqref{eq:n0_H2}.

One natural question is whether the periodic orbit we find here is unique. 
While we are not able to answer this question fully, our result shows at least 
that this limit cycle is isolated in a ball in $H^2$-norm. Furthermore, it is 
the only fixed point of a Poincar\'e map constructed by an averaging procedure, 
so that it is the only orbit with a period matching that of the limiting 
solutions given in Proposition~\ref{prop:eps0}. It is thinkable, however, 
that orbits of longer period exist, which would appear as periodic orbits 
of the Poincar\'e map, or that initial conditions that have a large slope 
or are non-monotonous lead to a more complex behaviour.

Another question is what kind of different behaviours can be observed 
if some of the assumptions are relaxed. We explore some of these questions 
in Section~\ref{sec:sim} via numerical simulations. It appears that the 
time separation parameter $\eta$ can be taken quite large, without changing 
the qualitative features of the result. 

Another way to test the limits of our result is to change the relative positions 
of the nullclines $g(E,M) = 0$ and $a(E,M) = 0$. As we will illustrate in 
Section~\ref{ssec:transient} below, 
when the slow nullcline $a(E,M) = 0$ intersects the critical curve $g(E,M) = 0$ 
closer to one of its extrema, the system transitions from large-amplitude 
to small-amplitude oscillations. This could be an indication that more 
complex mixed-mode oscillations are possible in variants of our model. 
Indeed, while two-dimensional FitzHugh--Nagumo ordinary differential equations do not 
support MMOs, the situation changes in the presence of non-autonomous perturbations~\cite{Longo_Queirolo_Kuehn_24}. As discussed in Proposition~\ref{prop:radial}
below, the coupling of $M$ to $\partial_x n(t,0)$ introduces such an effective 
time-dependent forcing in the system. Another source of MMOs in higher-dimensional systems 
are folded nodes and folded saddles~\cite{GuckHW,GuckHaiduc,KrupaWechsFSN,MMO_review}. 
MMOs in slow--fast systems of dimension $2$ and higher can also be created by additive noise. 
This was shown for a two-dimensional stochastic FitzHugh--Nagumo equation 
in~\cite{BerglundLandon}, while the effect of noise on MMOs near folded nodes 
in three-dimensional systems was investigated in~\cite{BGK12,Berglund_Gentz_Kuehn_2015}.


\section{Numerical simulations}
\label{sec:sim} 

In this section, we present numerical simulations of the transport-ODE system 
\eqref{eq:transport}--\eqref{eq:ODE_E}--\eqref{eq:def_M}.
To simulate the system, we introduce a finite-domain version of~\eqref{eq:transport}--\eqref{eq:ODE_E}--\eqref{eq:def_M} on the interval $[0,L]$, namely
\begin{equation}
\label{eq:num_transport}
 \partial_t n(t,x) = -\partial_x F(n(t,x),E(t)) + \eps \partial_{xx}n(t,x)\;,
\end{equation}
where
\begin{equation}
\label{eq:num_flux}
 F(n,E)=a_0 n(1-n)(E_a-E)\;.
\end{equation}
Equivalently, this corresponds to the choice
\begin{equation}
 a(E,M)=a_0(E_a-E)
\end{equation}
in~\eqref{eq:transport}. The fast variable is written in the same form as in~\eqref{eq:ODE_E},
\begin{equation}
\label{eq:num_fast}
 \eta \dot E = g(E,M)\;,
\end{equation}
with
\begin{equation}
\label{eq:num_g_explicit}
 g(E,M)=-b_3(E-E^\star)^3+b_1(E-E^\star)-b_0+cM\;.
\end{equation}
In the numerical simulations, we use the following mass-type variable in the finite-domain system:
\begin{equation}
\label{eq:num_M}
 M^{\rm sim}(t)=-\int_0^{L/2} n(t,x)\6x\;.
\end{equation}
With this convention, values of \(M^{\rm sim}(t)\) closer to zero correspond to a smaller mass in the left half of the domain, while more negative values correspond to a larger mass. The local maximum and local minimum values of \(M^{\rm sim}(t)\) therefore correspond to two different states of the periodic density profile. The critical manifold of the fast equation is defined by \(g(E,M^{\rm sim})=0\). For the cubic nonlinearity, this condition is equivalent to
\begin{equation}
\label{eq:num_h}
 M^{\rm sim}=\frac{b_3(E-E^\star)^3-b_1(E-E^\star)+b_0}{c}\;.
\end{equation}
This is the curve shown as the dashed line in the phase portraits below. Then, the density profiles are computed with a Lax-type finite-difference scheme.
We set
\begin{equation}
 \Delta x=\frac{L}{N}\;,\qquad x_i=i\Delta x\;,\qquad t_k=k\Delta t\;.
\end{equation}
We denote by $n_i^k$ the discrete density variable at the grid point $x_i$ and time $t_k$, and by $E^k$ the discrete field variable at time $t_k$. At each time step, $M_{\rm sim}^k$ is first computed from the current density profile. Then the field variable is updated by
\begin{equation}
\label{eq:num_scheme_E}
 E^{k+1}=E^k+\frac{\Delta t}{\eta}g(E^k,M_{\rm sim}^k)\;.
\end{equation}
The density is then updated using the updated field value $E^{k+1}$, for $1\leqs i\leqs N-1$, by
\begin{equation}
\label{eq:num_scheme_n}
 n_i^{k+1}=(1-\mu)n_i^k+\frac{\mu}{2}\Bigpar{n_{i+1}^k+n_{i-1}^k}-\frac{\Delta t}{2\Delta x}\Bigpar{F_{i+1}^k-F_{i-1}^k}\;,
\end{equation}
where $F_i^k=F(n_i^k,E^{k+1})$. Here $\mu$ denotes the dimensionless Lax diffusion parameter. Comparing this scheme with the standard explicit discretisation of the diffusion term $\eps\partial_{xx}n$, one has
\begin{equation}
\label{eq:num_mu_eps}
 \mu = 2\eps\frac{\Delta t}{\Delta x^2}\;.
\end{equation}
We also impose Dirichlet boundary conditions,
\begin{equation}
\label{eq:num_bc}
 n_0^{k+1}=0\;, \qquad n_{N}^{k+1}=1\;.
\end{equation}
These boundary values are the finite-domain analogues of the limiting conditions 
\begin{equation}
 \lim_{x\to-\infty}n(t,x)=0 
 \qquad\text{and}\qquad  
 \lim_{x\to+\infty}n(t,x)=1
\end{equation} 
imposed on~\eqref{eq:transport}. They are also compatible with the symmetric initial profiles used in the simulations. In particular, if the discrete profile satisfies
\begin{equation}
 n_{N-i}^k = 1-n_i^k\;,
\end{equation}
then the discrete total mass
\begin{equation}
 \Delta x \sum_{i=0}^{N} n_i^k
\end{equation}
is fixed by this symmetry, up to discretisation error. 

\begin{figure}[!tbp]
\centering
\includegraphics[width=0.7\linewidth]{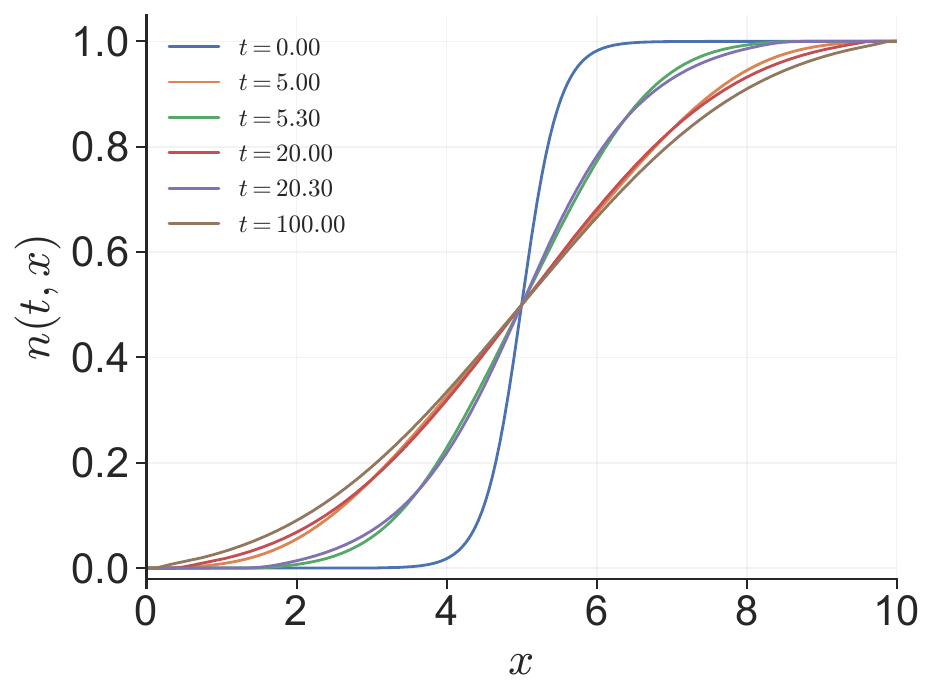}
\caption{
Time evolution of the density profile $n(t,x)$ for the transport-ODE model. The profiles are shown at the selected times $t=0$, $5.0$, $5.3$, $20.0$, $20.3$, and $100.0$. 
The pairs $t=5.0,5.3$ and $t=20.0,20.3$ illustrate the periodic behaviour of the density profile, while the profile at $t=100.0$ shows its long-time behaviour. 
Parameters: $L=10$, $\mu=0.1$, $\eps=0.05$, $a_0=7$, $E_a=1.56$, $b_3=2$, $b_1=0.6$, $b_0=-0.58$, $c=0.8$, $E^\star=1.5$, and $\eta=0.01$.
}
\label{fig:num_density_profiles}
\end{figure}

\subsection{Density profiles of the transport--ODE model}

Figure~\ref{fig:num_density_profiles} shows the time evolution of the density profile starting from the hyperbolic tangent initial condition given by
\begin{equation}
 n_0^{\tanh}(x)= \frac12\left[1+\tanh\left(2\left(x-\frac{L}{2}\right)\right)\right]\;.
\end{equation}
The profile evolves under the coupled transport--diffusion equation~\eqref{eq:num_transport} and the fast equation~\eqref{eq:num_fast} with $\eta$ small: $\eta=0.01$. The density profile exhibits an oscillatory behaviour, alternating between sharpening (approaching a shock)
and broadening (rarefaction)
during the evolution. The numerical results confirm that, after the initial transient stage, the density profile converges to a time-periodic profile whose period coincides with that of the limit cycle in the coupled $E$-$M^{\rm sim}$ dynamics. While the system is far from being infinite with this choice of parameters, we see on Figure~\ref{fig:num_density_profiles} that the qualitative behavior of the solution is as predicted by Theorem \ref{thm:main}, valid for an infinite system.


\subsection{Dependence on $\eta$}

We next examine the effect of the time-scale separation parameter $\eta$. For this purpose, we compute the trajectory of $(E(t),M^{\rm sim}(t))$ in the $E$-$M^{\rm sim}$ plane for $\eta=0.01$, $\eta=0.1$, and $\eta=1$. Figure~\ref{fig:num_phase_portraits_eta} shows the $(E(t),M^{\rm sim}(t))$ phase portrait for each value of $\eta$. For small $\eta$, the trajectory follows the critical manifold $g(E,M^{\rm sim})=0$ for a significant part of the cycle, as expected for a slow-fast system. The transitions between the two branches are then relatively sharp. As $\eta$ increases, the separation of time scales becomes weaker. The trajectory no longer jumps rapidly between the two branches, but instead forms a smoother closed orbit in the $E$-$M^{\rm sim}$ plane.
\begin{figure}[!tbp]
\centering

\begin{minipage}{0.32\linewidth}
\centering
\includegraphics[width=\linewidth]{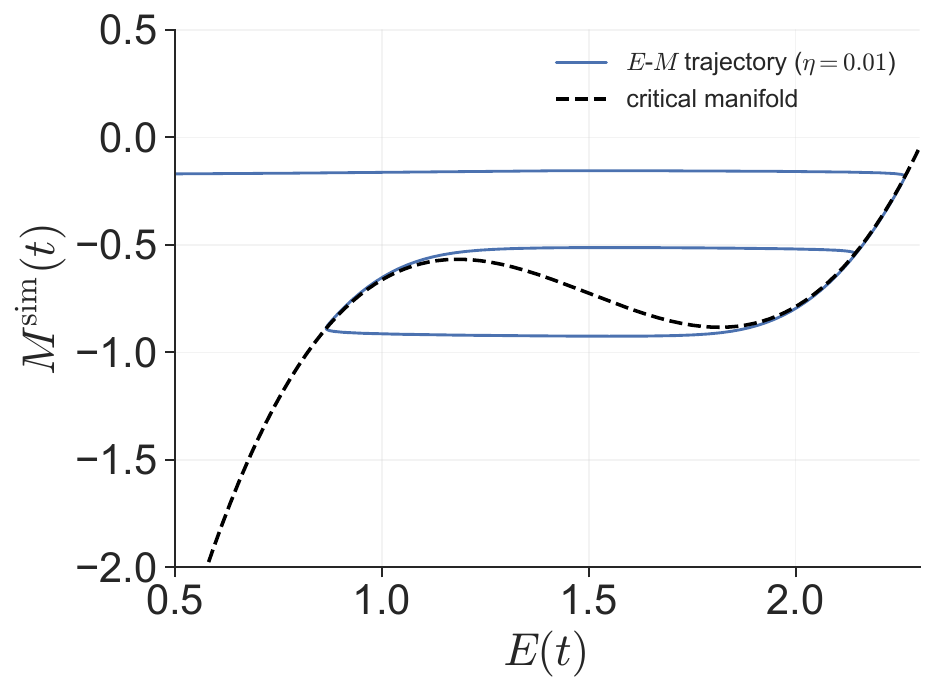}
\par
\small (a) $\eta=0.01$
\end{minipage}
\hfill
\begin{minipage}{0.32\linewidth}
\centering
\includegraphics[width=\linewidth]{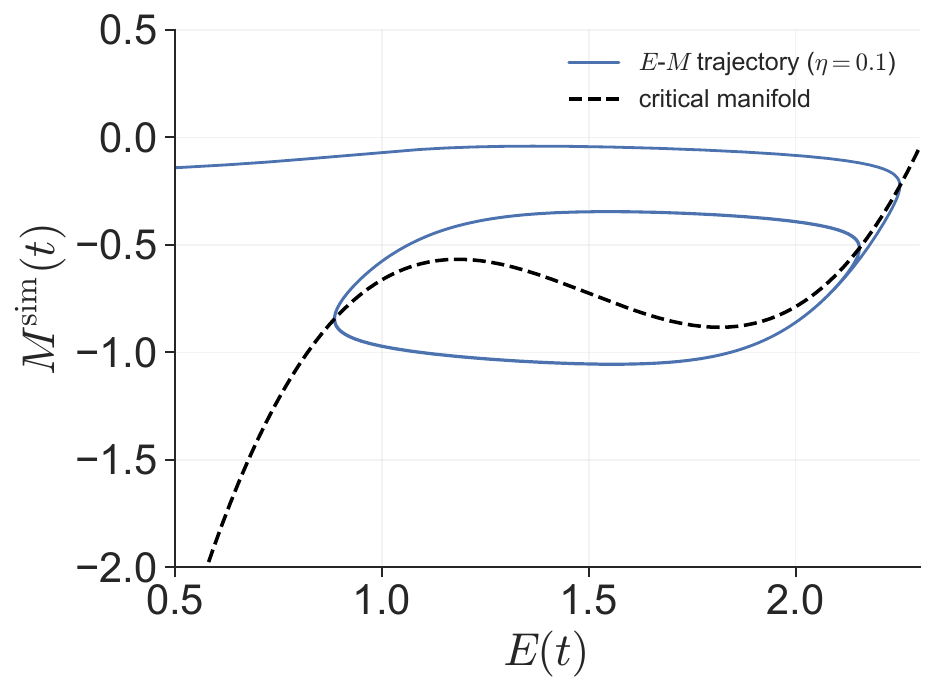}
\par
\small (b) $\eta=0.1$
\end{minipage}
\hfill
\begin{minipage}{0.32\linewidth}
\centering
\includegraphics[width=\linewidth]{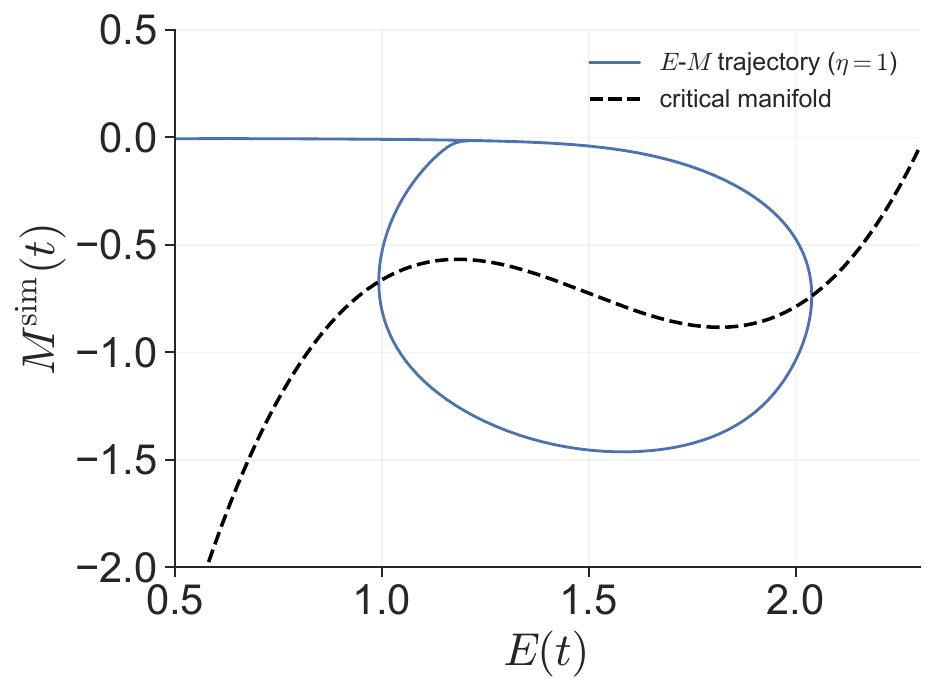}
\par
\small (c) $\eta=1.0$
\end{minipage}

\caption{
Phase portraits in the $E$-$M^{\rm sim}$ plane for different values of the time-scale separation parameter $\eta$. Panels (a), (b), and (c) correspond to $\eta=0.01$, $\eta=0.1$, and $\eta=1.0$, respectively. The curve shows the trajectory of $(E(t),M^{\rm sim}(t))$, and the dashed curve is the critical manifold \(g(E,M^{\rm sim})=0\). Parameters: $L=10$, $\mu=0.1$, $\eps=0.05$, $a_0=7$, $E_a=1.56$, $b_3=2$, $b_1=0.6$, $b_0=-0.58$, $c=0.8$, $E^\star=1.5$.
}
\label{fig:num_phase_portraits_eta}
\end{figure}
	
\begin{figure}[t]
    \centering
    \begin{tabular}{ccc}
        \includegraphics[width=0.31\textwidth]{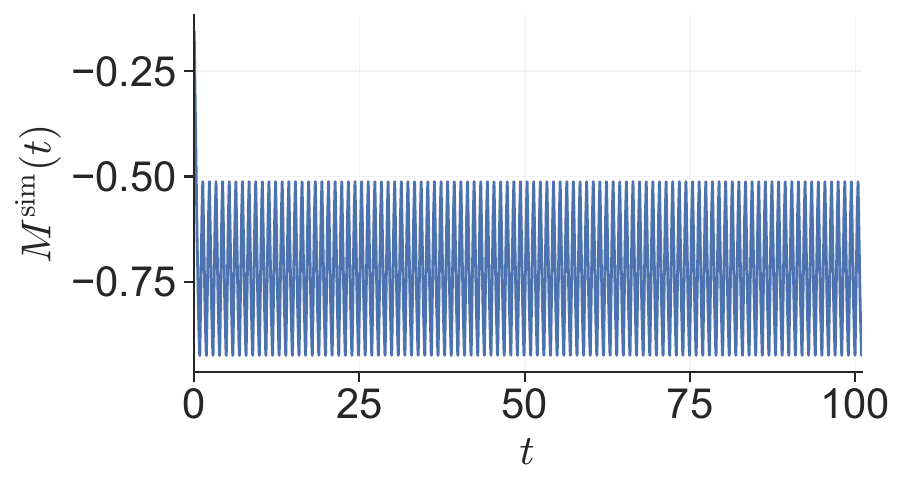}
        &
        \includegraphics[width=0.31\textwidth]{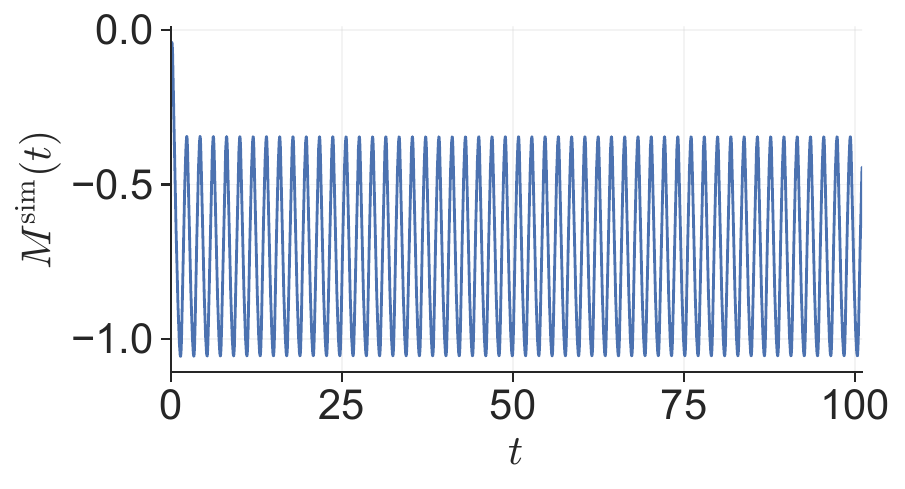}
        &
        \includegraphics[width=0.31\textwidth]{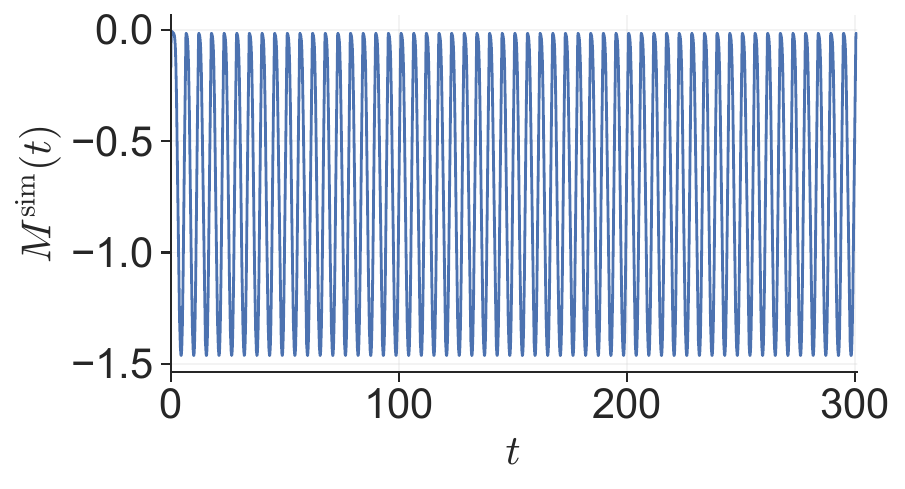}
        \\[-0.2em]
        (a) $\eta=0.01$
        &
        (b) $\eta=0.1$
        &
        (c) $\eta=1.0$
        \\[1.0em]
        \includegraphics[width=0.31\textwidth]{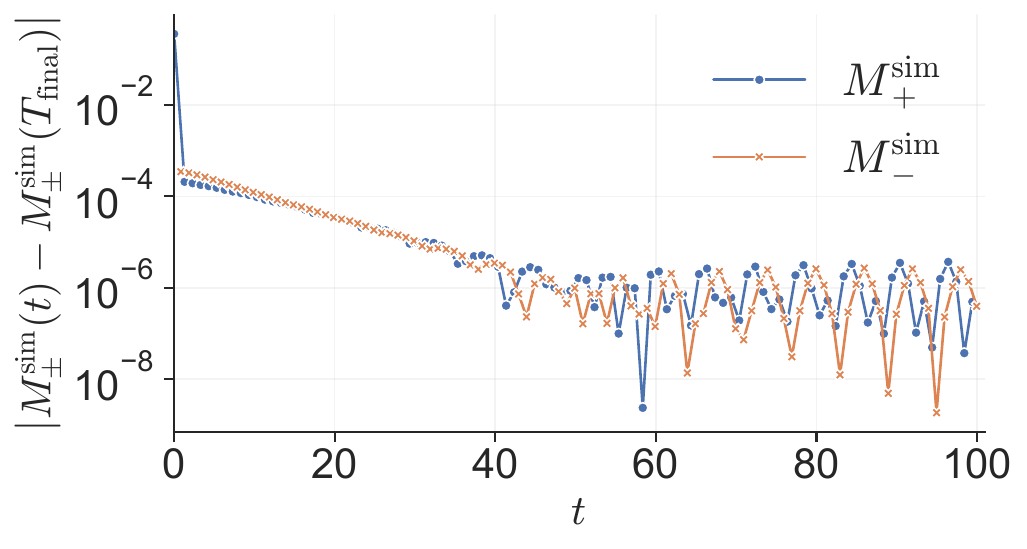}
        &
        \includegraphics[width=0.31\textwidth]{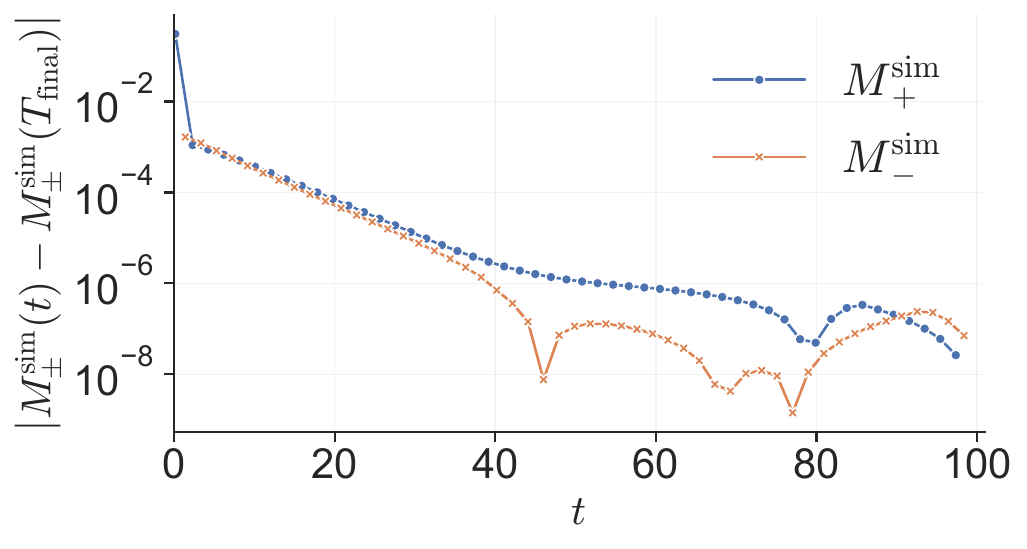}
        &
        \includegraphics[width=0.31\textwidth]{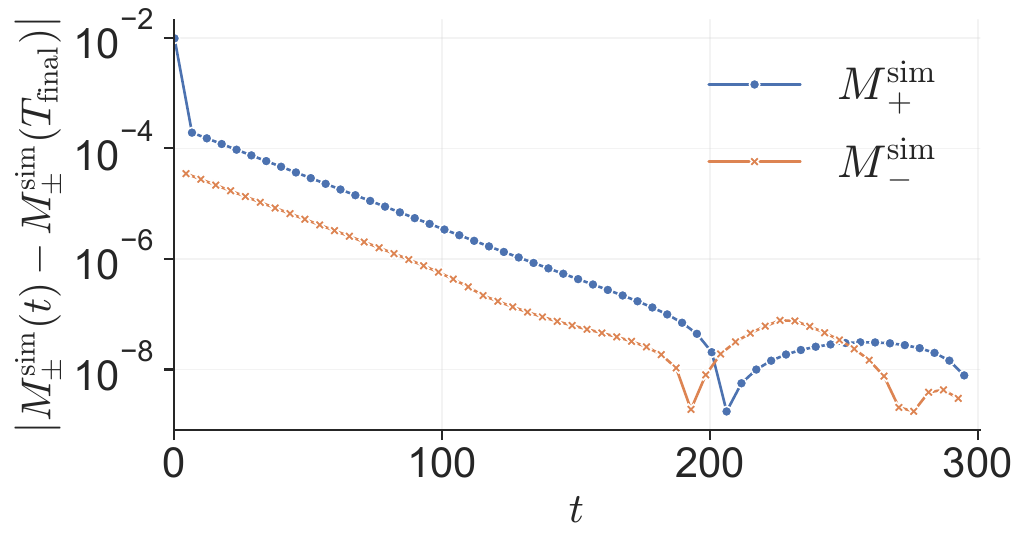}
        \\[-0.2em]
        (d) $\eta=0.01$
        &
        (e) $\eta=0.1$
        &
        (f) $\eta=1.0$
    \end{tabular}	
\caption{
Time series of $M^{\rm sim}(t)$ and convergence of the maximum and minimum values taken over each oscillation period, for different values of $\eta$. Panels (a)-(c) show the time series of $M^{\rm sim}(t)$ for $\eta=0.01$, $\eta=0.1$, and $\eta=1.0$, respectively. For each oscillation period, we take the maximum and minimum values of $M^{\rm sim}(t)$ and denote them by $M^{\rm sim}_+$ and $M^{\rm sim}_-$, respectively. Panels (d)-(f) show the differences $\left|M^{\rm sim}_{\pm}(t)-M^{\rm sim}_{\pm,\mathrm{final}}\right|$, where $M^{\rm sim}_{\pm,\mathrm{final}}$ is chosen as the corresponding final maximum or minimum near $T_{\rm final}=100$ in panels (d) and (e), and near $T_{\rm final}=300$ in panel (f). The blue curves show the differences for the maxima $M^{\rm sim}_+$, while the orange curves show the differences for the minima $M^{\rm sim}_-$. Numerical parameters: $\mu=0.1$ and $\eps=0.05$.}
    \label{fig:timeseries_extrema_eta}
\end{figure}

In addition, as shown in Figure~\ref{fig:timeseries_extrema_eta}, we examine the expected convergence toward an asymptotically periodic solution by using the time series of $M^{\rm sim}(t)$ and the maximum and minimum values taken in each oscillation period. Panels (a)-(c) show $M^{\rm sim}(t)$ for $\eta=0.01$, $\eta=0.1$, and $\eta=1.0$, respectively. Panels (d) and (e) show the differences between the maxima/minima in each oscillation period and the corresponding final maximum/minimum values near $T_{\rm final}=100$, while panel (f) uses the final values near $T_{\rm final}=300$. Since the oscillation period is longer for $\eta=1.0$, fewer cycles are completed over the same time interval, and a longer simulation is therefore required to observe convergence in the numerical simulation. For $\eta=0.01$ and $\eta=0.1$, these differences decrease to small values and then remain at these values with small fluctuations. For $\eta=1.0$, the decrease is slighter over the simulated time interval, since fewer oscillations occur in the same time window. This suggests that the approach to the asymptotic regime is related not only to the physical time, but also to the number of completed oscillation cycles. The stabilization of both the maxima and the minima of $M^{\rm sim}(t)$ provides numerical evidence consistent with convergence toward an asymptotically periodic solution.

\begin{figure}[!tbp]
\centering

\begin{minipage}{0.32\linewidth}
\centering
\includegraphics[width=\linewidth]{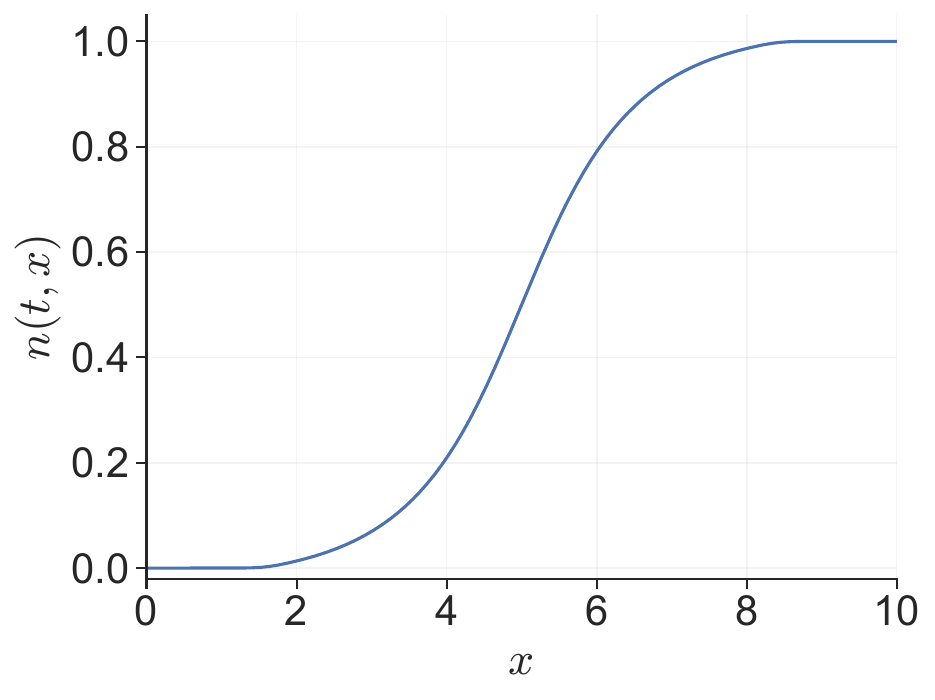}
\par
\small (a) $\eta=0.01$
\end{minipage}
\hfill
\begin{minipage}{0.32\linewidth}
\centering
\includegraphics[width=\linewidth]{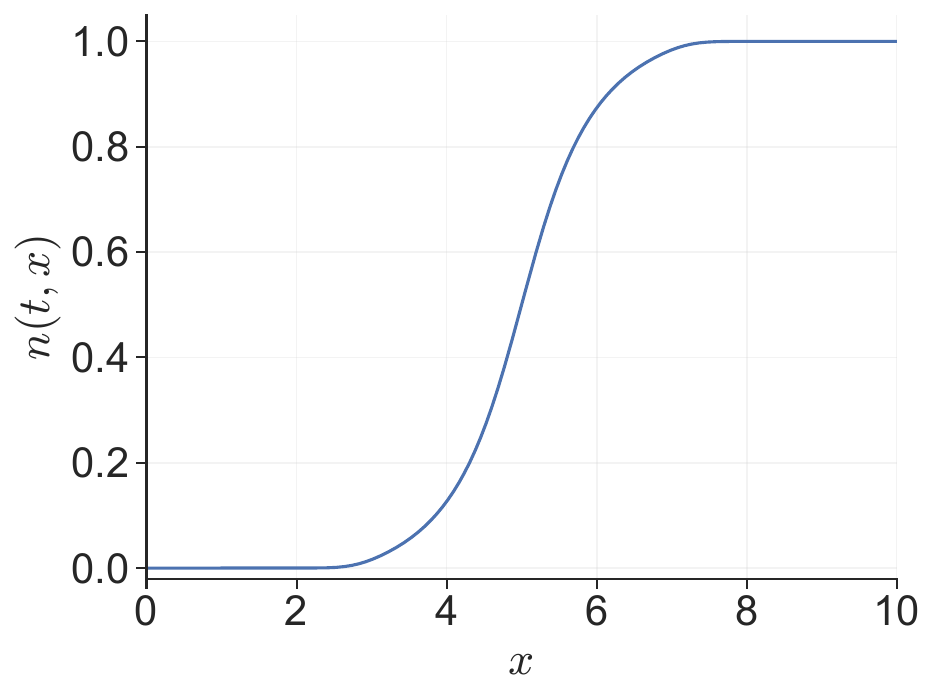}
\par
\small (b) $\eta=0.1$
\end{minipage}
\hfill
\begin{minipage}{0.32\linewidth}
\centering
\includegraphics[width=\linewidth]{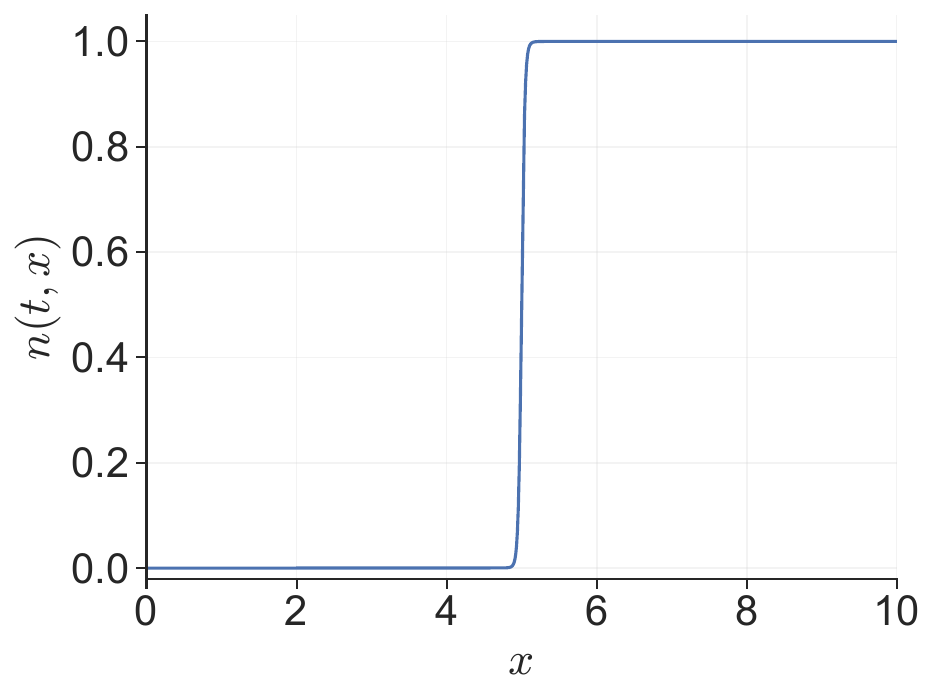}
\par
\small (c) $\eta=1.0$
\end{minipage}

\vspace{0.4cm}

\begin{minipage}{0.32\linewidth}
\centering
\includegraphics[width=\linewidth]{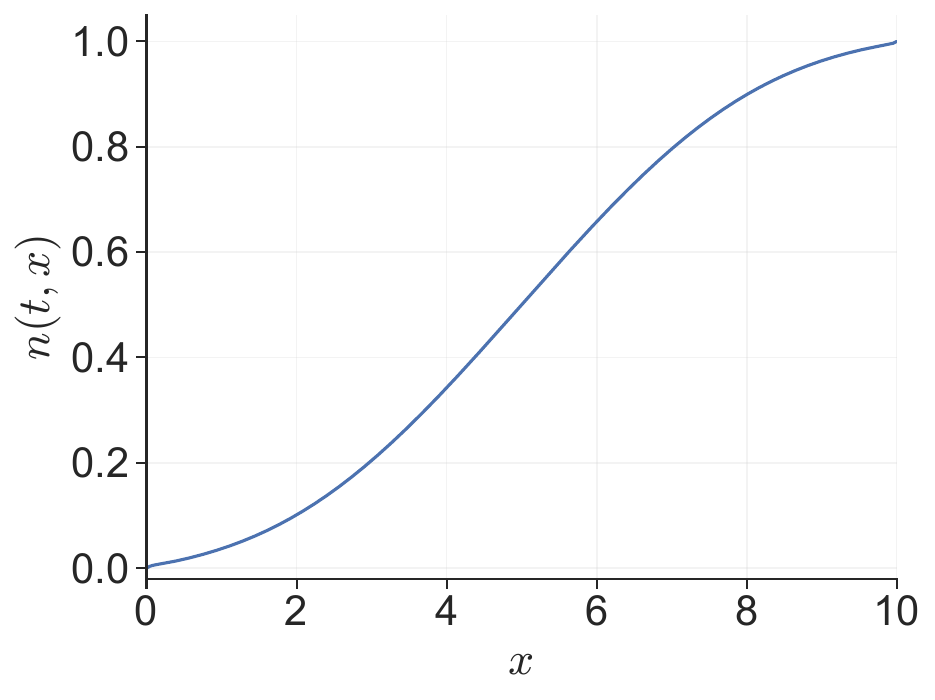}
\par
\small (d) $\eta=0.01$
\end{minipage}
\hfill
\begin{minipage}{0.32\linewidth}
\centering
\includegraphics[width=\linewidth]{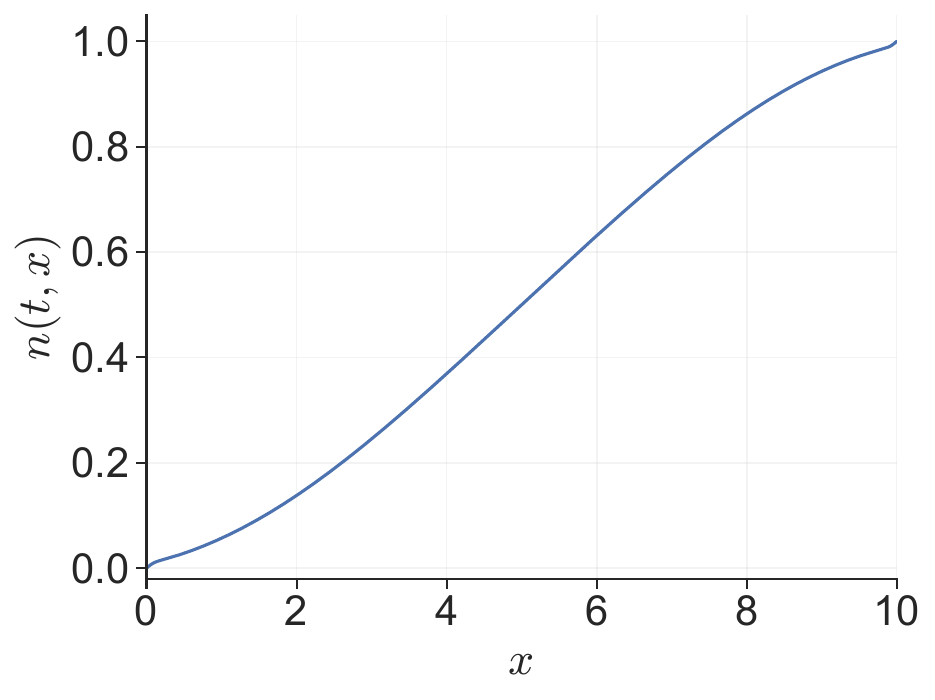}
\par
\small (e) $\eta=0.1$
\end{minipage}
\hfill
\begin{minipage}{0.32\linewidth}
\centering
\includegraphics[width=\linewidth]{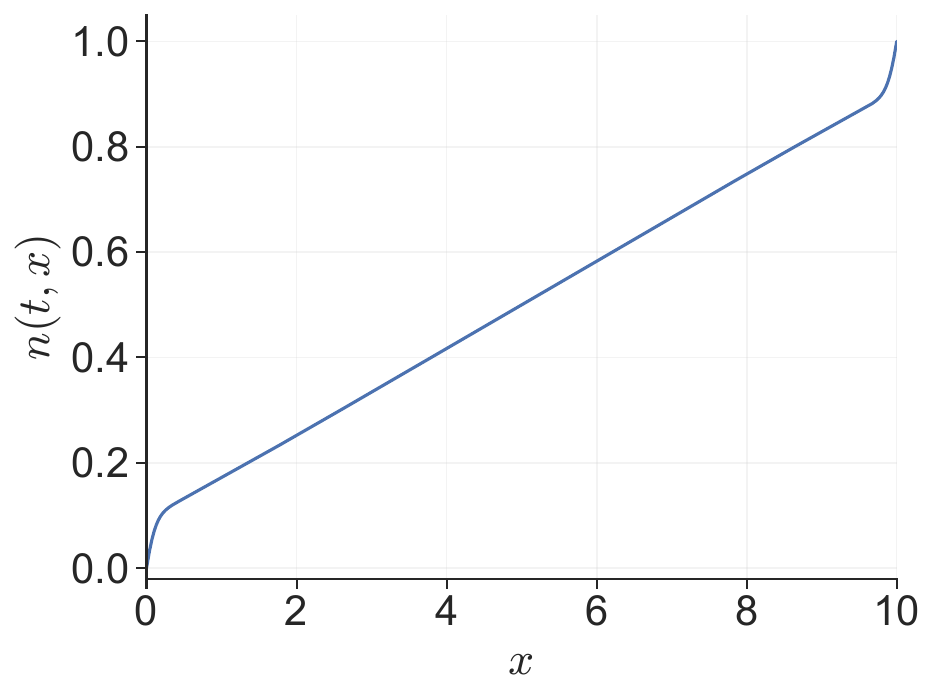}
\par
\small (f) $\eta=1.0$
\end{minipage}

\caption{
Density profiles near the local maximum and local minimum values of $M^{\rm sim}(t)$ along the cycle, for different values of the time-scale separation parameter $\eta$. Panels (a), (b), and (c) show the profiles near the local maximum side, denoted by $M^{\rm sim}_+$, for $\eta=0.01$, $\eta=0.1$, and $\eta=1.0$, respectively. Panels (d), (e), and (f) show the corresponding profiles near the local minimum side, denoted by $M^{\rm sim}_-$, for the same values of $\eta$. The selected times are $t=99.40$, $t=99.29$, and $t=95.41$ in panels (a), (b), and (c), and $t=100.95$, $t=100.31$, and $t=98.74$ in panels (d), (e), and (f), respectively. Parameters: $L=10$, $\mu=0.1$, $\eps=0.05$, $a_0=7$, $E_a=1.56$, $b_3=2$, $b_1=0.6$, $b_0=-0.58$, $c=0.8$, and $E^\star=1.5$.
}
\label{fig:num_profiles_eta_Mplus_Mminus}
\end{figure}

We also examine the density profiles near the local maximum and local minimum values of $M^{\rm sim}(t)$ along the trajectory, denoted by $M^{\rm sim}_+$ and $M^{\rm sim}_-$, respectively. Figure~\ref{fig:num_profiles_eta_Mplus_Mminus} shows the density profiles near the local maximum and local minimum values of $M^{\rm sim}(t)$ for different values of $\eta$. The upper row corresponds to the profiles near $M^{\rm sim}_+$, while the lower row corresponds to the profiles near $M^{\rm sim}_-$. Thus, the figure compares two representative states of the periodic density profile for each value of the time-scale parameter. The profiles show that $\eta$ controls how the density profile moves between these two representative states. For larger $\eta$, the transition time between the two states increases, and the profile evolves from one state to the other over a longer interval rather than jumping rapidly. This is consistent with the smoother phase portraits in Figure~\ref{fig:num_phase_portraits_eta}. During this longer transition, the density profile has more time to sharpen before switching to the other state, and it can also develop a broader shape before the next switch. For smaller $\eta$, the transition is faster, so the profile moves between the two states before such extreme shapes emerge. This behaviour can be observed in Figure~\ref{fig:num_profiles_eta_Mplus_Mminus} by comparing the cases $\eta=0.01$ and $\eta=1.0$. As $\eta$ increases, each oscillation takes longer, so the system remains in the rarefaction phase for a longer time. Consequently, the rarefaction-like profile spreads over the entire finite computational domain. As seen in panel~(f) of Figure~\ref{fig:num_profiles_eta_Mplus_Mminus}, the profile is then forced to match the fixed Dirichlet boundary values at the endpoints. The resulting endpoint bending is a finite-domain effect.

\subsection{Dependence on $\eps$}

In Figure~\ref{fig:timeseries_extrema_eta}, we examined the differences
$\left|M^{\rm sim}_{\pm}(t)-M^{\rm sim}_{\pm,\mathrm{final}}\right|$
between the maximum and minimum values of $M^{\rm sim}(t)$ in each oscillation period and their corresponding final values.  We now examine the dependence of the same differences on the diffusion parameter $\eps$, while keeping the time scale separation $\eta=0.01$ fixed.  In addition to the reference value $\eps=0.05$ used in Figure~\ref{fig:timeseries_extrema_eta}, we consider
\[
 \eps\in\{0.01,0.025,0.04,0.06,0.075,0.10\}\;.
\]
Figure~\ref{fig:extrema_convergence_epsilon} displays the resulting
differences for these six values of $\eps$.
\begin{figure}[!tbp]
\centering
\begin{tabular}{ccc}
\includegraphics[width=0.30\textwidth]{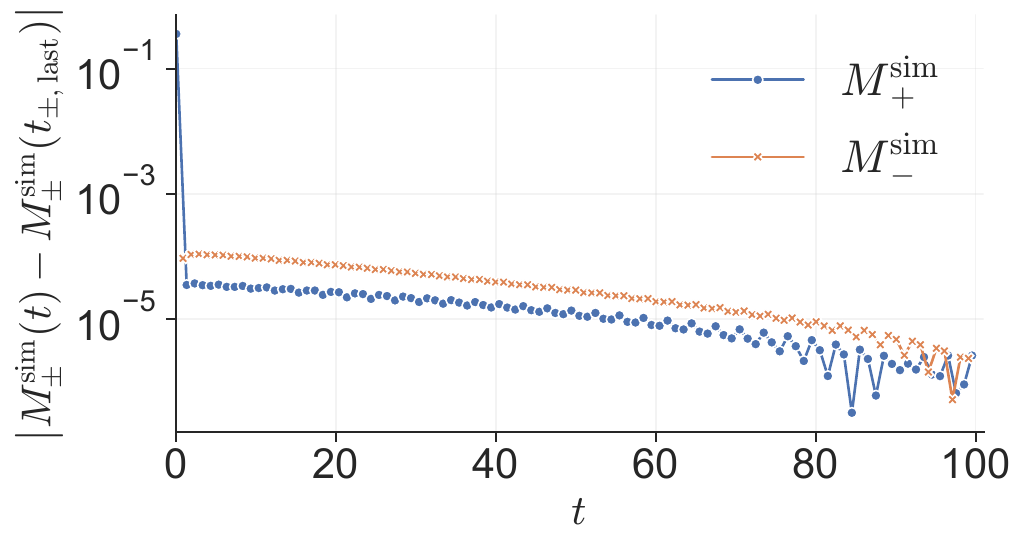}
&
\includegraphics[width=0.30\textwidth]{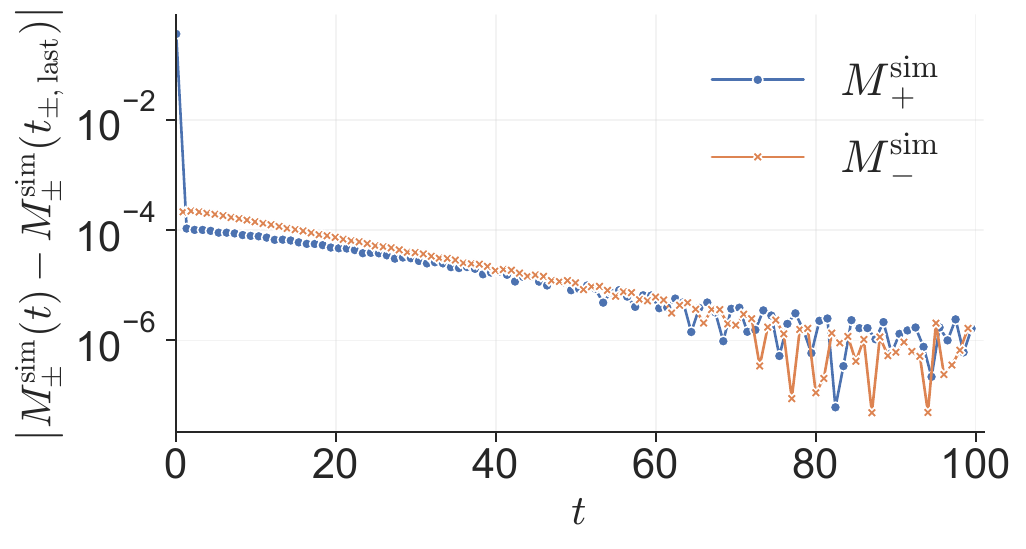}
&
\includegraphics[width=0.30\textwidth]{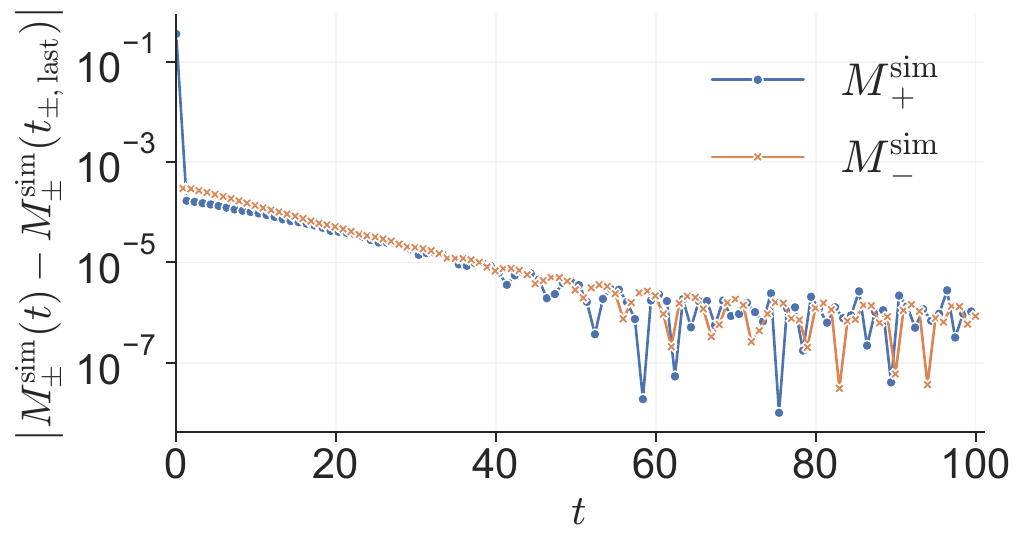}
\\[-0.2em]
(a) $\eps=0.01$
&
(b) $\eps=0.025$
&
(c) $\eps=0.04$
\\[1.0em]
\includegraphics[width=0.30\textwidth]{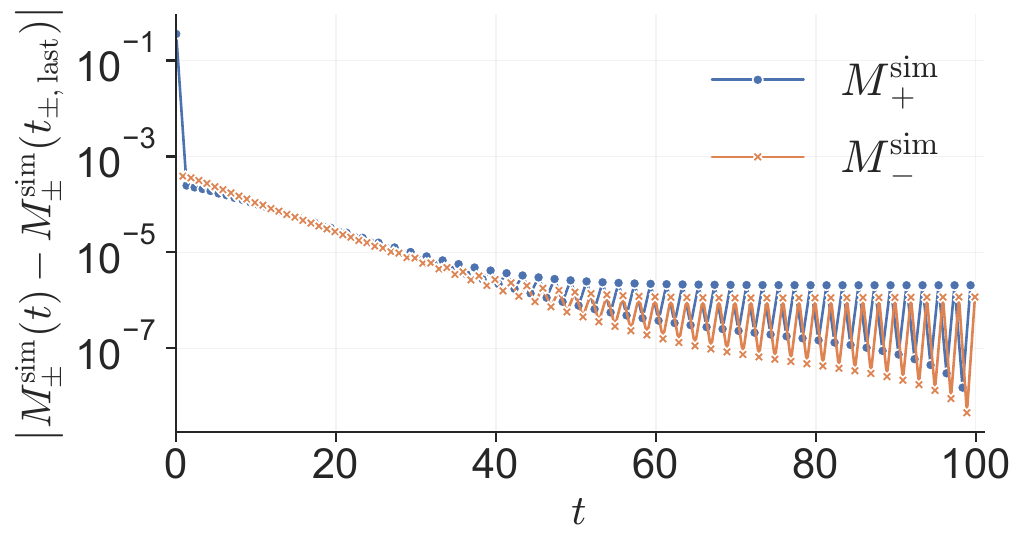}
&
\includegraphics[width=0.30\textwidth]{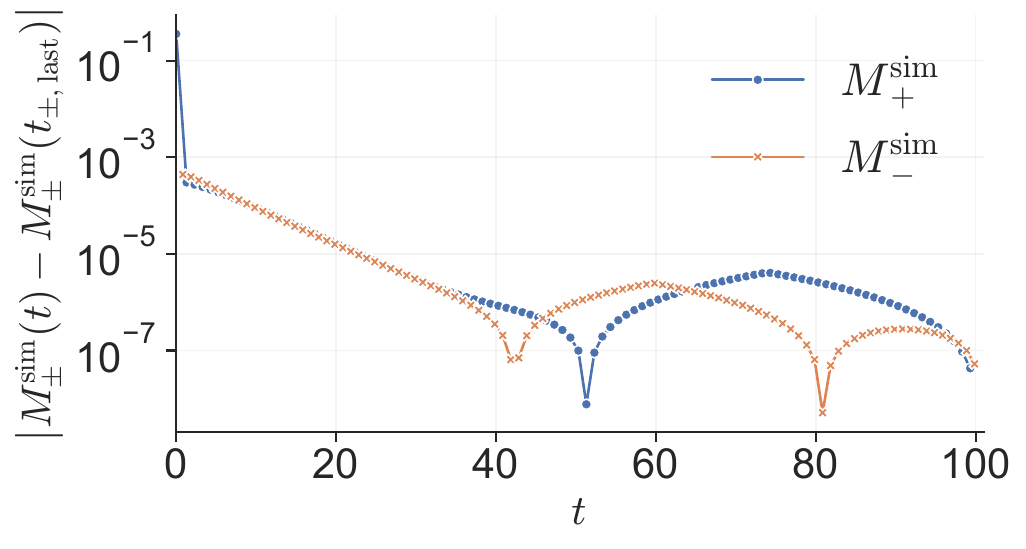}
&
\includegraphics[width=0.30\textwidth]{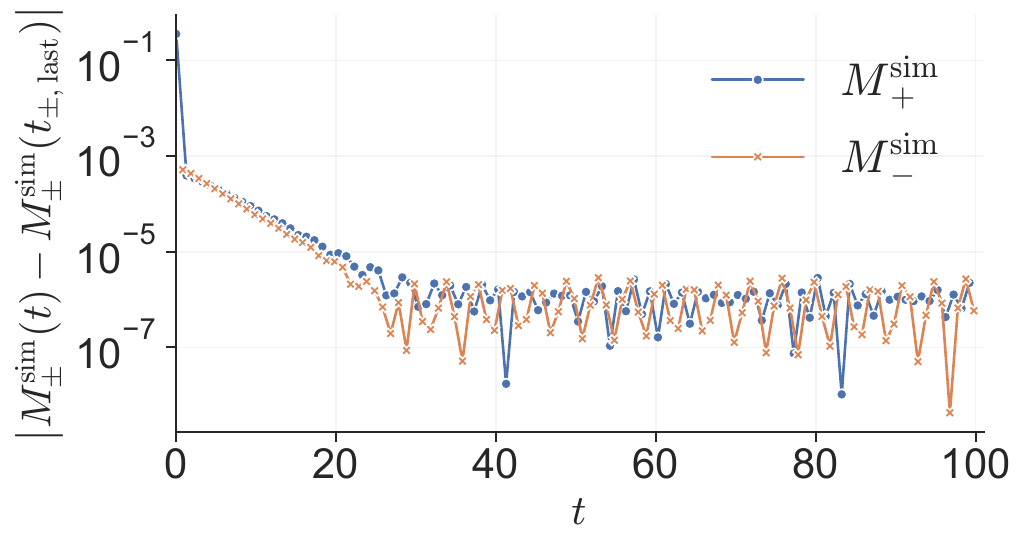}
\\[-0.2em]
(d) $\eps=0.06$
&
(e) $\eps=0.075$
&
(f) $\eps=0.10$
\end{tabular}
\caption{
Differences
$\left|M^{\rm sim}_{\pm}(t)-M^{\rm sim}_{\pm,\mathrm{final}}\right|$
between the maximum and minimum values of $M^{\rm sim}(t)$ in each oscillation period and their corresponding final values, computed as in Figure~\ref{fig:timeseries_extrema_eta}.  The blue and orange curves correspond to the maxima and minima, respectively.  Panels (a)--(f) correspond, respectively, to $(\eps,\mu)=(0.01,0.02)$, $(0.025,0.05)$, $(0.04,0.08)$, $(0.06,0.12)$, $(0.075,0.15)$, and $(0.10,0.20)$.  Parameters: $\eta=0.01$, $L=10$, $\Delta x=10^{-3}$, $\Delta t=10^{-6}$, $a_0=7$, $E_a=1.56$, $b_3=2$, $b_1=0.6$, $b_0=-0.58$, $c=0.8$, and $E^\star=1.5$, with $\mu=2\eps$.
}
\label{fig:extrema_convergence_epsilon}
\end{figure}
It shows that the differences decrease approximately exponentially before reaching the level of the numerical fluctuations.  The decay generally becomes faster as $\eps$ increases.  This confirms that increasing the diffusion coefficient accelerates convergence to the long-time oscillation. For each value of $\eps$, we fit the decay of the differences before the numerical fluctuations using
\[
 \left|M^{\rm sim}_{\pm}(t)-M^{\rm sim}_{\pm,\mathrm{final}}\right|
 \simeq C_\pm\e^{s_\pm t}\;.
\]
where $C_\pm$ is a constant and $s_\pm<0$ is the decay exponent. We
also include the previous data for $\eps=0.05$ from
Figure~\ref{fig:timeseries_extrema_eta}.  We then fit the dependence of
$s_\pm$ on $\eps$ using
\[
 s_\pm(\eps)=a_\pm\eps+b_\pm\;.
\]
\begin{figure}[h]
\centering
\begin{minipage}{0.48\linewidth}
\centering
\includegraphics[width=\linewidth]{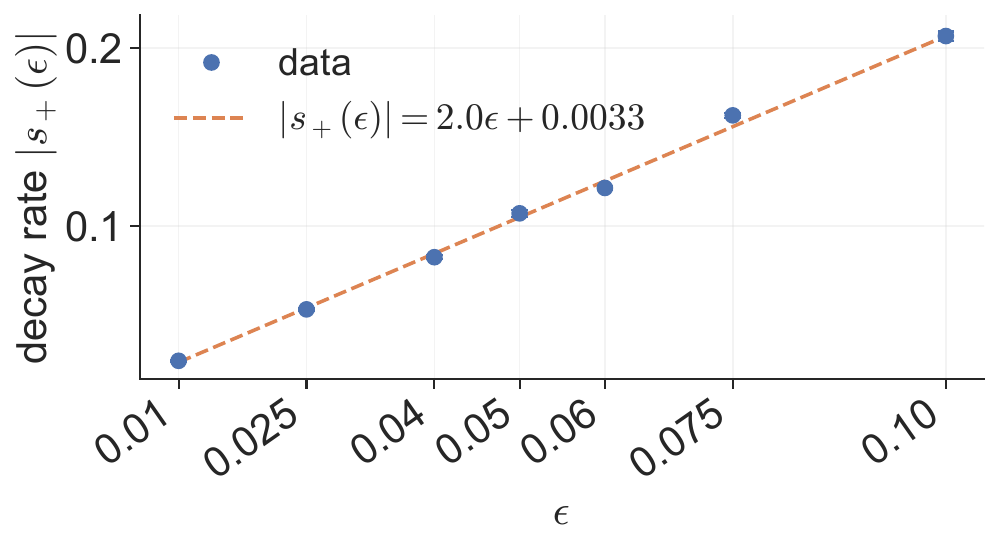}
\par
\small (a)
\end{minipage}
\hfill
\begin{minipage}{0.48\linewidth}
\centering
\includegraphics[width=\linewidth]{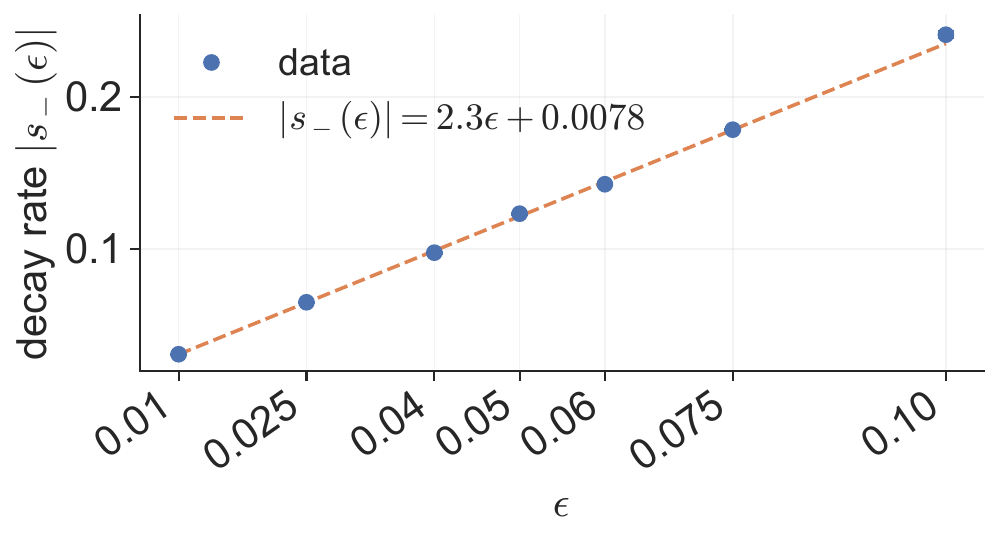}
\par
\small (b)
\end{minipage}
\caption{
Dependence of the fitted exponential-decay rates $|s_\pm|$ on the diffusion parameter $\eps$ for fixed time scale separation $\eta=0.01$.  The points correspond to $\eps\in\{0.01,0.025,0.04,0.05,0.06,0.075,0.10\}$, and the dashed lines show the fits $|s_\pm(\eps)|=a_\pm\eps+b_\pm$.  Panel (a) shows $|s_+|$ for the maxima $M^{\rm sim}_+$, with
$a_+=2.0$ and $b_+=0.0033$.  Panel (b) shows $|s_-|$ for the minima $M^{\rm sim}_-$, with $a_-=2.3$ and $b_-=0.0078$.  Parameters: $\eta=0.01$, $L=10$, $\Delta x=10^{-3}$, $\Delta t=10^{-6}$, $a_0=7$, $E_a=1.56$, $b_3=2$, $b_1=0.6$, $b_0=-0.58$, $c=0.8$, and $E^\star=1.5$, with $\mu=2\eps$.
}
\label{fig:decay_rate_epsilon}
\end{figure}
Figure~\ref{fig:decay_rate_epsilon} shows the dependence of the fitted
decay rates on the diffusion parameter $\eps$ for fixed
$\eta=0.01$.  The horizontal axis represents $\eps$.  The vertical axis
shows $|s_+|$ for the maxima $M^{\rm sim}_+$ in panel~(a), and $|s_-|$ for
the minima $M^{\rm sim}_-$ in panel~(b).  The plotted values indicate
that $|s_+|$ and $|s_-|$ vary approximately linearly with $\eps$.  The
fitted relations are
\[
 s_+(\eps)\simeq -2.0\eps-0.0033\;,
 \qquad
 s_-(\eps)\simeq -2.3\eps-0.0078\;.
\]
Since $s_\pm<0$, the decay rates are $|s_\pm|=-s_\pm$.  These decay
rates increase approximately linearly with $\eps$.
This behaviour is qualitatively consistent with the $\Order{\eps}$
decay rate in~\eqref{eq:bound_normH2_diff}.

\subsection{Comparison of profiles with different initial conditions}

\begin{figure}[!tbp]
\centering

\begin{minipage}{0.48\linewidth}
\centering
\includegraphics[width=\linewidth]{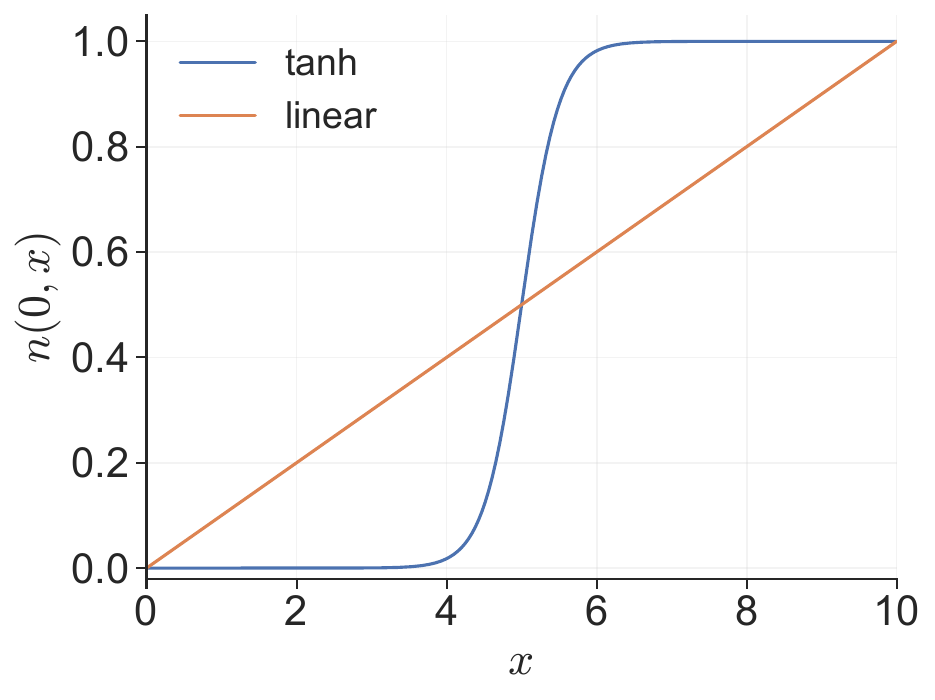}
\par
\small (a)
\end{minipage}
\hfill
\begin{minipage}{0.48\linewidth}
\centering
\includegraphics[width=\linewidth]{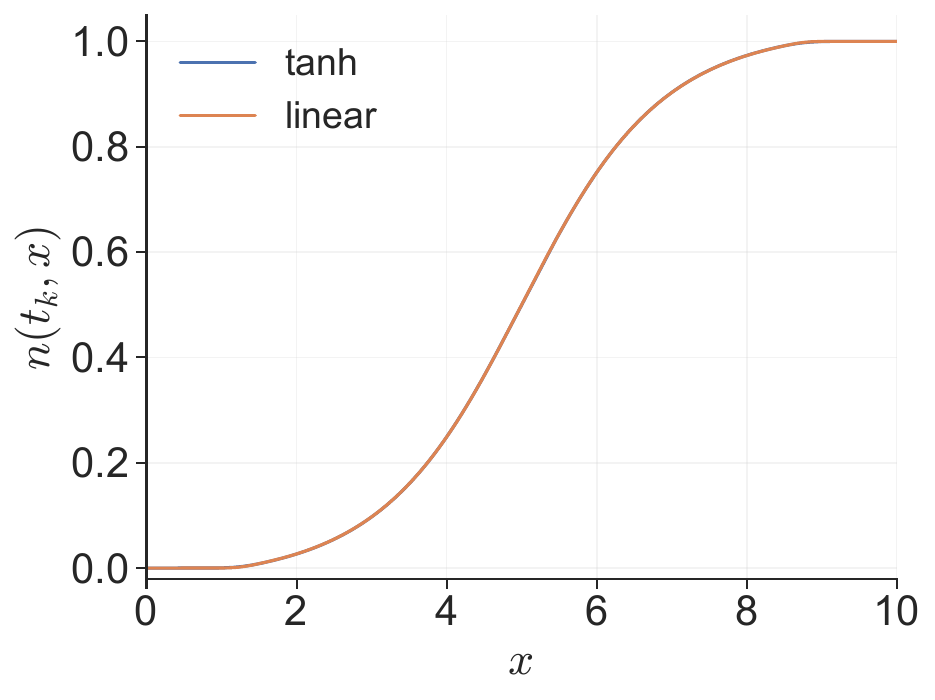}
\par
\small (b)
\end{minipage}

\caption{
Comparison of density profiles starting from different initial conditions. Panel (a) shows the initial profiles: hyperbolic tangent and linear profiles. Panel (b) shows the corresponding profiles after a sufficiently long time evolution, around $t\simeq 100$, at the time when $M^{\rm sim}(t)$ is closest to the reference value $M^{\rm sim}_{\rm ref}=-0.6$. The comparison is therefore made at the same dynamical stage, defined by the section $M^{\rm sim}=M^{\rm sim}_{\rm ref}$, after a sufficiently long time evolution. Parameters: $L=10$, $\mu=0.1$, $\eps=0.05$, $a_0=7$, $E_a=1.56$, $b_3=2$, $b_1=0.6$, $b_0=-0.58$, $c=0.8$, $E^\star=1.5$, and $\eta=0.01$.}
\label{fig:num_profiles_fixed_M}
\end{figure}

Next, we compare the time evolution of density profiles starting from different initial conditions. The purpose is to examine whether initial profiles with the same value of the conserved quantity \eqref{eq:conserved_quantityP} converge to the same periodic orbit. We use the following two symmetric initial profiles:
\begin{align}
 n_0^{\tanh}(x) &= \frac12\Bigbrak{1+\tanh\Bigpar{2\Bigpar{x-\frac{L}{2}}}}\;,\\
 n_0^{\mathrm{lin}}(x) &= \frac{x}{L}\;.
\end{align}
Being symmetric, the hyperbolic tangent and linear initial profiles have the same value of the conserved quantity (notice that for a finite domain, \eqref{eq:conserved_quantityP} is equal to the total mass, up to a finite constant). For each initial condition, we record the density profile at the time when $M^{\rm sim}(t)$ is closest to the reference value $M^{\rm sim}_{\rm ref}=-0.6$. This provides a way to compare the profiles at a comparable dynamical time along the periodic orbit.
Figure~\ref{fig:num_profiles_fixed_M} compares the long-time density profiles starting from different initial conditions, evaluated at the same reference value of $M^{\rm sim}$. Panels (a) and (b) show the evolution from the hyperbolic tangent and linear initial profiles. Panel (a) shows the initial profiles, while panel (b) shows the profiles around $t\simeq 100$, at the time when $M^{\rm sim}(t)$ is closest to $M^{\rm sim}_{\rm ref}=-0.6$. Since the total mass is conserved in the symmetric setting, up to discretisation error, the initial total mass equals the total mass at any later time, and can therefore be used to label the long-time profiles. At the reference value of $M^{\rm sim}$, profiles starting from initial conditions with the same total mass are indistinguishable. This illustrates the selection of a unique periodic orbit for symmetric initial conditions, as stated in Theorem \ref{thm:main}.


\subsection{Transient large- and small-amplitude oscillations}
\label{ssec:transient} 

We finally examine the dynamics when the parameter $E_a$ is chosen close to the lower-$E$ extremum of the critical manifold. From~\eqref{eq:num_h}, the extrema of the critical manifold are determined by differentiating $M^{\rm sim}$ with respect to $E$:
\begin{equation}
    \frac{dM^{\rm sim}}{dE}
    =
    \frac{3b_3(E-E^\star)^2-b_1}{c}\;.
\end{equation}
Setting this derivative equal to zero and using $b_3=2$, $b_1=0.6$, and $E^\star=1.5$, we obtain
\begin{equation}
    6(E-1.5)^2-0.6=0\;.
\end{equation}
The $E$-coordinates of the two extrema are therefore $E_-\simeq1.184$ and $E_+\simeq1.816$. Here, $E_-$ corresponds to the lower-$E$ local maximum, while $E_+$ corresponds to the higher-$E$ local minimum.
Figure~\ref{fig:num_time_series_E0_canard} shows the time series of $M^{\rm sim}(t)$ for $E_a=1.1925$, $1.193$, and $1.195$. These values lie just above $E_-$. All three simulations use the same initial conditions, namely the hyperbolic-tangent profile $n_0^{\tanh}$ and $E_0=0.4$. We also set $\eta=0.01$ and keep all other numerical parameters fixed.
\begin{figure}[!tbp]
\centering

\begin{minipage}{0.32\linewidth}
\centering
\includegraphics[width=\linewidth]{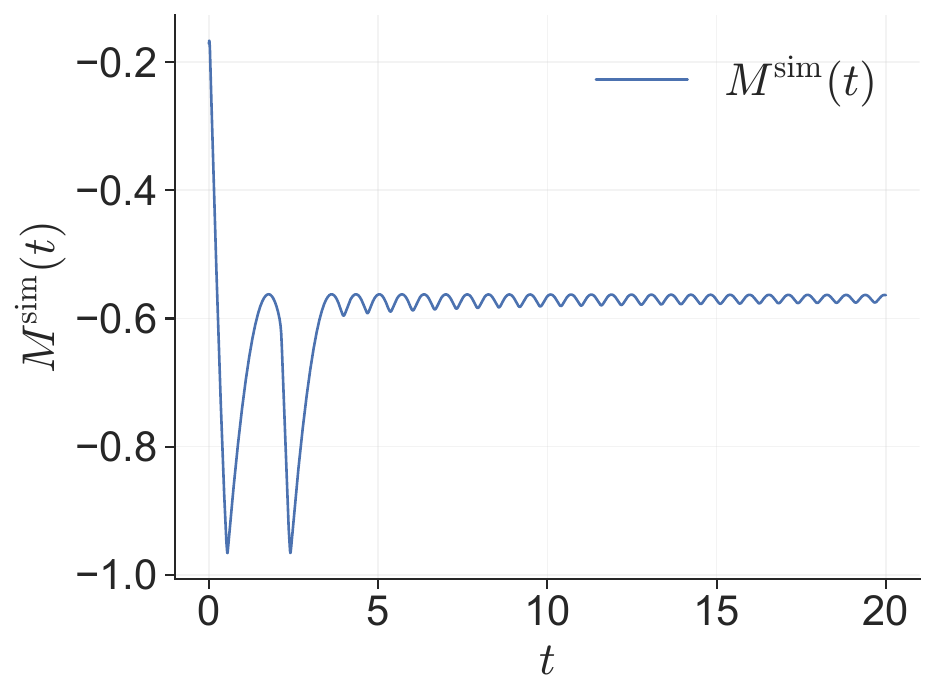}
\par
\small (a) $E_a=1.1925$
\end{minipage}
\hfill
\begin{minipage}{0.32\linewidth}
\centering
\includegraphics[width=\linewidth]{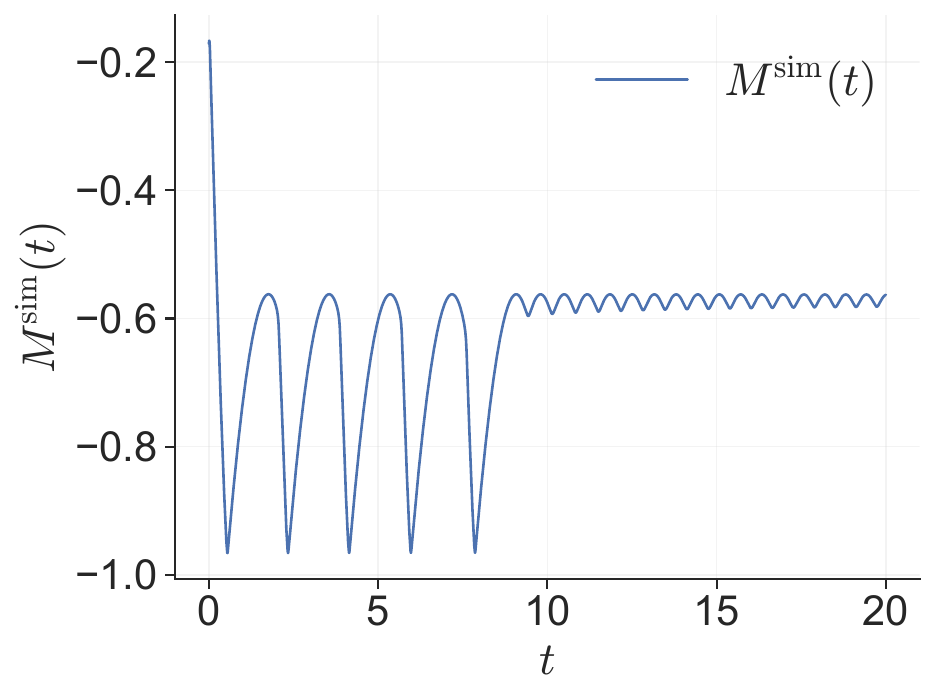}
\par
\small (b) $E_a=1.193$
\end{minipage}
\hfill
\begin{minipage}{0.32\linewidth}
\centering
\includegraphics[width=\linewidth]{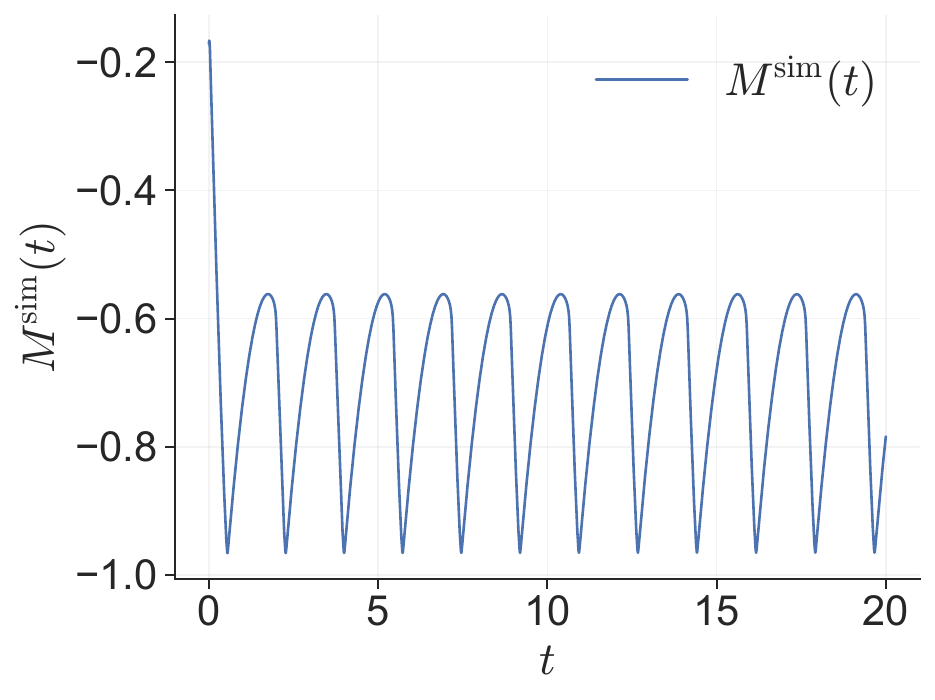}
\par
\small (c) $E_a=1.195$
\end{minipage}

\caption{
Time series of $M^{\rm sim}(t)$ for three nearby values of $E_a$. Panels (a), (b), and (c) correspond to $E_a=1.1925$, $E_a=1.193$, and $E_a=1.195$, respectively. The initial conditions are $n_0^{\tanh}$ and $E_0=0.4$ in all three panels, and all parameters other than $E_a$ are kept fixed. For the two smaller values of $E_a$, the initial large-amplitude oscillations are followed by a small-amplitude oscillatory regime. For $E_a=1.195$, the large-amplitude oscillations are maintained throughout the time interval $0\leq t\leq20$. Parameters: $L=10$, $\mu=0.1$, $\eps=0.05$, $a_0=7$, $b_3=2$, $b_1=0.6$, $b_0=-0.58$, $c=0.8$, $E^\star=1.5$, and $\eta=0.01$.
}
\label{fig:num_time_series_E0_canard}
\end{figure}
In all three panels, the large-amplitude oscillations extend approximately between $M^{\rm sim}=-1.0$ and $M^{\rm sim}=-0.6$. In panel~(a), for $E_a=1.1925$, two large-amplitude oscillations are followed by a small-amplitude oscillatory regime at approximately $t=4$. In panel~(b), for $E_a=1.193$, approximately five large-amplitude oscillations occur, and the small-amplitude regime is reached around $t=9$. In panel~(c), for $E_a=1.195$, the large-amplitude oscillations persist throughout the time interval $0\leq t\leq20$. Thus, within the simulated time interval, a transient mixed-amplitude pattern is observed for $E_a=1.1925$ and $E_a=1.193$, whereas no transition to the small-amplitude mode is observed for $E_a=1.195$. This comparison shows that the appearance of this pattern is highly sensitive to small changes in $E_a$.

This behaviour is compatible with the canard regime known for the two-dimensional FitzHugh--Nagumo  (or modified Van der Pol) system. In the corresponding slow-fast picture, a stable small-amplitude periodic orbit is created through a Hopf bifurcation when the slow nullcline passes near the fold of the critical manifold. As the parameter is varied further, the amplitude of this periodic orbit can increase rapidly over a small parameter interval, eventually producing a relaxation-oscillation orbit.


\section{Averaging}
\label{sec:averaging} 

The first main idea in the proof of Theorem~\ref{thm:main} is to use averaging. 
To put the system into a form where averaging applies, we use the same idea as in 
the proof of Proposition~\ref{prop:eps0}, which is to work with the reciprocal function 
$X(t,n): [0,\tau)\times[0,1]\to\R$ of $n(t,x)$, defined as the unique solution of 
\begin{equation}
\label{eq:def_X} 
 X(t,n(t,x)) = x\;.
\end{equation} 
This is possible as long as $x\mapsto n(t,x)$ is strictly increasing, hence 
the use of $\tau$ defined in~\eqref{eq:def_tau}. The 
system~\eqref{eq:transport}--\eqref{eq:ODE_E}--\eqref{eq:def_M}  then takes the following form.

\begin{proposition}
If  $x\mapsto n(t,x)$ is strictly increasing, the 
system~\eqref{eq:transport}--\eqref{eq:ODE_E}--\eqref{eq:def_M}  is equivalent to 
\begin{align}
 \partial_t X(t,n) &= (1-2n) a(E(t), M(t)) - \eps \partial_n \biggpar{\frac{1}{\partial_n X(t,n)}}\;, \\
 \dot{M}(t) &= \frac14 a(E(t), M(t)) - \frac{\eps}{\partial_n X(t,\frac12)}\;, \\
 \eta \dot{E}(t) &= g(E(t),M(t))\;.
 \label{eq:coupledX} 
\end{align}
\end{proposition}
\begin{proof}
Taking two $x$-derivatives of the relation $X(t,n(t,x)) = x$ yields 
\begin{equation}
 \partial_{nn}X (\partial_x n)^2 + \partial_n X \partial_{xx} n = 0\;,
\end{equation} 
which shows that 
\begin{equation}
 \partial_{xx}n = - \frac{\partial_{nn}X}{(\partial_n X)^3}\;.
\end{equation} 
This yields 
\begin{equation}
 \partial_t X = (1-2n) a(E,M) + \eps\frac{\partial_{nn}X}{(\partial_n X)^2}\;,
\end{equation} 
which is equivalent to~\eqref{eq:coupledX}.
\end{proof}

Note that the equation for $\dot M$ in~\eqref{eq:coupledX} is actually a consequence 
of the equation for $\partial_t X$. 
In order to make symmetries more apparent, we will use the additional change of variables 
\begin{equation}
 n = \frac12 (1 + m)\;,
\end{equation} 
which transforms the system~\eqref{eq:coupledX} into 
\begin{align}
 \partial_t X(t,m) &= -m a(E(t), M(t)) - \eps \partial_m \biggpar{\frac{1}{\partial_m X(t,m)}}\;, \\
 \label{eq:coupled2X} 
 \dot{M}(t) &= \frac14 a(E(t), M(t)) - \frac{\eps}{2\partial_m X(t,0)}\;, \\
 \eta \dot{E}(t) &= g(E(t),M(t))\;,
\end{align}
where now $X:[0,\tau)\times[-1,1]\to\R$. Owing to the symmetry assumption~\eqref{eq:symmetry0}, 
$X$ is an odd function of $m$. One can thus study the equation for $m\in[0,1)$, with 
boundary conditions 
\begin{equation}
\label{eq:bc_X} 
 X(t,0) = 0\;, \qquad 
 \lim_{m\to 1} X(t,m) = \infty\;.
\end{equation} 

\begin{remark}
\label{rem:MX} 
In the new variables, the definition~\eqref{eq:def_M} of $M(t)$ translates into 
\begin{equation}
\label{eq:MX} 
 M(t) = \frac12 \int_{-1}^0 X(t,m) \6m
 = -\frac12 \int_0^1 X(t,m) \6m\;.
\end{equation} 
This could be viewed as the existence of an additional conserved quantity, 
although it is merely a consequence of the definition of $M$.
\end{remark}

In what follows, we will first consider the simpler limiting case $\eta = 0$ in Section~\ref{ssec:eta0}, 
before moving to the general case in Section~\ref{ssec:eta_positive}. 


\subsection{The case $\eta = 0$}
\label{ssec:eta0} 

In the limiting case $\eta = 0$, we can consider that the fast variable 
$E$ follows the stable parts of the critical manifold $\set{g(E,M) = 0}$, 
jumping instantaneously from one stable branch to the other one at the fold 
points $(E_-,M_+)$ and $(E_+,M_-)$. This will be justified by the analysis 
presented in Section~\ref{ssec:eta_positive} for small $\eta > 0$.
This limiting behaviour suggests introducing a new time $s$ defined by 
\begin{equation}
\label{eq:def_ph} 
 M(s) = M_- + \ph(s)\;,
\end{equation} 
where 
\begin{equation}
 \ph(s) = 
 \begin{cases}
  s & 0\leqs s\leqs \Delta M\;, \\
  2\Delta M - s & \Delta M \leqs s \leqs 2\Delta M =: T\;,
 \end{cases}
\end{equation} 
with $\Delta M := M_+ - M_-$. We define the function 
\begin{equation}
 \tilde a(s) = 
 \begin{cases}
  a(E^*_-(M(s)), M(s)) & 0\leqs s\leqs \Delta M\;, \\
  a(E^*_+(M(s)), M(s)) & \Delta M \leqs s \leqs 2\Delta M\;,
 \end{cases}
\end{equation} 
where $E^*_\pm(M)$ denote the largest and smallest $E$ solving of $g(E,M) = 0$. 
Then the first two equations of the system~\eqref{eq:coupled2X} become
\begin{align}
 \partial_t X(t,m) &= -m \tilde a(s) - \eps \partial_m \biggpar{\frac{1}{\partial_m X(t,m)}}\;, \\
 \dot{M}(t) &= \frac14 \tilde a(s) - \frac{\eps}{2\partial_m X(t,0)}\;.
 \label{eq:coupled3X} 
\end{align}
Since 
\begin{equation}
\label{eq:dsdt} 
 \frac{\6s}{\6t} 
 = \frac{\6M}{\6t} \biggpar{\frac{\6M}{\6s}}^{-1} 
 = \biggpar{\frac14 \tilde a(s) - \frac{\eps}{2\partial_m X(t,0)}}\ph'(s)\;,
\end{equation} 
we can transform to the new time $s$.
This is justified by the fact that $\tilde a(s)$ is bounded away from $0$ by construction,
as long as $\partial_m X(t,0)$ stays larger than a constant times $\eps$. The result is 
\begin{equation}
 \partial_s X(s,m) 
 = -4\biggbrak{m + \eps \frac{1}{\tilde a(s)}\partial_m \biggpar{\frac{1}{\partial_m X(s,m)}}
 + \frac{2m\eps}{\tilde a(s)\partial_m X(s,0)} + \Order{\eps^2}} \ph'(s)\;,
\end{equation} 
where we have used the fact that $\ph'(s)\in\set{-1,1}$, and thus $\ph'(s) = 1/\ph'(s)$. 
We make a slight abuse of notation by using the same letter $X$ for the time-transformed 
function. 

In order to obtain a form of the equation suitable for averaging, we make the additional 
change of variables 
\begin{equation}
\label{eq:Ysm} 
 Y(s,m) = X(s,m) + 4m\ph(s)\;,
\end{equation} 
which yields 
\begin{equation}
\label{eq:dsY} 
 \partial_s Y(s,m) = \eps A(s,Y(s,m)) + \Order{\eps^2}\;,
\end{equation} 
where
\begin{equation}
 A(s,Y(s,m)) = -4\biggbrak{\partial_{m} \biggpar{\frac{1}{\partial_m Y(s,m)-4\ph(s)}}
 + \frac{2m}{\partial_m Y(s,0)-4\ph(s)}} \frac{\ph'(s)}{\tilde a(s)}\;.
\end{equation} 
We will later use the following observation, which is a direct consequence of 
Remark~\ref{rem:MX}.
\begin{proposition}[Conserved quantity]
\label{prop:conservation_Y_eta0} 
For all $s\geqs0$, one has 
\begin{equation}
 \int_{-1}^0 Y(s,m)\6m
 = -\int_0^1 Y(s,m)\6m 
 = 2M_-\;.
\end{equation} 
\end{proposition}
\begin{proof}
Combining~\eqref{eq:MX} and the definition~\eqref{eq:def_ph} of $\ph(s)$, we 
get 
\begin{equation}
 2M_- + 2\ph(s) 
 = 2M(s) 
 = \int_{-1}^0 Y(s,m)\6m - 4\ph(s)\int_{-1}^0 m\6m 
 = \int_{-1}^0 Y(s,m)\6m + 2\ph(s)\;.
\end{equation} 
An analogous computation holds for the integral from $0$ to $1$. 
\end{proof}
The main result of this subsection is the following averaging result.

\begin{proposition}[Averaging]
\label{prop:avrg_eta0} 
There exists a $T$-periodic change of variables of the form 
\begin{equation}
 Y(s,m) = \bar Y(s,m) + \Order{\eps}\;,
\end{equation} 
casting the system~\eqref{eq:dsY} into the form
\begin{equation}
\label{eq:dsYbar} 
 \partial_s \bar Y(s,m) = \eps \bar A(\bar Y(s,m)) + \Order{\eps^2}\;,
\end{equation} 
where 
\begin{equation}
\label{eq:def_Abar} 
 \bar A(\bar Y(s,m))
 = -4\partial_m \bigbrak{G(\partial_m \bar Y(s,m))} - 8mG(\partial_m \bar Y(s,0))\;,
\end{equation} 
with 
\begin{equation}
\label{eq:def_G} 
 G(z) = \frac1T \int_0^T \frac{\6s}{\abs{\tilde a(s)}(z - 4\ph(s))}\;.
\end{equation} 
\end{proposition}
\begin{proof}
The first step is standard in the theory of averaging. We decompose 
\begin{equation}
 A(s,Y) = \bar A(Y) + A^0(s,Y)\;,
\end{equation} 
where 
\begin{equation}
 \bar A(Y) = \frac1T \int_0^T A(s,Y)\6s
\end{equation} 
is the average of $A$, and $A^0(s,Y)$ has average zero. We then carry out the change of 
variables 
\begin{equation}
 Y = \bar Y + \eps w(s,\bar Y)\;,
\end{equation} 
which yields 
\begin{equation}
\label{eq:dsY1} 
 \partial_s \bar Y(s,m) + \eps\partial_s w(s,\bar Y(s,m))
 = \eps A(s, \bar Y(s,m)) + \Order{\eps^2}\;.
\end{equation} 
The function 
\begin{equation}
\label{eq:def_w0} 
 w(s,\bar Y) = \int_0^s A^0(s_1,\bar Y)\6s_1 
\end{equation} 
is periodic since $A^0$ has average zero. With this choice of $w$, 
the equation~\eqref{eq:dsY1} becomes 
\begin{equation}
 \partial_s \bar Y(s,m)
 = \eps \bar A(\bar Y(s,m)) + \Order{\eps^2}\;.
\end{equation} 
It remains to cast $\bar A$ into the expression~\eqref{eq:def_Abar}. 
Noting that $\tilde a(s)/\ph'(s) = \abs{\tilde a(s)}$ and writing $Z(\cdot,m) = \partial_m Y(\cdot,m)$, 
we first compute 
\begin{align}
 \int_0^m &\bar A(Y(\cdot,\bar m))\6\bar m \\
 ={}& -\frac4T \biggbrak{\int_0^T \frac{1}{\abs{\tilde a(s)}} 
 \int_0^m \partial_{m} \biggpar{\frac{1}{Z(\cdot,\bar m)-4\ph(s)}}
 \6\bar m \6 s
 + m^2 \int_0^T \frac{\6s}{\abs{\tilde a(s)}(Z(\cdot,0) - 4\ph(s))}} \\
 ={}& -\frac4T \biggbrak{\int_0^T \frac{\6s}{\abs{\tilde a(s)}(Z(\cdot,m)-4\ph(s))} 
 - (1-m^2) \int_0^T \frac{\6s}{\abs{\tilde a(s)}(Z(\cdot,0) - 4\ph(s))}} \\
 ={}& -4 \bigbrak{G(Z(\cdot,m)) - (1-m^2)G(Z(\cdot,0))}\;,
 \label{eq:Abar} 
\end{align}
where $G$ is defined in~\eqref{eq:def_G}. Taking the derivative with respect 
to $m$ yields~\eqref{eq:def_Abar}.  
\end{proof}

The interest of this result is that to leading order in $\eps$, the averaged 
equation~\eqref{eq:dsYbar} no longer depends on the time variable $s$. 
Therefore, if the order $\eps$ approximation admits a stationary solution that 
is hyperbolic, standard perturbation theory implies that the non-averaged 
equation~\eqref{eq:dsY} admits, for sufficiently small $\eps$, a periodic 
solution, that is also hyperbolic. We will examine such solutions in 
Section~\ref{sec:stability}. 


\subsection{The case $\eta > 0$}
\label{ssec:eta_positive} 

We now extend the above approach to the case of small positive $\eta$. 
We will proceed in several steps.

\begin{enumerate}
\item   Replace the variables $(M,E)$ by polar-like coordinates $(r,\theta)$, 
and write the equations in these new variables (Section~\ref{sssec:eta_polar}).

\item   Show that for appropriate initial conditions, $r(t)$ scales in the 
same way as the periodic solution $r^*_0(\theta)$ of the case $\eps = 0$
(Section~\ref{sssec:eta_radial}). 

\item   Find conditions under which $\theta(t)$ is strictly increasing, so that 
$\theta$ can be used as new time (Section~\ref{sssec:eta_time}).

\item   Average the equation for $X(\theta,m)$, to obtain an equation that does 
not depend on $r$ or $\theta$ to lowest order in $\eps$ (Section~\ref{sssec:eta_averaging}).

\item   Use a Poincar\'e map and the implicit function theorem to show that 
the averaged equation admits a unique stationary solution. We will carry out this 
last step in Section~\ref{ssec:unperturbed_eta}.
\end{enumerate}


\subsubsection{Polar-like coordinates}
\label{sssec:eta_polar} 

\begin{figure}
 \begin{center}
\scalebox{1.0}{
\begin{tikzpicture}[>=stealth',main node/.style={circle,minimum
size=0.01cm,inner sep=0.05cm,fill=white,draw},x=4cm,y=3cm]

%

\newcommand*{\Ez}{1.5}
\newcommand*{\Mz}{1.0}

\pgfmathsetmacro{\Em}{\Ez - 1/sqrt(3)};
\pgfmathsetmacro{\Mm}{\Mz + 2/(3*sqrt(3))};
\pgfmathsetmacro{\Ep}{\Ez + 1/sqrt(3)};
\pgfmathsetmacro{\Mp}{\Mz - 2/(3*sqrt(3))};

\path[fill=blue!25,line width=0] (-0.2,{\Mm}) rectangle ({\Em},{\Mp}); 
\path[fill=teal!25,line width=0] (-0.2,{\Mm}) rectangle ({\Em},1.9); 
\path[fill=green!25,line width=0] ({\Em},{\Mm}) rectangle ({\Ep},1.9); 
\path[fill=yellow!25,line width=0] ({\Ep},{\Mm}) rectangle (3.1,1.9); 
\path[fill=orange!25,line width=0] ({\Ep},{\Mp}) rectangle (3.1,{\Mm}); 
\path[fill=red!25,line width=0] ({\Ep},-0.2) rectangle (3.1,{\Mp}); 
\path[fill=purple!25,line width=0] ({\Em},-0.2) rectangle ({\Ep},{\Mp}); 
\path[fill=violet!25,line width=0] (-0.2,-0.2) rectangle ({\Em},{\Mp}); 

\draw[->,thick] (-0.2,0) -> (3.3,0);
\draw[->,thick] (0,-0.2) -> (0,2);

\draw[blue,very thick,-,smooth,domain=0.13:2.8,samples=500,/pgf/fpu,
/pgf/fpu/output format=fixed] plot (\x, {
\Mz + (\x - \Ez)^3 - (\x-\Ez)});

\draw[violet,very thick,-,smooth,domain=-0.2:1.9,samples=300,/pgf/fpu,
/pgf/fpu/output format=fixed] plot ({\Ez - 0.1*(\x-\Mz)^2}, \x);

\draw[semithick] (0,{\Mm}) -- (3.1,{\Mm});
\draw[semithick] (0,{\Mp}) -- (3.1,{\Mp});
\draw[semithick] ({\Em},0) -- ({\Em},1.9);
\draw[semithick] ({\Ep},0) -- ({\Ep},1.9);

\node[main node] at (\Em,\Mm) {};
\node[main node] at (\Em,0) {};
\node[main node] at (0,\Mm) {};

\node[main node] at (\Ep,\Mp) {};
\node[main node] at (\Ep,0) {};
\node[main node] at (0,\Mp) {};

\node[main node] at (\Ez,\Mz) {};

\node[] at (3.2,-0.1) {$E$};
\node[] at (-0.1,1.85) {$M$};

\node[] at ({\Em},-0.12) {$E_-$};
\node[] at ({\Ep},-0.12) {$E_+$};

\node[] at (-0.12,{\Mm}) {$M_+$};
\node[] at (-0.12,{\Mp}) {$M_-$};

\node[semithick,draw=blue,blue,circle,fill=white] at (0.2,{(\Mm+\Mp)/2}) {1};
\node[semithick,draw=teal,teal,circle,fill=white] at ({\Em-0.4},{\Mm+0.25}) {2};
\node[semithick,draw=green,green!50!black,circle,fill=white] at ({(\Em+\Ep)/2},{\Mm+0.25}) {3};
\node[semithick,draw=yellow,orange,circle,fill=white] at ({\Ep+0.3},{\Mm+0.25}) {4};
\node[semithick,draw=orange,orange!75!black,circle,fill=white] at ({\Ep+0.7},{(\Mm+\Mp)/2}) {5};
\node[semithick,draw=red,red,circle,fill=white] at ({\Ep+0.5},{\Mp-0.3}) {6};
\node[semithick,draw=purple,purple,circle,fill=white] at ({(\Ep+\Em)/2},{\Mp-0.3}) {7};
\node[semithick,draw=violet,violet,circle,fill=white] at ({\Em-0.4},{\Mp-0.3}) {8};

\end{tikzpicture}
}
\vspace{-5mm}
\end{center}
\caption[]{Decomposition of the $(E,M)$-plane into regions used to define 
the polar-like $(\theta,r)$ coordinates.}
\label{fig:rtheta} 
\end{figure}
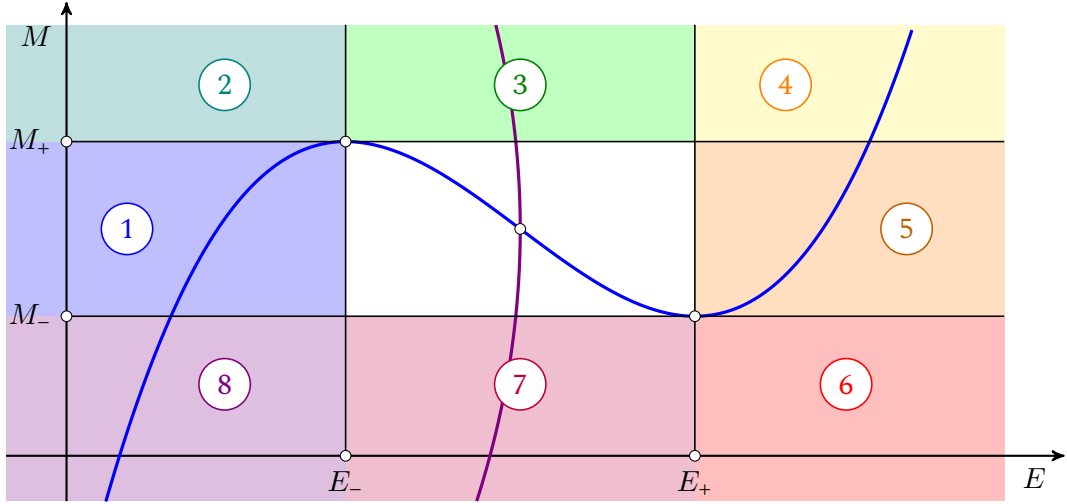

We introduce a parametrisation of the $(E,M)$-plane by two coordinates 
$(r,\theta)$, which are defined differently in $8$ different regions of the plane, 
as shown in Figure~\ref{fig:rtheta}. We do not define the new variables in the 
central rectangle $(E_-,E_+)\times(M_-,M_+)$, because we will make assumptions 
on the initial condition ensuring that the solution never enters that region. 
The regions $1$ and $5$ are treated in a similar way, as are the regions $3$ and 
$7$, and the four remaining regions. 

To this end, we introduce times 
\begin{align}
 T_0 &= 0\;, & 
 T_5 &= 2\Delta M + 3\;, \\
 T_1 &= \Delta M := M_+ - M_-\;, & 
 T_6 &= 2\Delta M + 4\;, \\
 T_2 &= \Delta M + 1\;, & 
 T_7 &= 2\Delta M + 5\;, \\
 T_3 &= \Delta M + 2\;, & 
 T_8 &= 2\Delta M + 6\;. \\
 T_4 &= \Delta M + 3\;,
\end{align}
The parametrisation is chosen in such a way that $(E,M)$ belongs to region $i$
if and only if $\theta\in[T_{i-1},T_i]$ and $r\geqs0$. Two regions overlap 
on the sets $\set{\theta = T_i}$. The precise definitions are as follows. 

\begin{itemize}
\item  \textbf{Region $1$: $E\leqs E_-$ and $M_- \leqs M \leqs M_+$:}
This region is parametrised as 
\begin{align}
 M &= M_- + \theta\;, \\
 E &= E_- - r\;.
\label{eq:polar_region1} 
\end{align}
The system~\eqref{eq:coupled2X} becomes 
\begin{align}
 \partial_t X
 &= -ma(E_--r, M_-+\theta) - \eps\partial_m\biggpar{\frac{1}{\partial_m X(t,m)}}\;,\\
 \dot \theta &= \frac14 a(E_--r, M_-+\theta) - \frac{\eps}{2\partial_m X(t,0)}\;,\\
 \eta \dot r &= -g(E_--r, M_-+\theta)\;.
 \label{eq:system_regime1} 
\end{align}

\item  \textbf{Region $2$: $E\leqs E_-$ and $M \geqs M_+$:}
In this region, we set 
\begin{align}
 M &= M_+ + r^2 \sin(\alpha(\theta))\;, \\
 E &= E_- - r\cos(\alpha(\theta))\;,
\end{align}
where 
\begin{equation}
 \alpha(\theta) = \frac\pi2 (\theta-T_1)\;.
\end{equation} 
Note that the parametrisation is continuous at $\theta = T_1$. 
The reason for the $r^2$ in the definition of $M$ is that it will simplify the 
$\eta$-dependence of the equations (at the price of creating a more complicated 
dependence on $\alpha(\theta)$). 
To compute the equations in the new variables, we observe that 
\begin{align}
 \dot M &= 2r\dot r\sin(\alpha(\theta)) + r^2 \alpha'(\theta) \dot\theta \cos(\alpha(\theta))\;, \\
 \dot E &= -\dot r\cos(\alpha(\theta)) + r \alpha'(\theta) \dot \theta \sin(\alpha(\theta))\;.
\end{align}
Using linear combinations of these equations, we get 
\begin{align}
 \dot \theta &= \frac{\cos(\alpha(\theta))}{r^2\alpha'(\theta)}\dot{M}
 + \frac{2\sin(\alpha(\theta))}{r\alpha'(\theta)}\dot E\;,\\
 \dot r &= \frac{\sin(\alpha(\theta))}{2r}\dot{M} - \cos(\alpha)\dot E\;,
 \label{eq:system_regime2} 
\end{align}
where the expressions for $\dot M$ and $\dot E$ follow from~\eqref{eq:coupled2X}. 
Note that the exact expressions for $\dot r$ and $\dot\theta$ will not matter in 
the analysis, since the system spends little time in Region $2$. We will only need 
to control the order of magnitude of these quantities in terms of the small 
parameter $\eta$. 

\item  \textbf{Region $3$: $E_- \leqs E \leqs E_+$ and $M \geqs M_+$:}
Here we take 
\begin{align}
 M &= M_+ + r^2\;, \\
 E &= E_- + (E_+ - E_-)(\theta - T_2)\;.
 \label{eq:E_theta_regime2} 
\end{align}
The system~\eqref{eq:coupled2X} becomes 
\begin{align}
 \partial_t X
 &= -ma(E(\theta), M_++r^2) - \eps\partial_m\biggpar{\frac{1}{\partial_m X(t,m)}}\;,\\
 \eta \dot \theta &= \frac{1}{E_+ - E_-}g(E(\theta), M_++r^2)\;,\\
 \dot r &= \frac{1}{2r}\biggbrak{\frac14 a(E_--r, M_-+\theta) 
 - \frac{\eps}{2\partial_m X(t,0)}}\;,
 \label{eq:system_regime3} 
\end{align}
where $E(\theta)$ is given in~\eqref{eq:E_theta_regime2}.
\end{itemize}

\begin{table}
\begin{center}
 \begin{tabular}{|c|l|l|l|l|}
\hline 
\vrule height 12pt depth 5pt width 0pt
Region & $(E,M)$ Domain & $M(\theta,r)$ & $E(\theta,r)$ & $\alpha(\theta)$ \\
\hline 
\hline 
\vrule height 12pt depth 5pt width 0pt
1 & $E\leqs E_-$ & $M_- + \theta$ & $E_- - r$ & \\
  & $M_- \leqs M \leqs M_+$  &&& \\
\hline
\vrule height 14pt depth 5pt width 0pt
2 & $E\leqs E_-$ & $M_+ + r^2 \sin(\alpha(\theta))$ & $E_- - r\cos(\alpha(\theta))$ &
$\frac\pi2 (\theta-T_1)$\\
  & $M \geqs M_+$  &&& \\
\hline
\vrule height 14pt depth 5pt width 0pt
3 &  $E_- \leqs E\leqs E_+$ & $M_+ + r^2$ & $E_- + (E_+ - E_-)(\theta - T_2)$ & \\
  & $M \geqs M_+$  &&& \\
\hline
\vrule height 14pt depth 5pt width 0pt
4 & $E \geqs E_+$ & $M_+ + r^2 \cos(\alpha(\theta))$ & $E_+ + r\sin(\alpha(\theta))$ & 
$\frac\pi2 (\theta-T_3)$\\
  & $M \geqs M_+$  &&& \\
\hline
\vrule height 12pt depth 5pt width 0pt
5 & $E \geqs E_+$ & $M_+ - (\theta - T_4)$ & $E_+ + r$ & \\
  & $M_- \leqs M \leqs M_+$  &&& \\
\hline
\vrule height 14pt depth 5pt width 0pt
6 & $E\geqs E_+$ & $M_- - r^2 \sin(\alpha(\theta))$ & $E_+ + r\cos(\alpha(\theta))$ &
$\frac\pi2 (\theta-T_5)$\\
  & $M \leqs M_-$  &&& \\
\hline
\vrule height 14pt depth 5pt width 0pt
7 &  $E_- \leqs E\leqs E_+$ & $M_- - r^2$ & $E_+ + (E_+ - E_-)(\theta - T_6)$ & \\
  & $M \leqs M_-$  &&& \\
\hline
\vrule height 14pt depth 5pt width 0pt
8 & $E \leqs E_-$ & $M_- - r^2 \cos(\alpha(\theta))$ & $E_- - r\sin(\alpha(\theta))$ & 
$\frac\pi2 (\theta-T_7)$\\
  & $M \leqs M_-$  &&& \\
\hline
\end{tabular}
\end{center}
\vspace{-3mm}
\caption[]{Parametrisation of regions by polar-like coordinates.}
\label{table:polar} 
\end{table}

The last five regions use similar changes of variables, that we list in Table~\ref{table:polar}.
One easily checks that the above transformation defines a continuous, piecewise smooth 
bijective map  
\[
\begin{alignedat}{2}
 \Phi: [0,T)\times &[0,\infty) &&\to \R^2 \setminus (E_-,E_+)\times(M_-,M_+) \\
 & (\theta,r)&&\mapsto (E,M)
\end{alignedat} 
\]
 where $T = T_8$.


\subsubsection{Behaviour of the radial coordinate}
\label{sssec:eta_radial} 

\begin{figure}
 \begin{center}
\scalebox{1.0}{
\begin{tikzpicture}[>=stealth',main node/.style={circle,minimum
size=0.01cm,inner sep=0.05cm,fill=white,draw},x=3.5cm,y=2.5cm]

\draw[->,thick] (-0.1,0) -> (2.9,0);
\draw[->,thick] (0,-0.1) -> (0,2);

%

\newcommand*{\Ez}{1.5}
\newcommand*{\Mz}{1.0}

\pgfmathsetmacro{\Em}{\Ez - 1/sqrt(3)};
\pgfmathsetmacro{\Mm}{\Mz + 2/(3*sqrt(3))};
\pgfmathsetmacro{\Ep}{\Ez + 1/sqrt(3)};
\pgfmathsetmacro{\Mp}{\Mz - 2/(3*sqrt(3))};
\pgfmathsetmacro{\Emm}{\Ez - 2/sqrt(3)};
\pgfmathsetmacro{\Epp}{\Ez + 2/sqrt(3)};


\draw[black,semithick,-,smooth,domain=0.17:2.8,samples=500,/pgf/fpu,
/pgf/fpu/output format=fixed] plot (\x, {
\Mz + (\x - \Ez)^3 - (\x-\Ez)});

\draw[blue,very thick,-,smooth,domain={\Emm}:{\Em},samples=400,/pgf/fpu,
/pgf/fpu/output format=fixed] plot (\x, {
\Mz + (\x - \Ez)^3 - (\x-\Ez)}) -- ({\Epp}, {\Mm});

\draw[blue,very thick,-,smooth,domain={\Epp}:{\Ep},samples=400,/pgf/fpu,
/pgf/fpu/output format=fixed] plot (\x, {
\Mz + (\x - \Ez)^3 - (\x-\Ez)}) -- ({\Emm}, {\Mp});

\draw[semithick, dashed] (0,{\Mm}) -- ({\Em},{\Mm}) -- ({\Em},0);
\draw[semithick, dashed] (0,{\Mp}) -- ({\Ep},{\Mp}) -- ({\Ep},0);

\draw[violet, semithick, <->] ({\Em}, {\Mm}) -- ({\Em}, {\Mp + 1.0});
\draw[violet, semithick, <->] ({\Em}, {\Mm}) -- ({\Em - 0.3}, {\Mm});

\node[violet] at ({\Em + 0.12}, {\Mm + 0.12}) {$\eta^{2/3}$};
\node[violet] at ({\Em - 0.12}, {\Mm + 0.07}) {$\eta^{1/3}$};

\node[main node] at (\Em,\Mm) {};
\node[main node] at (\Em,0) {};
\node[main node] at (0,\Mm) {};

\node[main node] at (\Ep,\Mp) {};
\node[main node] at (\Ep,0) {};
\node[main node] at (0,\Mp) {};


\draw[violet,semithick,-,smooth,domain=0.1:1.01,samples=100,/pgf/fpu,
/pgf/fpu/output format=fixed] plot ({\Em - 1.1*(\Em-\Emm)*sqrt(1-\x+0.01)}, {\Mp + \x}) 
-- ({\Epp - 0.05}, {\Mp + 1.0});

\draw[violet,semithick,-,smooth,domain=0:130,samples=20,/pgf/fpu,
/pgf/fpu/output format=fixed] plot ({\Epp - 0.05 + 0.104*sin(\x)}, {\Mp + 0.8 + 0.2*cos(\x)});

\draw[violet,semithick,-,smooth,domain=0.1:1.01,samples=100,/pgf/fpu,
/pgf/fpu/output format=fixed] plot ({\Ep + 1.1*(\Epp-\Ep)*sqrt(1-\x+0.01)}, {\Mm - \x}) 
-- ({\Emm + 0.05}, {\Mm - 1.0});

\draw[violet,semithick,-,smooth,domain=0:130,samples=20,/pgf/fpu,
/pgf/fpu/output format=fixed] plot ({\Emm + 0.05 - 0.104*sin(\x)}, {\Mm - 0.8 - 0.2*cos(\x)});

\node[] at (2.8,-0.1) {$E$};
\node[] at (-0.1,1.85) {$M$};

\node[] at ({\Em},-0.12) {$E_-$};
\node[] at ({\Ep},-0.12) {$E_+$};

\node[] at (-0.12,{\Mm}) {$M_+$};
\node[] at (-0.12,{\Mp}) {$M_-$};

\node[blue] at ({\Ez + 0.1}, {\Mm + 0.1}) {$\Gamma(0)$};
\node[violet] at ({\Ez + 0.2}, {\Mm + 0.35}) {$\Gamma(\eta)$};

\end{tikzpicture}
}
\vspace{-5mm}
\end{center}
\caption[]{Schematic representation of the singular periodic orbit $\Gamma(0)$, 
and of its deformation $\Gamma(\eta)$ for $\eta > 0$.}
\label{fig:periodic} 
\end{figure}
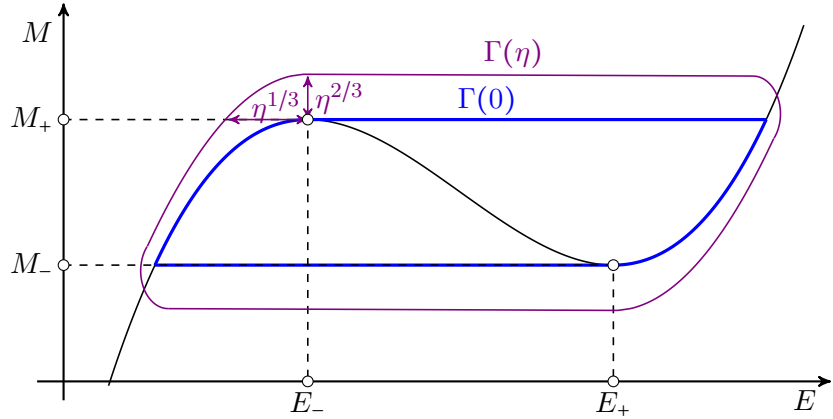

In this section, we show that the radial variable $r(t)$ behaves in a similar 
way for $\eps = 0$ and for small positive $\eps$. 
When $\eps = 0$, the equation for $(r,\theta)$ is equivalent to the two-dimensional 
slow--fast system~\eqref{eq:ODE_ME_eps0}. We know~\cite{PontRod,Haberman} that it 
admits a unique periodic orbit $\Gamma(\eta)$ (see Figure~\ref{fig:periodic}), 
which is exponentially attracting, and can be parametrised by $\theta$. 
We will denote this parametrisation $r = r^*_0(\theta;\eta)$. 
In the limit $\eta\to0$, the periodic solution $\Gamma(0)$ is called \emph{singular periodic orbit}, 
and it satisfies 
\begin{equation}
\label{eq:def_r0star} 
 r^*_0(\theta;\eta = 0) = 
 \begin{cases}
  E_- - E^*_-(M_-+\theta) & \text{for $0\leqs\theta\leqs T_1$\;,} \\
  E^*_+(M_+-(\theta-T_4)) - E_+ & \text{for $T_4\leqs\theta\leqs T_5$\;,} \\
  0 & \text{otherwise\;.}
 \end{cases}
\end{equation} 
Here the functions $E_- - E^*_-(M_-+\theta)$ and $E^*_+(M_+-(\theta-T_4)) - E_+$
correspond to the stable branches of the critical manifold in $(r,\theta)$ 
coordinates (cf.~\eqref{eq:polar_region1} and Table~\ref{table:polar}). 

Classical results (see for instance~\cite{PontRod,Haberman}) show that for 
small positive $\eta$, the periodic orbit behaves like 
\begin{equation}
 r^*_0(\theta;\eta) - r^*_0(\theta;\eta = 0) \asymp
 \begin{cases}
  \dfrac{\eta}{\max\set{T_1-\theta,\eta^{2/3}}} & \text{for $\theta\in[0, T_1]$\;,} \\[10pt]
  \eta^{1/3} & \text{for $\theta\in[T_1,T_3]\cup[T_5,T_7]$\;,} \\
  \dfrac{\eta^{1/3}}{\max\set{\sqrt{T_4-\theta},\eta^{1/3}}} & \text{for $\theta\in[T_3,T_4)$\;,} \\[10pt]
  \dfrac{\eta}{\max\set{T_5-\theta,\eta^{2/3}}} & \text{for $\theta\in[T_4, T_5]$\;,} \\[10pt]
  \dfrac{\eta^{1/3}}{\max\set{\sqrt{T-\theta},\eta^{1/3}}} & \text{for $\theta\in[T_7,T)$\;,}
 \end{cases}
 \label{eq:rstar_asymp} 
\end{equation} 
see Figure~\ref{fig:d_radial}. Here we write $a(\theta) \asymp b(\theta)$ to indicate 
that $c_- a(\theta) \leqs b(\theta) \leqs c_+ a(\theta)$ for constants $c_\pm > 0$ 
independent of $\theta$ and $\eta$. 

We now show that we have the same qualitative behaviour for the perturbed 
system~\eqref{eq:coupled2X} for small positive $\eps$. 

\begin{figure}
 \begin{center}
\scalebox{1.0}{
\begin{tikzpicture}[>=stealth',main node/.style={circle,minimum
size=0.01cm,inner sep=0.05cm,fill=white,draw},x=2.5cm,y=3cm]

\draw[->,thick] (-0.1,0) -> (5.6,0);
\draw[->,thick] (0,-0.1) -> (0,1.2);


\draw[blue,very thick,-,smooth,domain=0:1,samples=100,/pgf/fpu,
/pgf/fpu/output format=fixed] plot (\x, {sqrt(1-\x)}) -- (1,0) -- (4,0);

\draw[blue,very thick, dashed] (4,0) -- (4,1);

\draw[blue,very thick,-,smooth,domain=4:5,samples=100,/pgf/fpu,
/pgf/fpu/output format=fixed] plot (\x, {sqrt(5-\x)}) -- (5,0);


\draw[violet,semithick,-,smooth,domain=0:1,samples=100,/pgf/fpu,
/pgf/fpu/output format=fixed] plot (\x, {sqrt(1-\x+0.1)}) ;

\draw[violet,semithick,-,smooth,domain=1:2,samples=100,/pgf/fpu,
/pgf/fpu/output format=fixed] plot (\x, {sqrt(0.01/(\x-0.9))}) -- (3, {sqrt(0.01/1.1)});

\draw[violet,semithick,-,smooth,domain=3:3.9,samples=100,/pgf/fpu,
/pgf/fpu/output format=fixed] plot (\x, {sqrt(0.01/(1.1*(3.908-\x)))});

\draw[violet,semithick,-,smooth,domain=0:90,samples=100,/pgf/fpu,
/pgf/fpu/output format=fixed] plot ({4-0.1*cos(\x)}, {1 +(sqrt(1.1)-1)*sin(\x)});

\draw[violet,semithick,-,smooth,domain=4:5,samples=100,/pgf/fpu,
/pgf/fpu/output format=fixed] plot (\x, {sqrt(5-\x+0.1)}) ;

\foreach \x in {1,...,5}
\node[main node] at (\x,0) {};

\foreach \x in {1,...,5}
\node[] at ({\x},-0.12) {$T_\x$};

\node[blue] at (0.5,0.4) {$r^*_0(\theta;0)$};
\node[violet] at (1,0.7) {$r^*_0(\theta;\eta)$};

\node[] at (5.45,-0.1) {$\theta$};
\node[] at (-0.1,1.1) {$r$};

\end{tikzpicture}
}
\vspace{-5mm}
\end{center}
\caption[]{Radial behaviour of the periodic solution for $\eps = 0$, both 
in the singular limit $\eta = 0$, and for small, positive $\eta$. Only the 
$\theta$-interval $|0,T_5]$ is shown, the behaviour on $[T_5,T_8]$ is 
similar to the behaviour on $[T_1,T_4]$.}
\label{fig:d_radial} 
\end{figure}
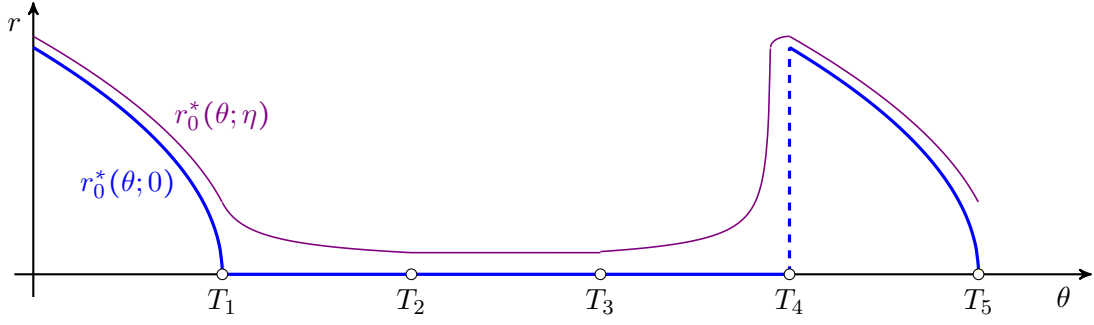

\begin{proposition}
\label{prop:radial} 
Assume there exists a constant $q_0 > 0$ such that 
\begin{equation}
\label{eq:assump_q0} 
  \partial_m X(\theta,0) \geqs q_0 \qquad \forall \theta\in[0,T]\;. 
\end{equation} 
Then there exists $\eps_0 > 0$ such that, for $0<\eps<\eps_0$, any $(M,E)$ 
components of a solution of~\eqref{eq:coupled2X} with initial condition $r(0)$ 
in a neighbourhood of order $\eta$ of $r^*_0(0;\eta)$ satisfy the 
asymptotics~\eqref{eq:rstar_asymp} for $\theta\in[0,T]$. 
\end{proposition}
\begin{proof}
We will rely on results in~\cite{Szmolyan_Wechselberger_JDE04}, based on 
geometric singular perturbation theory~\cite{Fenichel} (see 
also~\cite{Krupa-Szmolyan_SIMA01}). To cast our system 
into the form considered there, we make the change of notation 
$M\mapsto x$, $\theta\mapsto y$, $E\mapsto z$. Then the equations for $\dot{M}$ and 
$\dot{E}$ in~\eqref{eq:coupled2X} become 
\begin{align}
\label{eq:slowfast_SW} 
 \dot{x} &= g_1(x,y,z) := \frac14 a(z,x) - \eps p(y)\;, \\
 \dot{y} &= g_2(x,y,z) := 1\;, \\
\eta\dot{z} &= f(x,y,z) := g(z,x)\;,
\end{align}
where we have introduced $p(y) = 1/\partial_mX(\theta,0)$. The lower bound 
$q_0$ on $\partial_m X(\theta,0)$ implies an upper bound on $p(t)$. 
We first check that the four main assumptions 
in~\cite{Szmolyan_Wechselberger_JDE04} are satisfied.
\begin{enumerate}
\item   \textbf{Critical manifold:} The critical manifold 
$S = \setsuch{(x,y,z)}{f(x,y,z) = 0}$ admits the decomposition 
\begin{equation}
 S = S_{\textup{a}}^- \cup L^- \cup S_{\textup{r}} \cup L^+ \cup S_{\textup{a}}^+
\end{equation} 
into two attracting parts $S_{\textup{a}}^\pm$, a repelling part $S_{\textup{r}}$, 
and two fold curves $L^{\pm}$. In our case, the critical manifold has the equation 
$x = h(z)$ (where $h$ is defined in~\eqref{eq:def_h}), and the decomposition is 
given by restricting $z$ to 
the intervals $(-\infty,E_-)$, $\set{E_-}$, $(E_-,E_+)$, $\set{E_+}$, and $(E_+,\infty)$. 

\item   \textbf{Normal switching condition:} This condition reads 
\begin{equation}
 \begin{pmatrix}
  \partial_x f \\ \partial_y f
 \end{pmatrix}
 \cdot 
 \begin{pmatrix}
  g_1 \\ g_2 
 \end{pmatrix}
\biggr\vert_{L^\pm} \neq 0\;, 
\label{eq:normal_switching} 
\end{equation} 
and implies that there are no critical points on the fold curves. 
In our case, we have $\partial_y f = 0$, since $f$ does not depend on $y$. The condition 
thus reduces to $\partial_x f g_1 \neq 0$ on the fold curves, which is indeed satisfied 
if $\eps\abs{p(y)}$ is small enough. 

\item   \textbf{Transversality condition of reduced flow:}
The reduced flow on the critical manifold $S$ is obtained by eliminating the variable 
$x$ from~\eqref{eq:slowfast_SW}. Since $x = h(z)$ on $S$, we obtain 
\begin{align}
 h'(z)\dot{z} &= \frac14 a(z,h(z)) - \eps p(y)\;, \\
 \dot{y} &= 1\;.
\end{align}
This flow should be transversal to the projection of $L^-$ onto $S_{\textup{a}}^+$ 
along the fast $z$ direction, and similarly for the projection of $L^+$ onto 
$S_{\textup{a}}^-$. In our case, the fold curves are in fact straight lines, and 
these projections are just straight lines of the form $z = \text{constant}$.  
Again, our assumptions on $a$ in Assumption~\ref{assump:ag} guarantee that 
transversality holds if $\eps\abs{p(y)}$ is small enough. 

\item   \textbf{There exists a singular periodic orbit $\Gamma$:} This orbit is 
by definition the one with equation $r = r^*_0(\theta; \eta = 0)$.
\end{enumerate}
We can thus apply results in~\cite{Szmolyan_Wechselberger_JDE04} to our system. 
These results show in particular that there exists a slow manifold 
$S_{\textup{a},\eta}^-$, which is $\eta$-close to $S_{\textup{a}}^-$
and exponentially attracting as long as one stays bounded away from the 
fold curve $L^-$. Theorem~1 in~\cite{Szmolyan_Wechselberger_JDE04} shows that 
this slow manifold can be extended beyond the fold line, which it passes at 
distance of order $\eta^{1/3}$ in the $z$-direction, and of order $\eta^{2/3}$ 
in the $x$-direction. 

The slow manifold then extends along the fast $z$ direction to a vicinity of 
the projection of $L^-$ onto $S_{\textup{a}}^+$, where it approaches the 
slow manifold $S_{\textup{a},\eta}^+$, which is $\eta$-close to $S_{\textup{a}}^+$. 
Theorem~2 in~\cite{Szmolyan_Wechselberger_JDE04} describes a Poincar\'e map on 
a section $\Sigma_-$ transversal to $S_{\textup{a}}^-$. This Poincar\'e map is 
an exponentially small perturbation of the map corresponding to $\eta = 0$, 
which leaves the intersection of $\Sigma_-$ and $S_{\textup{a}}^-$ invariant. 

These results show that orbits starting close enough to $S_{\textup{a},\eta}^-$ 
on $\Sigma_-$ will return to $\Sigma_-$ at a nearby point. Since $S_{\textup{a},\eta}^-$ 
is exponentially attracting as long as the fold curve $L^-$ is not approached, 
orbits will remain exponentially close, first to $S_{\textup{a},\eta}^-$, and then to 
$S_{\textup{a},\eta}^+$. In particular, the behaviour~\eqref{eq:rstar_asymp} 
follows from the behaviour of the slow manifolds. 
\end{proof}

\begin{remark}
Theorem~4 in~\cite{Szmolyan_Wechselberger_JDE04} shows that under an additional 
hyperbolicity condition on the singular periodic orbit $\Gamma$, the system 
does have periodic orbits, corresponding to relaxation oscillations. 
We will adapt the argument to our situation later on. 
\end{remark}

\begin{remark}
If $\eps$ is too large, or if $\partial_m X(\theta,0)$ becomes too small, 
that is, when approaching a shock, the normal switching 
condition~\eqref{eq:normal_switching} is no longer satisfied. This allows 
for richer behaviour, such as folded nodes, which are a known source 
of mixed-mode oscillations~\cite{Szmolyan_Wechselberger_JDE01,MMO_review}. 
See also~\cite{Longo_Queirolo_Kuehn_24} for a case study of the 
FitzHugh--Nagumo equation perturbed by a time-dependent term. 
\end{remark}


\subsubsection{Time change}
\label{sssec:eta_time} 

In this section we determine conditions under which $\theta(t)$ is a strictly 
increasing function of $t$, and examine the structure of the equations when 
$\theta$ is taken as new time variable. To lighten notations, we introduce 
the sets
\begin{align}
 I_1 &= [T_0,T_1)\cup[T_4,T_5)\;, \\
 I_2 &= [T_1,T_2)\cup[T_3,T_4)\cup[T_5,T_6)\cup[T_7,T_8)\;, \\
 I_3 &= [T_2,T_3)\cup[T_6,T_7)\;,
\end{align}
which correspond, respectively, to the \lq\lq slow parts\rq\rq\ (regions $1$ 
and $5$), the \lq\lq corners\rq\rq\ (regions $2$, $4$, $6$ and $8$), and 
to the \lq\lq fast parts\rq\rq\ (regions $3$ and $7$). 
We will also set
\begin{equation}
 \tilde a = (a\circ\Phi)\;, \qquad 
 \tilde g = (g\circ\Phi)\;.
\end{equation} 
In particular, we have 
\begin{equation}
\label{eq:deg_atilde} 
\tilde a(\theta,r) = 
\begin{cases}
a(E_--r,M_-+\theta) &\text{for $\theta \in [0,T_1)\;,$} \\ 
a(E_++r,M_+-(\theta-T_4)) &\text{for $\theta \in [T_4,T_5)\;,$}
\end{cases}
\end{equation} 
and 
\begin{equation}
\tilde g(\theta,r) = 
\begin{cases}
g(E_--r,M_-+\theta) &\text{for $\theta \in [0,T_1)\;,$} \\ 
g(E_++r,M_+-(\theta-T_4)) &\text{for $\theta \in [T_4,T_5)\;.$}
\end{cases}
\end{equation} 

\begin{proposition}
\label{prop:X_theta} 
Assume there exist constants $q_0, R_+ > 0$ and $r_+ > r_- > 0$ such that
\begin{align}
\label{eq:lower_bound_dmX} 
 \partial_m X(t,0) &\geqs q_0 &\forall &t \text{ such that } \theta(t)\in[0,T]\;, \\
 \label{eq:bounds_rpm} 
 r_- \eta^{1/3} \leqs r(t) &\leqs r_+ \eta^{1/3}
  &\forall &t \text{ such that } \theta(t)\in I_2 \cup I_3\;, \\
r(t) &\leqs R_+ &\forall &t \text{ such that } \theta(t)\in[0,T]\;.
\label{eq:bounds_Rplus} 
\end{align} 
Then $t\mapsto\theta(t)$ is strictly increasing for $\theta(t)\in[0,T]$.
Furthermore, the system~\eqref{eq:coupled2X} is equivalent, in the new time, 
to
\begin{align}
\label{eq:dthetaX} 
 \partial_\theta X &= m A_0(\theta, r; \eta) + \eps A_1(X,\theta,r;\eta) + \Order{\eps^2}\;, \\
 \eta \frac{\6r}{\6\theta} &= B_0(\theta,r;\eta) 
 + \eps B_1(\theta,r,\partial_mX(\theta,0);\eta,\eps)\;,
\label{eq:dthetar} 
\end{align}
where the right-hand side of the first equation satisfies 
\begin{equation}
\label{eq:F0} 
 A_0(\theta, r; \eta) = 
\begin{cases}
-4 & \text{for $\theta \in [0,T_1)\;,$} \\
 4 & \text{for $\theta \in [T_4,T_5)\;,$} \\
 \Order{\eta^{2/3}} 
 & \text{for $\theta \in I_2\;,$}\\
 \Order{\eta^{1/3}} 
 & \text{for $\theta \in I_3\;,$}
\end{cases}
\end{equation} 
and 
\begin{equation}
\label{eq:A1} 
 A_1(X, \theta, r; \eta) = 
\begin{cases}
-\dfrac{4}{\tilde a(\theta,r)} 
\biggbrak{\partial_m \biggpar{\dfrac{1}{\partial_m X(\theta,m)}} 
+ \dfrac{2m}{\partial_m X(\theta,0)}} 
& \text{for $\theta \in [0,T_1)\;,$} \\[10pt]
\dfrac{4}{\tilde a(\theta,r)} 
\biggbrak{\partial_m \biggpar{\dfrac{1}{\partial_m X(\theta,m)}} 
+ \dfrac{2m}{\partial_m X(\theta,0)}}
& \text{for $\theta \in [T_4,T_5)\;,$} \\
 \Order{\eta^{2/3}}  
 & \text{for $\theta \in I_2\;,$}\\
 \Order{\eta^{1/3}}
 & \text{for $\theta \in I_3\;.$}
\end{cases}
\end{equation} 
The right-hand side of the second equation satisfies 
\begin{equation}
\label{eq:B0} 
 B_0(\theta,r;\eta)  = 
\begin{cases}
-4\dfrac{\tilde g(\theta,r)}{\tilde a(\theta,r)} 
& \text{for $\theta \in I_1\;,$} \\
 \Order{\eta^{4/3}} 
 & \text{for $\theta \in I_2\;,$}\\
 \Order{\eta^{1/3}} 
 & \text{for $\theta \in I_3\;,$}
\end{cases}
\end{equation} 
and 
\begin{equation}
\label{eq:B1} 
 B_1(\theta,r,\partial_mX(\theta,0);\eta,\eps) = 
\begin{cases}
-\dfrac{8\tilde g(\theta,r)}{\tilde a(\theta,r)^2\partial_m X(\theta,0)} 
= \Order{\eta^{2/3}}
& \text{for $\theta \in I_1\;,$} \\
 \Order{\eta^{4/3}}  
 & \text{for $\theta \in I_2\;,$}\\
 \Order{\eta} 
 & \text{for $\theta \in I_3\;.$}
\end{cases}
\end{equation} 
Finally, 
\begin{equation}
\label{eq:dB0} 
 \frac{\partial B_0}{\partial r}(\theta,r;\eta)  \asymp 
\begin{cases}
-\max\set{\sqrt{T_r-\theta},\eta^{1/3}}
& \text{for $\theta \in I_1\;,$} \\
 \Order{\eta} 
 & \text{for $\theta \in I_2\;,$}\\
 \Order{\eta^{2/3}} 
 & \text{for $\theta \in I_3\;,$}
\end{cases}
\end{equation} 
where $T_r$ denotes the right point of the current interval, 
meaning that $T_r = T_1$ if $\theta\in[0,T_1]$ and 
$T_r = T_5$ if $\theta\in[T_4,T_5]$. 
\end{proposition}
\begin{proof}
It is sufficient to discuss regions $1$, $2$ and $3$, since the other regions 
are analysed in a similar way. 
\begin{itemize}
\item   \textbf{Region $1$:}
We first note that~\eqref{eq:system_regime1} implies that $\dot\theta > 0$ 
provided 
\begin{equation}
 \partial_m X(t,0) > 2\eps \sup_{r\in[0,R_+],\theta\in[0,T_1]} \frac{1}{\tilde a(\theta,r)}\;.
\end{equation} 
This condition is actually weaker than~\eqref{eq:lower_bound_dmX}, but its stronger 
form allows us to make Taylor expansions in $\eps$. The expressions for 
$A_0$, $A_1$, $B_0$ and $B_1$ then follow immediately by expanding 
\begin{equation}
\label{eq:t_theta} 
 \partial_\theta X = \frac{\partial_t X}{\dot{\theta}}
 \text{\qquad and \qquad}
 \eta\frac{\6r}{\6\theta} = \eta\frac{\dot{r}}{\dot{\theta}}\;.
\end{equation} 
Furthermore, we have 
\begin{equation}
 \partial_r B_0(\theta, r)
 = -4\frac{\partial}{\partial r}
 \frac{\tilde g(\theta,r)}{\tilde a(\theta,r)}
 = -4 \frac{\partial_r \tilde g(\theta,r)}
 {\tilde a(\theta,r)}
 + \Order{\eta^{2/3}}\;,
\end{equation} 
since $\tilde a$ has order $1$, while $\tilde g$ has order $\eta^{2/3}$ (at most). 
Assumption~\ref{assump:ag} implies that $g(E,M)$ scales like $M - M_+ + (E - E_-)^2$ 
near $(E_-,M_+)$, which yields 
\begin{equation}
 \tilde g(\theta,r) \asymp -T_1 + \theta + r^2\;, \qquad 
 \partial_r \tilde g(\theta,r) \asymp r\;.
\end{equation} 
By~\eqref{eq:def_r0star}, \eqref{eq:rstar_asymp} and Proposition~\ref{prop:radial}, 
$r$ behaves like the maximum of $\sqrt{T_1-\theta}$ and $\eta^{1/3}$. 

Assumption~\ref{assump:ag} guarantees that $\partial_r B_0$ is negative for sufficiently 
small $\eta$. 
 
\item   \textbf{Region $2$:}
It follows from~\eqref{eq:system_regime2} that 
\begin{equation}
 \dot\theta \asymp \frac{1}{\alpha'(\theta)} 
 \biggpar{\frac{\dot{M}}{r^2} \vee \frac{\dot E}{r}}\;.
\end{equation} 
The expressions for $\dot M$ and $\dot E$ 
in~\eqref{eq:coupled2X}, assumption~\eqref{eq:bounds_rpm}
and the fact that $\alpha'(\theta)$ has order $1$ imply
that $\dot M \asymp 1$ and $\dot E \asymp \eta^{-1/3}$, so that 
\begin{equation}
 \dot\theta \asymp \frac{1}{\eta^{2/3}}\;.
\end{equation} 
The result then follows readily from~\eqref{eq:t_theta}. 
 
\item   \textbf{Region $3$:}
Here, we use a slightly more elaborate bound for $\eta\dot{\theta}$. 
Namely~\eqref{eq:system_regime3} and the assumptions imply 
\begin{equation}
 \eta\dot\theta \gtrsim \tilde g(\theta,r)\;,
\end{equation} 
where 
\begin{equation}
 \tilde g(\theta,r) 
 \asymp r^2 + (\theta - T_2)^2\;,
\end{equation} 
and the result follows from~\eqref{eq:t_theta} with $r\asymp\eta^{1/3}$. 
\qed
\end{itemize}
\renewcommand{\qed}{}
\end{proof}

Note that Proposition~\ref{prop:radial} shows that the bounds~\eqref{eq:bounds_rpm} 
and~\eqref{eq:bounds_Rplus} are satisfied if $r(0)$ is $\eta$-close to $r^*_0(0,\eta)$. 
We will show in Section~\ref{ssec:control_derivative} that the lower bound~\eqref{eq:lower_bound_dmX} is also satisfied.


\subsubsection{Averaging}
\label{sssec:eta_averaging} 

We now bring the equation~\eqref{eq:dthetaX} into a form suitable for averaging.
Similarly to the case $\eta = 0$, we make a change of variables 
\begin{equation}
\label{eq:YX} 
 Y(\theta,m) = X(\theta,m) + 4m\ph(\theta)\;,
\end{equation} 
where $\ph(\theta)$ is now defined by
\begin{equation}
 \ph(\theta) = M(\theta,r^*_0(\theta;\eta)) - M_-\;,
\end{equation} 
where $r^*_0(\theta;\eta)$ is the periodic 
solution for $\eps = 0$, satisfying~\eqref{eq:rstar_asymp}. 
It follows directly from the expressions for $M(\theta,r)$
in Table~\ref{table:polar} and the bounds~\eqref{eq:F0} on $r^*_0(\theta;\eta)$ 
that 
\begin{equation}
 \ph(\theta) = 
 \begin{cases}
  \theta & \text{for $\theta\in[0,T_1]$\;,} \\
  \Delta M + \Order{\eta^{1/3}} & \text{for $\theta\in(T_1,T_4)$\;,} \\
  \Delta M - (\theta-T_4) & \text{for $\theta\in[T_4,T_5]$\;,} \\
  \Order{\eta^{1/3}} & \text{for $\theta\in(T_5,T]$\;.} 
 \end{cases}
\end{equation} 
Furthermore, since $A_0$ describes the evolution of $X$ for $\eps = 0$,  \eqref{eq:coupled2X} with $\eps = 0$ yields  
\begin{equation}
 mA_0 
 = \partial_\theta X
 = \frac{\partial_tX}{\dot\theta}
 = -m\frac{a(E,M)}{\dot\theta} 
 = -4m \frac{\dot M}{\dot\theta}
 = -4m \frac{\6M}{\6\theta}\;, 
\end{equation} 
evaluated along the limit cycle of~\eqref{eq:ODE_ME_eps0}. 
Therefore, we have 
\begin{equation}
 \ph'(\theta) = -\frac14 A_0(\theta,r^*_0(\theta;\eta);\eta)\;.
\end{equation} 
The analogue of Proposition~\ref{prop:conservation_Y_eta0} reads as follows.
\begin{proposition}[Conserved quantity]
\label{prop:conservation_Y_eta_pos}
Assuming $r(\theta)$ satisfies~\eqref{eq:bounds_rpm}, we have
\begin{equation}
\label{eq:conservation_Y_eta_pos} 
 \int_{-1}^0 Y(\theta,m)\6m = - \int_0^1 Y(\theta,m)\6m = 
 \begin{cases}
  2M_- & \text{if $\theta\in I_1$\;,} \\
  2M_- + \Order{\eta^{1/3}} & \text{if $\theta\in I_2\cup I_3$\;.}
 \end{cases}
\end{equation} 
\end{proposition}
\begin{proof}
Remark~\eqref{rem:MX} implies 
\begin{equation}
 M(\theta) 
 = \frac12 \int_{-1}^0 X(\theta,m)\6m 
 = \frac12 \int_{-1}^0 Y(\theta,m)\6m + \ph(\theta)\;.
\end{equation} 
Therefore, 
\begin{equation}
 \int_{-1}^0 Y(\theta,m)\6m
 = 2\bigbrak{M(\theta) - \ph(\theta)}
 = 2\bigbrak{M(\theta) - M(\theta,r^*_0(\theta;\eta)) + M_-}\;.
\end{equation} 
If $\theta\in I_1$, we have $M(\theta) = M(\theta,r^*_0(\theta;\eta))$ 
by construction. For $\theta\in I_2\cup I_3$, the assumption on $r(\theta)$ 
implies that these quantities differ by $\Order{\eta^{1/3}}$ at most.
\end{proof}
Note that by Remark~\ref{rem:MX}, there is an exactly conserved quantity 
that is $\eta^{1/3}$-close to the above integrals. However, the approximate 
relation~\eqref{eq:conservation_Y_eta_pos} will suffice for our purposes.
The following result prepares the computation of the averaged system.

\begin{proposition}
\label{prop:averaging_dthetaY} 
There exists a constant $c_0 > 0$ such that, 
if the initial condition at $\theta = 0$ satisfies $\abs{r(0)-r^*_0(0,\eta)}\leqs c_0$,  
then for all $\theta > 0$, the function $Y$ satisfies 
\begin{equation}
\label{eq:dt_Y} 
 \partial_\theta Y(\theta,m) 
 = \eps A(Y(\theta,m),\theta,r;\eta) 
 + \Order{m\e^{-\kappa/\eta}} + \Order{\eps^2}\;,
\end{equation} 
where $\kappa > 0$, and 
\begin{equation}
 A(Y(\theta,m),\theta,r;\eta) 
 = A_1(Y(\theta,m) - 4m \ph(\theta),\theta,r;\eta)\;. 
\end{equation} 
\end{proposition}
\begin{proof}
Performing the change of variables~\eqref{eq:YX}, we find 
\begin{equation}
 \partial_\theta Y(\theta,m) 
 = m \bigbrak{A_0(\theta,r;\eta) - A_0(\theta,r^*_0(\theta,\eta);\eta)}
 + \eps A_1(Y - 4m \ph(\theta),\theta,r;\eta) + \Order{\eps^2}\;.
\end{equation} 
During the time interval $[0,T_1]$, the term in square brackets is actually equal 
to zero, due to~\eqref{eq:F0}. Since the limit cycle is linearly stable, 
and attracting at rate $1/\eta$ for $\theta\in[0,T_1]$, 
there exists a constant $\kappa > 0$ such that $\abs{r(\theta) - r^*_0(\theta,\eta)}
\leqs c_0 \e^{-\kappa/\eta}$ for any $\theta \geqs T_1$. 
Therefore, the term in square brackets in exponentially small. 
\end{proof}

This result shows that we are indeed in the regime suitable for averaging, provided 
we assume 
\begin{equation}
 \e^{-\kappa/\eta} = \Order{\eps^2}\;,
\end{equation} 
which means that $\eta = \Order{1/\log(\eps^{-1})}$. 
The system~\eqref{eq:dthetaX}--\eqref{eq:dthetar} can thus be written 
\begin{align}
 \partial_\theta Y(\theta,m) 
 &= \eps A(Y(\theta,m),\theta,r;\eta) 
 + \Order{\eps^2}\;, \\
 \eta \frac{\6r}{\6\theta}
 &= B_0(\theta,r;\eta) + \eps B_1(\theta,r,\partial_m Y(\theta,0);\eta,\eps)\;.
\label{eq:dsY11} 
\end{align} 
We can now show the following analogue of Proposition~\ref{prop:avrg_eta0}.

\begin{proposition}[Averaging]
\label{prop:avrg_eta} 
Assume $r(0) = r^*_0(0,\eta) + \Order{\eta}$.
There exists a time-periodic change of variables of the form 
\begin{equation}
 Y(\theta,m) = \bar Y(\theta,m) + \Order{\eps}\;,
\end{equation} 
casting the system~\eqref{eq:dsY11} into the form
\begin{equation}
\label{eq:dsYbar1} 
 \partial_\theta \bar Y(\theta,m) = \eps \bar A(\bar Y(\theta,m)) + \Order{\eps^2}
  + \Order{\eta^{1/3}\eps}\;,
\end{equation} 
where 
\begin{equation}
\label{eq:def_Abar1} 
 \bar A(\bar Y(\theta,m))
 = -4\partial_m \bigbrak{G(\partial_m \bar Y(\theta,m))} 
 - 8mG(\partial_m \bar Y(\theta,0))\;,
\end{equation} 
with 
\begin{equation}
\label{eq:def_G1} 
 G(z) = \frac1T 
 \biggbrak{
 \int_0^{T_1} \frac{\6\theta}{\tilde a(\theta,r^*_0(\theta;0))(z - 4\ph(\theta))}
 - \int_{T_4}^{T_5} \frac{\6\theta}{\tilde a(\theta,r^*_0(\theta;0))(z - 4\ph(\theta))}}\;.
\end{equation} 
\end{proposition}
\begin{proof}
The proof is an adaptation of the proof of Proposition~\ref{prop:avrg_eta0}. We define an average of $A$ by 
\begin{equation}
 \bar A(Y) = \frac1T \int_0^T A(Y,\theta,r^*_0(\theta,0);0) \6\theta\;.
\end{equation} 
Here we choose to define $\bar A$ via $A$ at $\eta = 0$, which is slightly more convenient. 
Another option would be to use $A(Y,\theta,r^*_0(\theta,\eta);\eta)$ instead, but this would not affect 
the leading term in the result. The change of variables is of the form 
\begin{equation}
 Y = \bar Y + \eps w(\theta,\bar Y)\;,
\end{equation} 
where $w$ is a primitive of $A^0 = A - \bar A$. The resulting equation for 
$\bar Y$ is 
\begin{equation}
 \partial_\theta \bar Y(\theta,m) = \eps \bar A(\bar Y(\theta,m)) 
 + \eps\bigbrak{A(\bar Y,\theta,r;\eta) - A(\bar Y,\theta,r^*_0(\theta,0);0)}
 + \Order{\eps^2}\;.
\end{equation} 
We claim that the term in square brackets has order $\eta^{1/3}$. Indeed, for 
$\theta\in I_2\cup I_3$, this follows directly from the expression~\eqref{eq:A1} 
for $A_1$ in Proposition~\ref{prop:X_theta}. 
For $\theta\in I_1$, we note that Proposition~\ref{prop:radial} implies 
$r(\theta) = r^*_0(\theta;0) + \Order{\eta^{1/3}}$, so that
\begin{equation}
 \tilde a(\theta,r(\theta))
 = \tilde a(\theta,r^*_0(\theta,0))[1 + \Order{\eta^{1/3}}]\;,
\end{equation} 
since $\tilde a(\theta,r^*_0(\theta,0))$ is bounded 
away from $0$ for $\theta\in I_1$. By~\eqref{eq:A1}, this implies the claimed bound.

It remains to compute the average $\bar A$ of $A$. For this, we start by computing 
\begin{equation}
 \int_0^m \bar A(Y(\cdot,\bar m))\6\bar m
 = \frac{1}{T} \int_0^T \int_0^m A_1(Y(\cdot,\bar m)-4\bar m\ph(\theta),\theta,r^*_0(\theta,0);0)
 \6\bar m\6 \theta\;.
\end{equation} 
The integral over $\bar m$ has different values depending on the value of $\theta$. 
For $\theta\in[0,T_1)$, it is given by 
\begin{equation}
\label{eq:int_A1_0T1} 
\int_0^m A_1 \6\bar m = 
 -\frac{4}{\tilde a(\theta,r^*_0(\theta,0))}
 \biggbrak{\frac{1}{\partial_m Y(\cdot,m) - 4\ph(\theta)}
 - \frac{1-m^2}{\partial_m Y(\cdot,0) - 4\ph(\theta)}}\;.
\end{equation} 
The same expression, except with opposite sign, 
holds for $\theta\in[T_4,T_5)$. 
For $\theta\in I_2\cup I_3$, the integral of $A_1$ vanishes by~\eqref{eq:A1} since $\eta = 0$.
This yields 
\begin{equation}
 \int_0^m \bar A(Y(\cdot,\bar m))\6\bar m
 = -4\bigbrak{G(\partial_m Y(\cdot,m)) - (1-m^2)G(\partial_m Y(\cdot,0))}\;,
\end{equation} 
with $G$ given by~\eqref{eq:def_G1}, 
and the expression~\eqref{eq:def_Abar1} for $\bar A$ follows by taking the 
derivative with respect to $m$. 
\end{proof}

\begin{remark}
It follows from~\eqref{eq:deg_atilde} and~\eqref{eq:def_r0star} that 
for $\theta\in[0,T_1)$, one has
\begin{equation}
 \tilde a(\theta,r^*_0(\theta;0))
 = a(E_- - r^*_0(\theta;0), M_-+\theta)
 = a(E^*_-(M_-+\theta), M_-+\theta)\;, 
\end{equation} 
and similarly for $\theta\in[T_4,T_5)$. Therefore, the expression~\eqref{eq:def_G1}
for $G$ is actually equivalent to~\eqref{eq:def_G}. 
\end{remark}

\begin{remark}
The assumption that $r(0)$ should be $\eta$-close to $r^*_0(0,\eta)$ is not 
really restrictive, since the slow manifolds $S_{\textup{a}}^{\pm}$ occurring in 
the proof of Proposition~\ref{prop:radial} are attracting neighbouring 
trajectories exponentially fast when $\theta\in I_1$. 
\end{remark}




\section{Periodic solution and stability}
\label{sec:stability} 

In this section, we analyse the averaged equation~\eqref{eq:dsYbar}, 
respectively~\eqref{eq:dsYbar1} for $\eta > 0$, when keeping only the leading 
order term in $\eps$. To cast the equation into a more common form, we first 
scale time by a factor $4\eps$, calling the new time $t$, and make the changes 
of notations $\bar Y\mapsto u$ and $m\mapsto x$, keeping in mind that $x$ and $t$ 
are different from the original space and time variables.
The resulting equation has the form 
\begin{equation}
\label{eq:dtu} 
 \partial_t u(t,x) = -\partial_x\bigbrak{G(\partial_x u(t,x))} - 2x G(\partial_x u(t,0))\;,
\end{equation} 
with boundary conditions 
\begin{equation}
\label{eq:dtu_bc} 
 u(t,0) = 0\;, \qquad 
 \lim_{x\to 1} u(t,x) = +\infty\;.
\end{equation} 
We recall that by Proposition~\ref{prop:symmetry}, we may assume that $u$ is an odd function of $x$. 

The function $G$ is defined in~\eqref{eq:def_G},
respectively~\eqref{eq:def_G1} for $\eta > 0$.
We first establish some properties of $G$.

\begin{proposition}[Properties of $G$]
\label{prop:G} 
There exists $z_0 > 0$ such that $G$ is defined on $(z_0,\infty)$. On that interval,
is has the following properties:
\begin{itemize}
\item   $G(z) > 0$, and 
\begin{equation}
\label{eq:bc_G} 
 \lim_{z\to z_0} G(z) = +\infty\;, \qquad 
 \lim_{z\to +\infty} G(z) = 0\;;
\end{equation} 
\item   $z\mapsto G(z)$ is strictly decreasing;
\item   $z\mapsto G(z)$ is strictly convex. 
\end{itemize}
Furthermore, for large positive $z$, 
\begin{align}
 G(z) &= \frac{G_1}{z} \bigbrak{1 + \Order{z^{-1}}}\;, &
 G''(z) &= \frac{2G_1}{z^3} \bigbrak{1 + \Order{z^{-1}}}\;, \\
 G'(z) &= -\frac{G_1}{z^2} \bigbrak{1 + \Order{z^{-1}}}\;, &
 G'''(z) &= -\frac{6G_1}{z^4} \bigbrak{1 + \Order{z^{-1}}}\;,
\label{eq:G_derivatives} 
\end{align} 
where 
\begin{equation}
 G_1 := \frac1T \int_0^T \frac{\6s}{\abs{\tilde a(s)}}\;.
\end{equation} 
\end{proposition}
\begin{proof}
Let 
\begin{equation}
 z_0 = 4\sup_{0\leqs s \leqs T} \ph(s)\;.
\end{equation}
In the case $\eta = 0$, the supremum is reached for $s = s^* = T/2 = \Delta M$,
and has value $4\Delta M$. For $\eta > 0$, it is reached for some $s^*\in[T_1,T_4]$,
and has value $4T_1 + \Order{\eta^{1/3}} = 4\Delta M + \Order{\eta^{1/3}}$.
The integral~\eqref{eq:bc_G} is well-defined and positive for $z > z_0$. 
For $z\leqs z_0$, there is a non-integrable singularity at $s^*$.
Since 
\begin{equation}
 G'(z) = -\frac1T \int_0^T \frac{\6s}{\abs{\tilde a(s)}(z - 4\ph(s))^2}
\end{equation} 
in the case $\eta = 0$, we see that $G'(z) < 0$ for $z > z_0$, 
and a similar argument applies to the second derivative of $G$. 
The situation is analogous for $\eta > 0$. The asymptotic 
behaviour~\eqref{eq:G_derivatives} is obtained by writing 
\begin{equation}
 \abs{\tilde a(s)}(z - 4\ph(s))
 = z\abs{\tilde a(s)} \bigbrak{1 + \Order{z^{-1}}}\;,
\end{equation} 
which is justified for large $z$ since $\ph$ is bounded, and evaluating the integrals 
defining $G$ and its derivatives. 
\end{proof}


\subsection{Uniqueness of the stationary solution}
\label{ssec:periodic_unique} 

In order to determine stationary solutions of~\eqref{eq:dtu}, we first 
consider the function $v(t,x) = \partial_x u(t,x)$, which satisfies 
\begin{equation}
\label{eq:dtv} ´
 \partial_t v(t,x) = -\partial_{xx}\bigbrak{G(v(t,x))} - 2 G(v(t,0))\;.
\end{equation} 
This is now an even function, satisfying the boundary conditions 
\begin{equation}
 \partial_x v(t,0) = 0\;, \qquad 
 \lim_{x\to 1} v(t,x) = +\infty\;.
\end{equation} 
Next, we set 
\begin{equation}
\label{eq:def_w} 
 w(t,x) = G(v(t,x))\;.
\end{equation} 
This function satisfies 
\begin{equation}
\label{eq:dtw} 
 \partial_t w(t,x) = - G'\bigpar{G^{-1}(w(t,x))}
 \bigbrak{\partial_{xx}w(t,x) - 2w(t,0)}\;.
\end{equation} 
The function $w$ is again an even function of $x$, and Proposition~\ref{prop:G}
implies that the boundary conditions are 
\begin{equation}
\label{eq:w_bc} 
 \partial_x w(t,0) = 0\;, \qquad 
 \lim_{x\to 1} w(t,x) = 0\;.
\end{equation} 
Stationary solutions $w^*$ of~\eqref{eq:dtw} satisfy 
\begin{equation}
 \partial_{xx} w^*(x) = 2 w^*(0)\;,
\end{equation} 
as well as the boundary conditions~\eqref{eq:w_bc}. There is a one-parameter 
family of solutions, given by 
\begin{equation}
 w^*(x) = w^*(0)(1 - x^2)\;.
\end{equation} 
To lift the degeneracy, we recall the existence of a conserved quantity, 
which can be checked by a direct computation as follows.

\begin{proposition}[First integral]
\label{prop:first_integral_averaged} 
Solutions of the equation~\eqref{eq:dtu} conserve the quantity 
\begin{equation}
\label{eq:K_first_integral} 
 K(t) = \int_0^1 u(t,x) \6x\;.
\end{equation} 
\end{proposition}
\begin{proof}
We have 
\begin{align}
 \partial_t K(t) 
 &= \int_0^1 \partial_t u(t,x) \6x \\
 &= -\int_0^1 \partial_x\bigbrak{G(\partial_x u(t,x))} \6x - G(\partial_x u(t,0)) \\
 &= -\lim_{x\to 1} G(\partial_x u(t,x))\;,
\end{align}
which vanishes, thanks to the boundary conditions~\eqref{eq:dtu_bc} and~\eqref{eq:bc_G}. 
\end{proof}

In fact, Propositions~\ref{prop:conservation_Y_eta0} and~\ref{prop:conservation_Y_eta_pos} imply that the constant value of $K(t)$ is 
necessarily equal to $-2M_- + \Order{\eps}$. When analysing the untruncated 
averaged equation in Section~\ref{sec:periodic_unperturbed}, we will verify that 
there is only one periodic orbit. For now, it will be sufficient to know that 
$K(t)$ is conserved, without using its precise value. The form of $G$ imposes 
however a lower bound on $K$.

\begin{proposition}
The equation~\eqref{eq:dtw} admits a stationary 
solution uniquely determined by the value of $K(0)$, provided $K(0) > \frac{z_0}2$.
\end{proposition}
\begin{proof}
Writing $c = w^*(0)$, we have 
\begin{equation}
 u^*(x) = \int_0^x G^{-1}(c(1 - x_1^2)) \6x_1\;.
\end{equation} 
This implies that if $K^*(c)$ denotes the constant value of $K(t)$, then
\begin{align}
 K^*(c)
 &= \int_0^1 \int_0^x G^{-1}(c(1 - x_1^2)) \6x_1 \6x \\
 &= \int_0^1 (1-x_1) G^{-1}(c(1 - x_1^2)) \6x_1\;.
\end{align}
By Proposition~\ref{prop:G}, this is a decreasing function of $c$, satisfying
\begin{equation}
 \lim_{c\to0} K^*(c) = \infty\;, \qquad 
 \lim_{c\to\infty} K^*(c) = \frac{z_0}2\;.
\end{equation} 
Therefore, there exists a unique value of $c$ for any $K(0) = K^*(c) > z_0$. 
\end{proof}

Note that in terms of $v$, the invariant has the expression
\begin{equation}
 K = \int_0^1 \int_0^x v(y)\6y \6x 
 = \int_0^1 (1-y) v(y)\6y\;.
\end{equation} 
In terms of $w$, we conclude that the codimension $1$ manifolds  
\begin{equation}
\label{eq:def_SofK} 
 \sS(K) = 
 \biggset{w:[0,1]\to\R \;\colon\; w(1) = 0, w'(0) = 0, 
 \int_0^1 (1-x)G^{-1}(w(x))\6x = K}
\end{equation} 
are invariant under~\eqref{eq:dtw}. The admissible sets $\set{\sS(K)}_{K > z_0/2}$ 
provide a foliation of the function space. 
As pointed out above, the actual value of $K$ is given by 
$-2M_- + \Order{\eps}$, and 
since $z_0 = 4(M_+ - M_-) + \Order{\eta^{1/3}}$ 
(see the proof of Proposition~\ref{prop:G}), 
the condition $K > \frac{z_0}2$ translates to
$-2M_- + \Order{\eps} \leqs 2(M_+ - M_-) + \Order{\eta^{1/3}}$, 
that is, $M_+ < \Order{\eps} + \Order{\eta^{1/3}}$. In practice, our 
sign convention implies that $M$ has to be negative, and smaller values of $\abs{M}$ 
correspond to steeper density profiles $n(t,x)$.


\subsection{Stability of the stationary solution}
\label{ssec:periodic_stability} 

The equation linearised around a stationary solution $w^*(x)$ is given by 
\begin{equation}
\label{eq:w_linearised} 
 \partial_t w_1(t,x) = H(x)\bigbrak{\partial_{xx} w_1(t,x) + 2 w_1(t,0)}\;, 
\end{equation} 
where 
\begin{equation}
 H(x) = -G'(v^*(x))\;.
\end{equation} 

\begin{proposition}
\label{prop:H}
The function $H$ has the following properties:
\begin{itemize}
\item   $H$ is even, strictly positive on $(-1,1)$, and $H(1) = 0$;
\item   $H$ is strictly decreasing on $(0,1)$ and 
\begin{equation}
\label{eq:limit_Hprime1} 
 \lim_{x\to1} H'(x) = 0\;;
\end{equation} 
\item   $\lim_{x\to1} H''(x) > 0$.  
\end{itemize}
\end{proposition}
\begin{proof}
$H$ is even because $v^*$ is even. It is positive because $G$ is decreasing on $[0,1)$, 
and it is decreasing on $(0,1)$ by convexity of $G$. 
To analyse the behaviour as $x\to 1$, we write $x = 1 - y$, and use the 
fact that 
\begin{equation}
 w^*(1-y) = 2cy + \Order{y^2}
\end{equation} 
where $c = w^*(0)$. By Proposition~\ref{prop:G}, this yields 
\begin{equation}
 v^*(1-y) = G^{-1} \bigpar{2cy + \Order{y^2}} 
 = \frac{G_1}{2cy} \bigbrak{1 + \Order{y}}\;.
\end{equation} 
Taking derivatives of the relation $w^*(x) = G(v^*(x))$, we further obtain 
\begin{equation}
 \partial_x v^*(1-y) = -\frac{G_1}{2cy^2} \bigbrak{1 + \Order{y}}\;, \qquad 
 \partial_{xx} v^*(1-y) = \frac{G_1}{cy^3} \bigbrak{1 + \Order{y}}\;.
\end{equation} 
It follows that 
\begin{equation}
 H(1-y) = -G'\bigpar{v^*(1-y)}
 = 2cy \bigbrak{1 + \Order{y}}\;, 
\end{equation} 
which implies $H(1) = 0$. Furthermore, 
\begin{equation}
 H'(1-y) = -G''\bigpar{v^*(1-y)} \partial_x v^*(1-y)
 = \frac{8c^2}{G_1} y \bigbrak{1 + \Order{y}}\;, 
\end{equation} 
which proves~\eqref{eq:limit_Hprime1}.
Finally, we find 
\begin{align}
 H''(1-y) 
 &= -G'''\bigpar{v^*(1-y)} \partial_x v^*(1-y)^2
 -G''\bigpar{v^*(1-y)} \partial_{xx} v^*(1-y) \\
 &= \frac{8c^2}{G_1} \bigbrak{1 + \Order{y}}\;,
\end{align}
which converges to $8c^2/G_1 > 0$ as $y\to0$. 
\end{proof}

Equation~\eqref{eq:w_linearised} also admits a first integral, as a consequence 
of the first integral~\eqref{eq:K_first_integral}. Indeed, since 
\begin{equation}
 G^{-1}(w^*(x) + w_1(x))
 = G^{-1}(w^*(x)) - \frac{1}{H(x)}w_1(x) + \Order{\norm{w_1(x)}^2}\;, 
\end{equation} 
we conclude that the quantity 
\begin{equation}
 \tilde K(t) = \int_0^1 \frac{1-x}{H(x)} w_1(t,x)\6x
\end{equation} 
should be conserved. This can be checked by a direct computation:

\begin{proposition}
\label{prop:Ktilde} 
The integral $\tilde K(t)$ is conserved by the evolution of~\eqref{eq:w_linearised}. 
\end{proposition}
\begin{proof}
We have 
\begin{align}
\frac{\6}{\6t} \tilde K(t) 
&= \int_0^1 \frac{1-x}{H(x)} \partial_t w_1(t,x)\6x \\
&= \int_0^1 (1-x)\bigbrak{\partial_{xx}w_1(t,x) + 2 w_1(t,0)} \6x \\
&= (1-x)\bigbrak{\partial_x w_1(t,x) + 2x w_1(t,0)} \biggr|_0^1 
+ \int_0^1 \bigbrak{\partial_x w_1(t,x) + 2x w_1(t,0)} \6x \\
&= -\partial_x w_1(t,0) + w_1(t,1) - w_1(t,0) + w_1(t,0) \\ 
&= 0
\end{align}
thanks to the boundary conditions~\eqref{eq:w_bc}.
\end{proof}

Geometrically speaking, the set
\begin{equation}
\label{eq:def_S0} 
 \sS_0 = 
 \biggset{w_1:[0,1]\to\R \;\colon\; w_1(1) = 0, w_1'(0) = 0, 
 \int_0^1 \frac{1-x}{H(x)}w_1(x)\6x = 0}
\end{equation} 
represents the tangent plane to the manifold $\sS(K)$ at $w = w^*$. 

Defining linear operators $\sA_0$ and $\sA_1$ by 
\begin{align}
 (\sA_0 w)(x) &= -H(x) w''(x)\;, \\
 (\sA_1 w)(x) &= -2H(x) w(0)\;,
\label{eq:A0A1} 
\end{align}
we can rewrite the linearised equation~\eqref{eq:w_linearised} as 
\begin{equation}
 \partial_t w_1(t,x) = -(\sA w_1)(t,x)\;, 
 \qquad 
 \sA = \sA_0 + \sA_1\;.
 \label{eq:def_sA} 
\end{equation} 
The stabiliy of $w^*(x)$ is thus related to the spectrum of the 
operator $\sA = \sA_0 + \sA_1$ acting on even functions that vanish in $\pm1$. 
Determining this spectrum is a variant of a Sturm--Liouville problem.
The following result lists some basic properties of $\sA$, $\sA_0$ and $\sA_1$. 

\begin{proposition}
\label{prop:A0} 
The linear operators $\sA_0$, $\sA_1$ and $\sA = \sA_0 + \sA_1$ have the following properties.
\begin{enumerate}
\item   The linear operator $\sA_0$ is self-adjoint in $\sH = L^2([0,1], H(x)^{-1}\6x)$ 
and positive. Therefore, the spectrum of $\sA_0$ is included in $[\delta,+\infty)$ 
for some $\delta > 0$. 
\item   The function $w_0:x\mapsto 1-x^2$ is in the kernel of $\sA$.
\item   $\sA_1$ is a rank-$1$ operator.
\item   The plane $\sS_0$ is invariant under $\sA$ and transversal to 
the line $\vspan(w_0) = \setsuch{\lambda w_0}{\lambda\in\R}$. 
\end{enumerate}
\end{proposition}
\begin{proof}
We work on the Hilbert space $\sH$ of functions $f:[0,1]\to\R$ satisfying 
the boundary conditions $f'(0) = 0$ and $f(1) = 0$, equipped with the inner product 
\begin{equation}
\label{eq:weighted_inner_product} 
 \pscal{f}{g} = \int_0^1 f(x)g(x) \frac{\6x}{H(x)}\;.
\end{equation} 
The boundary condition $f(1) = 0$ is actually a consequence of $\norm{f}_\sH
= \pscal{f}{f}^{1/2}$ being finite. We have 
\begin{equation}
 \pscal{f}{\sA_0 g} 
 = -\int_0^1 f(x) g''(x)\6x 
 = \int_0^1 f'(x)g'(x)\6x 
 = \pscal{\sA_0 f}{g}\;,
\end{equation} 
since the boundary terms vanish owing to the boundary conditions. 
This shows that $\sA_0$ is self-adjoint, while positivity follows from 
the fact that 
\begin{equation}
\label{eq:L1norm} 
 \pscal{f}{\sA_0 f} = \int_0^1 f'(x)^2\6x 
 = \norm{(f')^2}_{L^1}
 \geqs 0\;.
\end{equation} 
In fact, $\pscal{f}{\sA_0 f} = 0$ if and only if $f$ is constant, 
and the only constant function in $\sH$ is the identically zero function. 
This implies the claim on the spectrum of $\sA_0$. 

The fact that $\sA w_0 = 0$ follows from a direct computation, while 
$\sA_1$ has rank $1$ by definition. Finally, the invariance of 
$\sS_0$ is a consequence of Proposition~\ref{prop:Ktilde}, and one 
easily checks that $\pscal{1-x}{w_0} > 0$, implying transversality. 
\end{proof}

Note that we have obtained a decomposition 
\begin{equation}
\label{eq:decomp_H} 
 \sH = \sS_0 \oplus \vspan(w_0)
\end{equation} 
of $\sH$ into subspaces invariant under $\sA$, where $\vspan{w_0}$ 
corresponds to the eigenvalue $0$. Our aim is now to show that 
all eigenvalues of $\sA$ restricted to $\sS_0$ have a strictly 
positive real part. 
To this end we analyse the Sturm--Liouville eigenvalue problem 
\begin{equation}
\label{eq:w_SL} 
 \sA w = \lambda w 
\end{equation} 
for $w\in\sH$. The following lemma is a classical result from the theory 
of rank $1$ perturbations.

\begin{lemma}
\label{lem:rank1}
Let $\lambda$ be an eigenvalue of $\sA$ which is not an eigenvalue of $\sA_0$. Then 
\begin{equation}
\label{eq:defF} 
 F(\lambda) := H(0)\pscal{\delta_0}{(\sA_0-\lambda)^{-1}H} = \frac12\;, 
\end{equation} 
where $(\sA_0-\lambda)^{-1}$ denotes the resolvent of $\sA_0$, 
and $\delta_0$ is the Dirac distribution at $0$, so that 
$H(0)\pscal{\delta_0}{f} = f(0)$. 
\end{lemma}
\begin{proof}
The eigenvalue equation~\eqref{eq:w_SL} can be rewritten as 
\begin{equation}
\label{eq:A0w} 
 \bigpar{(\sA_0 - \lambda) w}(x) = 2w(0) H(x)\;,
\end{equation} 
which is equivalent to 
\begin{equation}
\label{eq:A0w1} 
 w(x) = 2w(0) \bigpar{(\sA_0 - \lambda)^{-1}H}(x)\;.
\end{equation} 
If $w(0)$ were equal to $0$, \eqref{eq:A0w} would imply that $w$ is an 
eigenfunction of $\sA_0$, which is excluded by our assumption that 
$\lambda$ is not an eigenvalue of $\sA_0$. Therefore, we may evaluate~\eqref{eq:A0w1}
at $x = 0$, and divide by $w(0)$, which yields 
\begin{equation}
 1 = 2 \bigpar{(\sA_0 - \lambda)^{-1}H}(0)\;.
\end{equation} 
This is equivalent to~\eqref{eq:defF}. 
\end{proof}

Since $w(0)$ is in the kernel of $\sA$, $0$ is an eigenvalue of $\sA$, and we should have 
$F(0) = \frac12$. This can be checked directly. 

\begin{proposition}
\label{prop:F(0)} 
One has $F(0) = \frac12$. 
\end{proposition}
\begin{proof}
Let $v = \sA_0^{-1}H$. This means that $\sA_0 v = H$, which is equivalent to 
$v''(x) = -1$.
The solution in $\sH$ is $v(x) = \frac12(1-x^2)$. 
In particular, $F(0) = H(0)\pscal{\delta_0}{v} = v(0) = \frac12$. 
\end{proof}

The resolvent can be represented as the Laplace transform 
\begin{equation}
\label{eq:resolvent_Laplace} 
 (\sA_0 - \lambda)^{-1} 
 = \int_0^\infty \e^{\lambda t} \e^{-t\sA_0} \6t\;,
\end{equation} 
where $\e^{-t\sA_0}$ is the \lq\lq heat kernel\rq\rq\ associated with $\sA_0$. 
It is defined by $u(t,x) = (\e^{-t\sA_0}g)(x)$ being the solution of 
\begin{align}
 \partial_t u(t,x) &= -\sA_0 u(t,x)\;, \\  
 \lim_{t\to0} u(t,x) &= g(x)\;.
\label{eq:heat} 
\end{align} 
By the Hille--Yosida theorem, since the spectrum of $\sA_0$ is contained in 
$[\delta,+\infty)$, the operators $\e^{-t\sA_0}$ form a strongly continuous 
semigroup with generator $-\sA_0$. Then a proof of the representation~\eqref{eq:resolvent_Laplace} 
of the resolvent can be found, for instance, in~\cite[Proposition~4.11]{Hairer_LN_2009}.

It follows from~\eqref{eq:defF} and~\eqref{eq:resolvent_Laplace} that we can rewrite 
\begin{equation}
 F(\lambda) = \int_0^\infty \e^{\lambda t} q(t) \6t\;, \qquad 
 q(t) := H(0)\pscal{\delta_0}{\e^{-t\sA_0}H}\;,
\end{equation} 
that is, $F$ is the Laplace transform of a positive measure. 

\begin{proposition}
\label{prop:a} 
The function $q$ has the following properties:
\begin{itemize}
\item   $q(0) > 0$. 
\item   $t\mapsto q(t)$ is continuous on $[0,\infty)$;
\item   there exists a constant $C > 0$ such that 
$\abs{q(t)} \leqs C\e^{-\delta t}$ for all $t\geqs0$, where $\delta$ is the bottom of the spectrum 
of $\sA_0$, cf.\ Proposition~\ref{prop:A0};
\item   $q(t)\geqs 0$ for all $t\geqs0$. 
\end{itemize}
\end{proposition}
\begin{proof}
We have $q(0) = H(0)$, which is positive by Proposition~\ref{prop:H}. 

Continuity of $q$ at $0$ follows by writing 
$q(t) - q(0) = (\e^{-t\sA_0}H - H)(0)$ as an integral of $\e^{-t\lambda} - 1$ 
against the spectral measure of $\sA_0$. Continuity 
at any $t>0$ then follows from the semigroup property.

The exponential decay of $q(t)$ can be shown by writing 
\begin{equation}
 q(t) = H(0)\pscal{\delta_0}{\e^{-t\sA_0}H} 
 = H(0)\pscal{\sA_0^{-1/2}\delta_0}{\e^{-t\sA_0}\sA_0^{1/2}H}\;. 
\end{equation} 
Indeed, if $\sA_0^{1/2}H \in L^2$ and $\sA_0^{-1/2}\delta_0 \in L^2$, this entails 
the required bound 
\begin{equation}
 \abs{q(t)} 
 \leqs H(0)\e^{-\delta t} \bignorm{\sA_0^{1/2}H}_{L^2} \bignorm{\sA_0^{-1/2}\delta_0}_{L^2}\;,
\end{equation} 
since the spectrum of $\sA_0$ is included in $[\delta,\infty)$. 
We have $\sA_0^{1/2}H \in L^2$, because 
\begin{equation}
 \bignorm{\sA_0^{1/2}H}_{L^2}^2 
 = \pscal{\sA_0^{1/2}H}{\sA_0^{1/2}H}
 = \pscal{\sA_0H}{H}
 = \bignorm{(H')^2}_{L^1}
 < \infty\;,
\end{equation} 
cf.~\eqref{eq:L1norm}.
To show that $\sA_0^{-1/2}\delta_0\in L^2$, we write 
\begin{equation}
 \bignorm{\sA_0^{-1/2}\delta_0}_{L^2}^2 
 = \pscal{\sA_0^{-1/2}\delta_0}{\sA_0^{-1/2}\delta_0} 
 = \pscal{\delta_0}{\sA_0^{-1}\delta_0}
 = \lim_{\eps\to0} \pscal{\ph_\eps}{\sA_0^{-1}\ph_\eps}\;,
\end{equation} 
where $\ph_\eps$ is a regularisation of $\delta_0$. To this end, we fix a $\sC^\infty$ 
function $\ph$ supported on $[-1,1]$ and of integral $1$, and set $\ph_\eps(x) = 
\eps^{-1}\ph(\eps^{-1} x)$. Then $v_\eps = \sA_0^{-1}\ph_\eps$ satisfies 
\begin{equation}
 v_\eps''(x) = -\frac{1}{\eps} \frac{\ph(\eps^{-1}x)}{H(x)}\;.
\end{equation} 
Let $c_\eps = v_\eps'(-1)$. Since $\ph_\eps$ is supported on $[-\eps,\eps]$, 
$v_\eps'$ is constant outside this interval, while for $x\in[-\eps,\eps]$, 
\begin{equation}
 v_\eps'(x) = c_\eps - \int_{-1}^{\eps^{-1}x} \frac{\ph(z)}{H(\eps z)}\6z\;.
\end{equation} 
Since $v_\eps$ should be even, $v_\eps'$ should be odd and we set 
\begin{equation}
 c_\eps = \frac12 \int_{-1}^1 \frac{\ph(z)}{H(\eps z)} \6z\;.
\end{equation} 
Integrating once again, we find for $x\in[-\eps,\eps]$
\begin{equation}
  v_\eps(x) 
 = c_\eps(1-\eps) 
 + \int_{-\eps}^x \Bigbrak{c_\eps - \int_{-1}^{y/\eps} \frac{\ph(z)}{H(\eps z)}\6z} \6y
 = c_\eps + \Order{\eps}\;.
\end{equation} 
It follows that 
\begin{equation}
 \pscal{\ph_\eps}{\sA_0^{-1}\ph_\eps}
 = \pscal{\ph_\eps}{v_\eps} 
 = \frac{1}{\eps} \int_{-\eps}^\eps \ph(\eps^{-1}x) v_\eps(x)\frac{\6x }{H(x)}
 = \int_{-1}^1 \ph(y) v_\eps(\eps y)\frac{\6y}{H(\eps y)} 
 = \frac{c_\eps}{H(0)} + \Order{\eps}
\end{equation}
since $\ph$ has integral $1$. 
This converges to $\frac12 H(0)^{-2} < \infty$ as $\eps\to 0$, completing the proof that 
$\smash{\sA_0^{-1/2}\delta_0\in L^2}$, and therefore of the exponential decay of $q$. 

Finally, positivity of $q$ is a consequence of the maximum principle
applied to~\eqref{eq:heat} with $g = H$. 
\end{proof}

We can now show that the Sturm--Liouville problem~\eqref{eq:w_SL} has 
only solutions with strictly positive real part, except $\lambda = 0$. 

\begin{proposition}
$\lambda = 0$ is the only solution of the Sturm--Liouville problem~\eqref{eq:w_SL}
with negative or zero real part. Furthermore, there exists a constant $C > 0$ such 
that 
\begin{equation}
\label{eq:bound_F_distance} 
 \abs{F(\lambda)} \leqs 
 \frac{C}{\dist(\lambda,\spec(\sA_0))}\;,
\end{equation} 
where $\dist(\lambda,\spec(\sA_0)) = \inf\setsuch{\abs{\lambda-\lambda_0}}{\lambda_0 
\in \spec(\sA_0)}$ is the distance between $\lambda$ and the spectrum of $\sA_0$. 
Finally, $0$ is a simple eigenvalue of $\sA$. 
\end{proposition}
\begin{proof}
We have 
\begin{equation}
 F(0) = \int_0^\infty q(t) \6t < \infty
\end{equation} 
by Proposition~\ref{prop:a} (this was also shown in Proposition~\ref{prop:F(0)}). 
If $\re(\lambda) < 0$, then 
\begin{equation}
 \abs{F(\lambda)} 
 \leqs \int_0^\infty \e^{\re(\lambda)t} \abs{q(t)} \6t 
 < F(0)\;,
\end{equation}
where the inequality is strict because $q$ is continuous and 
$\e^{\re(\lambda)t} < 1$ for $t > 0$. 
If $\lambda = \icx\omega$ is purely imaginary, then 
\begin{equation}
 F(\icx\omega) 
 = \int_0^\infty \cos(\omega t)q(t)\6t 
 + \icx \int_0^\infty \sin(\omega t)q(t)\6t\;.
\end{equation} 
Therefore, the condition $F(\icx\omega) - F(0) = 0$ implies in particular 
\begin{equation}
 \int_0^\infty (1 - \cos(\omega t))q(t)\6t = 0\;. 
\end{equation} 
Since $q(t)\geqs 0$, this implies $(1 - \cos(\omega t))q(t) = 0$ for 
all $t\geqs 0$. Since $q(0) = 1$, this can only be the case if $\omega = 0$. 
The bound~\eqref{eq:bound_F_distance} follows from a classical bound in~\cite{Kato95}.
Finally, we have 
\begin{equation}
 F'(0) = \int_0^\infty t q(t)\6t > 0\;,
\end{equation} 
which shows that $0$ is a simple eigenvalue of $\sA$, because the resolvent 
is a meromorphic function. 
\end{proof}

\begin{corollary}[Spectral gap]
\label{cor:spectral_gap} 
There exists $\delta_0 > 0$ such that all eigenvalues of $\sA$ restricted to 
the invariant hyperplane $\sS_0$ have a real part greater or equal than $\delta_0$.
Furthermore, all eigenvalues have an imaginary part satisfying 
$\abs{\im(\lambda)} \leqs 2C$ (see Figure~\ref{fig:spectrum}).
\end{corollary}
\begin{proof}
Since $\sA w_0 = 0$ and $0$ is simple, the eigenvalues of the restriction of $\sA$ 
to $\sS_0$ have strictly positive real part. Furthermore, if 
$\abs{\im(\lambda)} > 2C$, since the spectrum of $\sA_0$ is contained in the real line, \eqref{eq:bound_F_distance} shows that $\abs{F(\lambda)} < \frac12$, so that $\lambda$ 
cannot be an eigenvalue. This implies that there cannot be any sequence of eigenvalues with increasing imaginary part converging to the imaginary axis.
\end{proof}

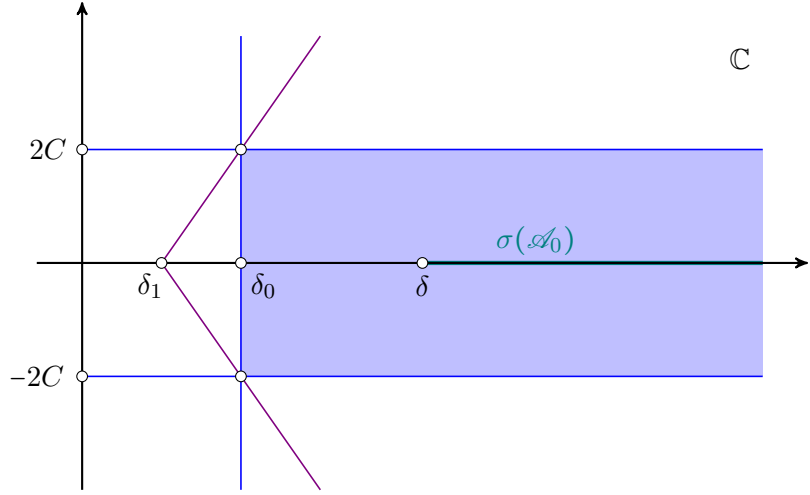
\begin{figure}
 \begin{center}
\scalebox{1.0}{
\begin{tikzpicture}[>=stealth',main node/.style={circle,minimum
size=0.01cm,inner sep=0.05cm,fill=white,draw},x=3cm,y=3cm]

\newcommand*{\ddelta}{1.5}
\newcommand*{\deltaz}{0.7}
\newcommand*{\deltaone}{0.35}
\newcommand*{\CC}{0.5}
\pgfmathsetmacro{\xx}{\deltaone + (\deltaz - \deltaone)/\CC};

\path[fill=blue!25,line width=0] ({\deltaz},{-\CC}) rectangle (3,{\CC}); 

\draw[ultra thick,teal] (\ddelta,0) -- (3,0);

\draw[->,thick] (-0.2,0) -> (3.2,0);
\draw[->,thick] (0,-1) -> (0,1.15);

\draw[semithick, blue] (\deltaz,-1) -- (\deltaz,1); 
\draw[semithick, blue] (0,\CC) -- (3,\CC);
\draw[semithick, blue] (0,-\CC) -- (3,-\CC);
\draw[semithick, violet] (\deltaone,0) -- (\xx,1);
\draw[semithick, violet] (\deltaone,0) -- (\xx,-1);

\node[main node] at (\ddelta,0) {};
\node[main node] at (\deltaz,0) {};
\node[main node] at (\deltaone,0) {};
\node[main node] at (0,\CC) {};
\node[main node] at (0,{-\CC}) {};
\node[main node] at (\deltaz,\CC) {};
\node[main node] at (\deltaz,{-\CC}) {};

\node[] at ({\ddelta},-0.1) {$\delta$}; 
\node[] at ({\deltaz + 0.1},-0.1) {$\delta_0$}; 
\node[] at ({\deltaone - 0.05},-0.1) {$\delta_1$}; 
\node[] at (-0.15,{\CC}) {$2C$}; 
\node[] at (-0.2,{-\CC}) {$-2C$}; 
\node[] at (2.9,0.9) {$\C$};
\node[teal] at (2,0.1) {$\sigma(\sA_0)$};

\end{tikzpicture}
}
\vspace{-5mm}
\end{center}
\caption[]{The spectrum of $\sA_0$ belongs to $[\delta,\infty)$. 
Corollary~\ref{cor:spectral_gap} shows that the spectrum of $\sA$
restricted to the invariant plane $\sS_0$ is 
located in the shaded region. The sector with vertex $(\delta_1,0)$ 
is used in Section~\ref{ssec:control_derivative} to show that the 
semigroup $\e^{-t\sA}$ is analytic.}
\label{fig:spectrum} 
\end{figure}


\subsection{Controlling the derivative of $u$}
\label{ssec:control_derivative} 

The purpose of this section is to show that if the initial condition $u(0,\cdot)$ 
is close, in a suitable sense, to the equilibrium solution $u^*$, then the derivative 
$\partial_x u(t,x)$ remains strictly positive for all times. This is needed to justify 
the assumption~\eqref{eq:assump_q0}--\eqref{eq:lower_bound_dmX} on the lower bound 
of $\partial_m X(t,0)$, as well as the validity of the change of variables~\eqref{eq:def_X}
from $n$ to $X$. In view of the definition~\eqref{eq:def_w} of $w$, it is sufficient 
to show the existence of constants $c_+ > c_- > 0$ such that 
\begin{equation}
\label{eq:bound_w} 
 c_- (1-x^2) \leqs w(t,x) \leqs c_+ (1-x^2)
\end{equation} 
for all $t\geqs0$. 

The first step is to make the change of variables 
\begin{equation}
 w(t,x) = w^*(x) + z(t,x)\;,
\end{equation} 
which yields the equation 
\begin{align}
\label{eq:ztx} 
 \partial_t z(t,x) 
 &= -G'(G^{-1}(w^*(x) + z(t,x))) \bigbrak{\partial_{xx} z(t,x) + 2z(t,0)} \\
 &= -(\sA z)(t,x) + b(z(t,x),t,x)\;,
\end{align}
where $b$ is a nonlinear term of order $\norm{z(t,\cdot)}^2$, and we 
may assume that the initial condition $z(0,\cdot)$ belongs to the invariant 
plane $\sS_0$. By Duhamel's formula, the solution satisfies  
\begin{equation}
\label{eq:Duhamel_z} 
 z(t,x) = \e^{-t\sA} z(0,x) 
 + \int_0^t \e^{-(t-s)\sA} b(z(s,x),s,x)\6s\;,
\end{equation} 
where $\e^{-t\sA}$ denotes the semigroup generated by $-\sA$. 
We will make use of the following non-linear analogue of Gr\"onwall's lemma.

\begin{lemma}
\label{lem:nonlin_Gronwall}  
Assume that there are constants $C, \delta_1 > 0$ and a norm $\norm{\cdot}$ such that 
\begin{equation}
\label{eq:bound_semigroup} 
 \bignorm{\e^{-t\sA}z} \leqs C\e^{-\delta_1 t}\norm{z}
\end{equation} 
for all $z$ in the domain of $\e^{-t\sA}$ and all $t\geqs0$. Then for any $h > 0$, there 
exists $h_0 > 0$ such that the solution of~\eqref{eq:ztx} with initial condition satisfying 
$\norm{z(0,\cdot)} \leqs h_0$ satisfies 
\begin{equation}
 \norm{z(t,\cdot)} \leqs h \qquad \forall t\geqs0\;.
\end{equation} 
\end{lemma}
\begin{proof}
Let $M>0$ be such that 
$\norm{b(z,t,x)} \leqs M\norm{z}^2$, and let 
\begin{equation}
 \tau = \inf\bigsetsuch{t>0}{\norm{z(t,\cdot)} > h}\;.
\end{equation} 
Then, if $\norm{z(0,\cdot)} \leqs h_0 < h$, \eqref{eq:Duhamel_z} yields 
\begin{equation}
 \norm{z(t\wedge\tau,\cdot)}
 \leqs C\e^{-\delta_1 t} h_0 + \int_0^{t\wedge\tau} C\e^{-\delta_1(t-s)} Mh^2 \6s 
 \leqs C h_0 + \frac{CMh^2}{\delta_1}\;.
\end{equation} 
We may assume that $h < \delta_1(2CM)^{-1}$, since if the claim holds for some $h$, it
holds for all larger $h$. Taking $h_0 = h(2C)^{-1}$ yields 
$\norm{z(t\wedge\tau,\cdot)} < h$ for all $t\geqs0$, which shows that $\tau = +\infty$. 
\end{proof}

In order to prove~\eqref{eq:bound_w}, it remains to choose an appropriate norm 
and to show that it satisfies~\eqref{eq:bound_semigroup}. To this end, we note that 
even though $\sA$ is not self-adjoint, Corollary~\ref{cor:spectral_gap} implies that 
$\e^{-t\sA}$ is an analytic semigroup, in the sense 
of~\cite[Definition~4.1]{Hairer_LN_2009}. Indeed, this follows, for instance, 
from~\cite[Theorem~4.22]{Hairer_LN_2009}. Properties of analytic semigroups then 
imply that the bound~\eqref{eq:bound_semigroup} holds for the weighted $L^2$-norm 
given by the inner product~\eqref{eq:weighted_inner_product}. The constant 
$\delta_1 > 0$ can be taken as the tip of a sector in the complex plane containing 
the spectrum of $\sA$ restricted to the invariant plane $\sS_0$ (see 
Figure~\ref{fig:spectrum}). We now want to upgrade this to a stronger Sobolev-like 
norm, allowing to control $z$ pointwise. 

\begin{proposition}
\label{prop:Sobolev_type} 
The norm $\norm{\cdot}$ defined by 
\begin{equation}
\label{eq:norm_Sob_type} 
 \norm{z}^2 
 = z(0)^2 + \int_{-1}^1 
 \Bigbrak{\frac{z(x)^2}{H(x)} + z'(x)^2} \6x
\end{equation} 
satisfies~\eqref{eq:bound_semigroup} with $C$ replaced by some finite constant $C_1 \geqs C$. 
\end{proposition}
\begin{proof}
We will relate the norm~\eqref{eq:norm_Sob_type} to norms of the form $\norm{\sA^\alpha z}_{L^2}$ 
for suitable values of $\alpha$, and use the fact that 
\begin{equation}
\label{eq:norm_exp_decay} 
 \norm{\sA^\alpha \e^{-t\sA}z}_{L^2} 
 = \norm{\e^{-t\sA}\sA^\alpha z}_{L^2} 
 \leqs C\e^{-\delta_1 t} \norm{\sA^\alpha z}_{L^2}\;. 
\end{equation} 
We already mentioned the case $\alpha = 0$, which allows to deal with the term in 
$z(x)^2 H(x)^{-1}$ of the norm. For $\alpha = \frac12$, we find 
\begin{align}
 \bignorm{\sA^{1/2}z}_{L^2}^2 
 = \pscal{\sA^{1/2}z}{\sA^{1/2}z} 
 &= \pscal{z}{\sA z} \\
 &= \pscal{z}{\sA_0 z} + \pscal{z}{\sA_1 z} \\
 &= -\int_{-1}^1 z(x) z''(x)\6x - 2z(0) \int_{-1}^1 z(x)\6x \\
 &= \int_{-1}^1 z'(x)^2 \6x - 2z(0) \int_{-1}^1 z(x)\6x\;.
\end{align}
To obtain the last line, we have integrated by parts, and used the fact that 
$\norm{z}_{L^2} < \infty$ implies that $z$ vanishes in $\pm1$. Writing 
$z(t,\cdot) = \e^{-t\sA}z$, by \eqref{eq:norm_exp_decay}, 
this implies 
\begin{equation}
\label{eq:bound_H1} 
 \int_{-1}^1 \partial_xz(t,x)^2 \6x 
 - 2 z(t,0) \int_{-1}^1z(t,x) \6x
 \leqs C^2 \e^{-2\delta_1 t} \Bigbrak{\int_{-1}^1 z'(x)^2 \6x + 2\abs{z(0)}\int_{-1}^1 \abs{z(x)}\6x}\;.
\end{equation} 
The Cauchy--Schwarz inequality yields 
\begin{equation}
\label{eq:H1_CS} 
 \Bigabs{\int_{-1}^1z(t,x) \6x} 
 = \pscal{z(t,\cdot)}{H} 
 \leqs c_0^{1/2}\norm{z(t,\cdot)}_{L^2}\;, \qquad  
 c_0 := \int_{-1}^1 H(x)\6x\;.
\end{equation} 
Introduce now, for a $\gamma > 0$ to be chosen below, the time
\begin{equation}
 \tau = \inf\bigsetsuch{t>0}{\abs{z(t,0)} > \gamma C \e^{-\delta_1t}\norm{z(0,\cdot)}}\;.
\end{equation} 
Then we obtain 
\begin{equation}
 \int_{-1}^1 \partial_xz(t\wedge\tau,x)^2 \6x
 \leqs \bigbrak{1 + (2+\gamma) c_0^{1/2}}C^2 
 \e^{-2\delta_1 (t\wedge\tau)} \norm{z(0,\cdot)}^2\;. 
\end{equation} 
This shows that $z$ belongs to $H^1$ up to time $\tau$, so that by Morrey's inequality, 
it also belongs to $\sC^{1/2}$. In particular, there exists a universal constant 
$C'$ such that 
\begin{equation}
 z(t\wedge\tau,0)^2 
 \leqs C' \bigbrak{1 + (2+\gamma) c_0^{1/2}}C^2 \e^{-2\delta_1 (t\wedge\tau)} \norm{z(0,\cdot)}^2\;. 
\end{equation} 
Taking $\gamma$ such that $\gamma^2 > C' \bigbrak{1 + (2+\gamma) c_0^{1/2}}$, this yields 
$\tau = +\infty$, from which the result follows. 
\end{proof}

\begin{corollary}
\label{cor:u_derivative} 
Assume $w(0,\cdot)$ is sufficiently close, in the norm~\eqref{eq:norm_Sob_type}, to 
$w^*$. Then $x\mapsto u(t,x)$ is strictly increasing for all $t\geqs0$.
\end{corollary}
\begin{proof}
Combining Proposition~\ref{prop:Sobolev_type} and Lemma~\ref{lem:nonlin_Gronwall}, 
we obtain that for any $h > 0$, there exists $h_0 > 0$ such that if 
$\norm{z(0,\cdot)} \leqs h_0$, then $\norm{z(t,\cdot)} \leqs h$ for all $t\geqs0$. 
By Morrey's inequality, this implies that $z(t,\cdot)$ is also pointwise 
small, so that~\eqref{eq:bound_w} is satified for $h_0$ small enough. This implies 
\begin{equation}
 G^{-1}(c_+(1-x^2))
 \leqs \partial_x u(t,x) 
 \leqs  G^{-1}(c_-(1-x^2))
\end{equation} 
for all $t$, which is strictly positive. 
\end{proof}

Using the properties of $G$ obtained in Propositions~\ref{prop:G} and~\ref{prop:H}, 
one checks that $w(t,\cdot)$ being close to $w^*$ in the norm~\eqref{eq:norm_Sob_type} 
is equivalent to $u(t,\cdot)$ being $H^2$-close to $u^*$.


\section{Extension to the untruncated averaged equation}
\label{sec:periodic_unperturbed} 

To complete the proof of Theorem~\ref{thm:main}, we have to show that the 
remainder of order $\eps^2$ in~\eqref{eq:dsYbar}, respectively 
of order $\eps^2 + \eta^{1/3}\eps$ in~\eqref{eq:dsYbar1}, does not change 
the uniqueness and stability of the periodic orbit. We do this in the limiting 
case $\eta = 0$ in Section~\ref{ssec:unperturbed_eta0}, and for small 
positive $\eta$ in Section~\ref{ssec:unperturbed_eta}.


\subsection{The case $\eta = 0$}
\label{ssec:unperturbed_eta0} 

The main idea to show uniqueness and stability of a periodic orbit is 
to use a Poincar\'e map and the implicit function theorem. 
In order to deal with the zero eigenvalue of the linearisation, 
we will use the fact, shown in Proposition~\ref{prop:conservation_Y_eta0}, 
that 
\begin{equation}
\label{eq:integral_Y} 
 \int_0^1 Y(s,m) \6m = -2M_-
\end{equation} 
for all $s\geqs0$. The value $-2M_-$ is determined by the function $g(E,M)$, 
so it is fixed by the equation. 

\begin{proof}[{\sc Proof of Theorem~\ref{thm:main} when $\eta = 0$}]
Let $Y^*(m)$ be the unique solution of
\begin{equation}
 \bar A(Y^*(m)) = 0
\end{equation} 
compatible with~\eqref{eq:integral_Y}.  
We define a Poincar\'e map $\Pi$ by 
\begin{equation}
\Pi: (Y(0,m),\eps) \longmapsto Y(T,m)\;,
\end{equation} 
that is, $\Pi$ maps the initial condition to the solution at time $T$. 
Since $Y = \bar Y + \eps w(s,\bar Y)$, where $w$ is $T$-periodic in $s$, 
and $w(0,\bar Y) = 0$ thanks to~\eqref{eq:def_w0}, we also have 
$Y(0,m) = \bar Y(0,m)$ and $Y(T,m) = \bar Y(T,m)$, and therefore 
\begin{equation}
 \Pi(\bar Y(0,m),\eps) = \bar Y(T,m)\;, 
\end{equation} 
where $\bar Y(s,m)$ satisfies the equation~\eqref{eq:dsYbar}. 
The time scaling $t = \eps s$ transforms~\eqref{eq:dsYbar} into 
\begin{equation}
\label{eq:dtYbar} 
 \partial_t \bar Y(t,m) = \bar A(\bar Y(t,m)) + \eps R(t,\bar Y(t,m))\;,
\end{equation} 
for some smooth remainder $R$, where now the scaled time $t$ varies in $[0,\eps T]$. 
We write
\begin{equation}
\label{eq:Y0m} 
 \bar Y(t,m) = Y^*(m) + Z(t,m)\;, 
\end{equation} 
where the initial condition $Z(0,m)$ satisfies  
\begin{equation}
\label{eq:Z_integral} 
 \int_0^1 Z(0,m) \6m = 0\;,
\end{equation} 
so that $\bar Y(0,m)$ is compatible with~\eqref{eq:integral_Y}.
Then $Z$ satisfies an equation of the form 
\begin{align}
 \partial_t Z(t,m) 
 &= \bar A(Y^*(m) + Z(t,m)) + \eps R(t,Y^*(m) + Z(t,m)) \\
 &= -\tilde\sA Z(t,m) + b(Z(t,m)) + \eps R(t,Y^*(m) + Z(t,m))\;,
\end{align} 
where $b(Z) = \Order{\norm{Z}_{L^2}^2}$, and $-\tilde\sA$ is the G\^ateaux derivative 
of $\bar A$ at $Y^*(m)$, which is equivalent to the operator $-\sA$ defined 
in~\eqref{eq:def_sA} modulo the change of variables $w = G(\partial_m\bar Y)$. 
By Duhamel's principle (variation of constants formula), we have 
\begin{equation}
 Z(t,\cdot) 
 = \e^{-t\tilde\sA} Z(0,\cdot) 
 + \int_0^t \e^{-(t-s)\tilde\sA} b(Z(s,\cdot)) \6s 
 + \eps \int_0^t \e^{-(t-s)\tilde\sA} R(s,Y^*(\cdot) + Z(s,\cdot)) \6s\;,
\end{equation} 
where $\e^{-t\tilde\sA}$ denotes the semigroup of the linear equation 
$\partial_t Z = -\tilde\sA Z$. Let $C_0 = \norm{Z(0,\cdot)}_{L^2}$, and define, 
given $C_1 > C_0$, 
\begin{equation}
 \tau = 
 \inf\setsuch{t > 0}{\norm{Z(t,\cdot)}_{L^2} = C_1}\;.
\end{equation} 
Then there exist constants $K_0, K_1 > 0$ such that for all $t\leqs\tau\wedge\eps T$, 
one has 
\begin{equation}
 \norm{Z(t,\cdot)}_{L^2} 
 \leqs K_0 C_0 + \eps K_1(C_1^2 + \eps)\;.
\end{equation} 
Choosing for instance $C_1 = 2K_0C_0$, we obtain 
\begin{equation}
 \norm{Z(\tau\wedge\eps T,\cdot)}_{L^2} 
 \leqs K_0 C_0 \brak{1+4\eps K_1K_0C_0} + \eps^2 K_1\;.
\end{equation} 
Taking first $C_0$ small enough for the term in brackets to be smaller 
than $\frac32$, say for $\eps\leqs 1$, and then $\eps$ small enough to ensure 
$\eps^2 K_1\leqs \frac12 K_0C_0$, we obtain $\tau \geqs \eps T$, so that we can write 
\begin{equation}
 Z(\eps T,\cdot) 
 = \e^{-\eps T\tilde\sA} Z(0,\cdot) 
 + \Order{\eps\norm{Z(0,\cdot)}_{L^2}^2}
 + \Order{\eps^2}\;.
\end{equation} 
The Poincar\'e map $\Pi$ thus has the form 
\begin{equation}
 \Pi(Y^* + Z, \eps)
 = Y^* + \e^{-\eps T\tilde\sA} Z 
 + \Order{\eps\norm{Z}_{L^2}^2}
 + \Order{\eps^2}\;.
\end{equation} 
It follows that the map 
\begin{equation}
 \sF(Y,\eps) = \frac{\Pi(Y,\eps)-Y}{\eps}
\end{equation} 
satisfies 
\begin{align}
 \sF(Y^* + Z,\eps)
 &= \frac{1}{\eps}(\e^{-\eps T\tilde\sA} - \id) Z
 + \Order{\norm{Z}_{L^2}^2}
 + \Order{\eps} \\
 &= -T\tilde\sA Z 
 + \Order{\norm{Z}_{L^2}^2}
 + \Order{\eps}\;.
\end{align}
In particular, we have 
\begin{equation}
 \sF(Y^*,0) = 0\;,
\end{equation} 
while the G\^ateaux derivative of $\sF$ in the direction $Z$ at $(Y^*,0)$ is given by 
\begin{equation}
 D\sF(Y^*,0)[Z]
 := \lim_{h\to0} \frac{\sF(Y^*+hZ) - \sF(Y^*)}{h}
 = -T\tilde\sA Z\;.
\end{equation} 
We know from Section~\ref{ssec:periodic_stability} that the linear operator 
$-\tilde\sA$ has a single zero eigenvalue, while all other eigenvalues have a strictly 
negative real part. 
However, since $Z$ satisfies~\eqref{eq:Z_integral}, the initial condition 
$Y^* + Z$ belongs to the invariant manifold $\sS(-2M_-)$ (cf.~\eqref{eq:def_SofK}), 
which translates to $Z$ belonging to the tangent space to $\sS(-2M_-)$ at 
$Y^*$, which is equivalent to the hyperplane $\sS_0$ (cf.~\eqref{eq:def_S0}). 
In this subspace, $\tilde\sA$ has no zero eigenvalue. Therefore, the implicit 
function theorem implies that for small positive $\eps$, the equation 
$\sF(Y,\eps) = 0$ has a unique solution $Y^*(\eps)$ in a neighbourhood 
of $(Y^*,0)$, such that $Y^*(0) = Y^*$. Furthermore, by regular perturbation 
theory, this solution remains an attracting fixed point of the Poincar\'e map 
$\Pi$ for sufficiently small $\eps$. 
\end{proof}

\begin{remark}
In the spirit of Section~\ref{ssec:control_derivative}, one can show 
by a perturbation argument that the attracting periodic orbit constructed in 
the above proof attracts a ball in $H^2$-norm, which allows in particular 
to control derivatives of the solution.
\end{remark}


\subsection{The case $\eta > 0$}
\label{ssec:unperturbed_eta} 

Here, instead of~\eqref{eq:integral_Y}, it will be sufficient 
to use the fact that 
\begin{equation}
\label{eq:conservation_eta_pos} 
 \int_0^1 Y(\theta,m)\6m = -2M_- 
 \qquad \forall \theta\in I_1\;,
\end{equation} 
which has been established in Proposition~\ref{prop:conservation_Y_eta_pos}. 

\begin{proof}[{\sc Proof of Theorem~\ref{thm:main} when $\eta > 0$}]
The main difference with the case $\eta = 0$ is that we have to take 
care of the radial variable $r$. In order to take advantage of 
contraction along the stable parts of the critical manifold, it will 
be more convenient to slightly shift the Poincar\'e section. We 
therefore set $\theta_0 = T_1/2$ and define a Poincar\'e map $\Pi$ by 
\begin{equation}
\Pi: (Y(\theta_0,m), r(\theta_0),\eps,\eta) 
\longmapsto 
(Y(\theta_0 + T,m),r(\theta_0 + T))\;.
\end{equation} 
Note that the conservation law~\eqref{eq:conservation_eta_pos} implies 
\begin{equation}
 \int_0^1 Y(\theta_0,m)\6m 
 = \int_0^1 Y(\theta_0+T,m)\6m\;.
\end{equation} 
As before (modulo a change of integration constant in the function $\ph$), 
we can replace $Y$ by $\bar Y$. We write $\bar Y = Y^* + Z$, where $Z$ 
satisfies 
\begin{equation}
\label{eq:integral_Z} 
 \int_0^1 Z(\theta_0,m) \6m = 0\;,
\end{equation} 
so that $\bar Y(\theta_0,m)$ is compatible with~\eqref{eq:conservation_eta_pos}. 
By the same argument as in the previous section, we obtain  
\begin{equation}
 Z(\theta_0 + T,\cdot) 
 = \e^{-\eps T\tilde\sA} Z(\theta_0,\cdot) 
 + \Order{\eps\norm{Z(\theta_0,\cdot)}_{L^2}^2}
 + \Order{\eps^2+\eps\eta^{2/3}}\;,
 \label{eq:ZthetaT} 
\end{equation} 
the only difference being the additional error term in $\eps\eta^{2/3}$.
Note that the dependence on $r(0)$ is here hidden in the error terms. 
It remains to deal with the radial component, which we write as 
\begin{equation}
 r(\theta) = r^*_0(\theta;\eta) + q(\theta;\eta)\;,
\end{equation} 
where we recall that $r^*_0(\theta;\eta)$ denotes the periodic orbit of the 
case $\eps = 0$ (cf.~\eqref{eq:rstar_asymp}). 
Since $r(\theta)$ satisfies~\eqref{eq:dthetar}, we obtain 
\begin{equation}
 \eta\frac{\6q}{\6\theta}
 = \partial_r B_0(\theta, r^*_0(\theta;\eta))q 
 + \beta(q) 
 + \eps B_1(\theta, r^*_0(\theta;\eta) + q(\theta), \partial_m X(\theta,0); 
 \eta,\eps)\;,
\end{equation}
where $\beta(q) = \Order{q^2}$. 
By a similar variation-of-constants argument as in the previous proof, 
the solution of this equation satisfies 
\begin{align}
 q(\theta) = \e^{\alpha(\theta,\theta_0)/\eta}q(\theta_0) 
 &{}+ \frac{1}{\eta} \int_{\theta_0}^\theta \e^{\alpha(\theta,s)/\eta}\beta(q(s))\6s \\
 &{}+ \frac{\eps}{\eta} \int_{\theta_0}^\theta \e^{\alpha(\theta,s)/\eta}
 B_1(s, r^*_0(s) + q(s), \partial_m X(s,0); \eta,\eps) \6s\;,
\end{align} 
where
\begin{equation}
 \alpha(\theta_2,\theta_1) = 
 \int_{\theta_1}^{\theta_2} \partial_r B_0(s, r^*_0(s;\eta))\6s\;.
\end{equation} 
Using~\eqref{eq:dB0}, one obtains that $\alpha(\theta_0 + T,\theta_0) =: -\kappa$ is strictly 
negative, while 
\begin{equation}
 \alpha(\theta + T,\theta) \leqs 
 -a_1 (\theta_0 + T - \theta)
\end{equation}
for all $\theta \leqs \theta_0 + T$, for a constant $a_1 > 0$. 
This is because the integral is dominated by $s$ close to 
$\theta + T$, where $\partial_r B_0$ is negative. 
This shows in particular that 
\begin{equation}
 \int_{\theta_0}^{\theta_0 + T} \e^{\alpha(\theta_0 + T,\theta)/\eta} \6\theta 
 \leqs \frac{\eta}{a_1}\;.
\end{equation} 
From this bound and the bounds~\eqref{eq:B1} on $B_1$, we obtain, 
using as before a suitable first-exit time $\tau$,
\begin{equation}
\label{eq:q_theta0T} 
 q(\theta_0 + T) = \e^{-\kappa/\eta} q(\theta_0) 
 + \Order{q(\theta_0)^2}
 + \Order{\eps\eta^{2/3}}\;.
\end{equation} 
Writing $\Pi = (\Pi_Y,\Pi_r)$, we define the map 
\begin{equation}
 \sF(Y,r,\eps,\eta) 
 = \biggpar{\frac{\Pi_Y(Y,r,\eps,\eta) - Y}{\eps}, \Pi_r(Y,r,\eps,\eta) - r}\;.
\end{equation} 
Then~\eqref{eq:ZthetaT} and~\eqref{eq:q_theta0T} imply 
\begin{multline}
 \sF(Y^*+Z, r^*_0+q,\eps,\eta) \\
 = \bigpar{-T\tilde\sA Z + \Order{\norm{Z}_{L^2}^2} + \Order{\eps+\eta^{2/3}}, 
 (\e^{-\kappa/\eta} - 1)q + \Order{q^2} + \Order{\eps\eta^{2/3}}}\;.
\end{multline} 
In particular, we have 
\begin{equation}
 \sF(Y^*, r^*_0,0,0) = (0,0)\;,
\end{equation} 
and the dervative in the direction $(Z,q)$ satisfies 
\begin{equation}
 D\sF(Y^*, r^*_0,0,0)[Z,q] = 
 \begin{pmatrix}
  -T\tilde\sA Z & 0 \\
  0 & -1
 \end{pmatrix}\;.
\label{eq:DF_eta_positive} 
\end{equation} 
The linear map $\tilde\sA$, when restricted to the invariant subspace given 
by~\eqref{eq:integral_Z}, has no zero eigenvalue, since the only zero eigenvalue 
is associated to the conservation law~\eqref{eq:conservation_eta_pos}. 
Therefore, as before, \eqref{eq:DF_eta_positive} is an invertible linear map, so that 
the implicit function theorem yields uniqueness of the fixed point of $\Pi$ for small 
positive $\eps$ and $\eta$. By regular perturbation theory, this fixed point remains 
asymptotically stable for sufficiently small $\eps$ and $\eta$.
\end{proof}

\begin{remark}
The strong contraction, at rate $1/\eta$, in the fast $E$-direction implies that 
during the first oscillation, the system will reach a neighborhood of size 
$\eta^{2/3}$ at most of the singular periodic orbit in the $(E,M)$ plane. 
However, during subsequent oscillations, the convergence rate in the $E$-direction 
will in general be of order $\eps$. 
\end{remark}%


\appendix 

\section{Existence and uniqueness of solutions}
\label{app:existence} 

The aim of this appendix is to provide the proof of the existence and uniqueness 
result Theorem~\ref{thm:existence}. This is standard in the uncoupled case, 
see for instance~\cite[Chapter II]{Godlewski_Raviart} or~\cite[Chapter~6]{Serre_book}, 
but some work is required to extend the result to the coupled case. 

We will denote by $L$ a common Lipshitz constant of the functions $f$, 
$a$ and $g$, and by $K$ a constant such that 
\begin{align}
 \abs{a(E,M)} & \leqs K \quad \forall E, M\in\R\;, \\
 \norm{\partial_x n^*}_{H^2} &\leqs K\;, 
 &\norm{f(n^*)}_{H^2} &\leqs K\;, \\
 \abs{f(n^*(0))} &\leqs K\;, 
 &\abs{\partial_x n^*(0)} &\leqs K\;.
\end{align}
We start by noting that the function $u(t,x) = n(t,x) - n^*(x)$ satisfies 
\begin{equation}
 \partial_t u(t,x) = \eps \partial_{xx} u(t,x)
 -a(E(t),M(t)) \partial_x \bigbrak{f(n^*(x)+u(t,x))}
 + \eps \partial_{xx}n^*(x)\;.
\end{equation} 
For a constant $\lambda > 0$, we perform the scaling 
\begin{equation}
 u(t,x) = \e^{\lambda t} \tilde u(t,x)\;, \qquad 
 e(t) = \e^{\lambda t} \tilde e(t)\;,\qquad 
 m(t) = \e^{\lambda t} \tilde m(t)\;, 
\end{equation} 
which results in the system
\begin{align}
 \partial_t \tilde u &= \eps\partial_{xx}\tilde u -\lambda \tilde u + b(t,x)\;, \\
 \dot{\tilde E} &= -\lambda \tilde E + c(t)\;, \\
 \dot{\tilde M} &= -\lambda \tilde M + k(t)\;,
\label{eq:u_scaled} 
\end{align} 
where 
\begin{align}
 b(t,x) &= -e^{-\lambda t}a(\e^{\lambda t}\tilde E(t),\e^{\lambda t}\tilde M(t))
 \partial_x \bigbrak{f(n^*(x)+\e^{\lambda t}\tilde u(t,x))}
 + \eps e^{-\lambda t}\partial_{xx}n^*(x)\;, \\
 \label{eq:bck} 
 c(t) &= \e^{-\lambda t}\eta^{-1} g(\e^{\lambda t}\tilde E(t),\e^{\lambda t}\tilde M(t)) \;,\\
 k(t) &= e^{-\lambda t}\bigbrak{-a(\e^{\lambda t}\tilde E(t),\e^{\lambda t}\tilde M(t)) 
 f(n^*(0)+\e^{\lambda t}\tilde u(t,0))
 - \eps \bigpar{\partial_x n^*(0) + \e^{\lambda t}\partial_x \tilde u(t,0)}}\;.
\end{align} 
We are going to apply Banach's fixed-point theorem on a bounded time interval $[0,T]$, 
and then iterate the bound to cover all positive times. Given $R > 0$ and initial conditions 
$\tilde u_0\in H^2(\R) \cap L^1(\R)$ and $\tilde e_0\in \R$, we define the set of 
functions 
\begin{align}
\label{eq:def_BRT_scaled} 
 \sB_{R,T} = 
 \bigsetsuch{(\tilde u,\tilde e,\tilde m)\;}{\;&\tilde u(0) = \tilde u_0, \tilde e(0) = \tilde e_0, 
 \tilde m(0) = \tilde m_0, \\
 &\norm{\tilde u(t,\cdot)}_{H^2} \leqs R, \abs{\tilde e(t)}\leqs R, \abs{\tilde m(t)}\leqs R 
 \;\forall t\in[0,T]}\;,
\end{align} 
where 
\begin{equation}
 \tilde m(0) = -\int_{-\infty}^0 \bigbrak{n^*(x) + \tilde u_0(x)}\6x\;.
\end{equation} 
We consider the effect of the map 
\begin{equation}
 \sF_\lambda: (\tilde u,\tilde e,\tilde m)
 \mapsto (\tilde U,\tilde E,\tilde M)\;,
\end{equation} 
where $(\tilde U, \tilde E, \tilde M)$ solve
\begin{align}
\label{eq:UEM} 
 \partial_t \tilde U &= \eps\partial_{xx}\tilde U -\lambda \tilde U + b(t,x)\;, \\
 \dot{\tilde E} &= -\lambda \tilde E + c(t)\;, \\
 \dot{\tilde M} &= -\lambda \tilde M + k(t)\;,
\end{align} 
with initial conditions $\tilde u_0$, $\tilde e_0$ and $\tilde m_0$. 
Here $c$ and $k$ are defined as in~\eqref{eq:bck} above, but with $\tilde e$ and 
$\tilde m$ instead of $\tilde E$ and $\tilde M$. We want to show that for suitable 
$T$ and $\lambda$, the map $\sF_\lambda$ leaves $\sB_{R,T}$ invariant, and is a 
strict contraction in that set. In addition, we want to have a stronger control 
on the norms at time $T$, in order to be able to iterate the bound. 

Let $\e^{\eps t\Delta}u_0$ denote the unique solution of the scaled heat equation 
$\partial_t u = \eps \partial_{xx}u$ with initial condition $u(0,\cdot) = u_0$, given 
by convolution with the scaled heat kernel. Then Duhamel's formula shows that 
\begin{equation}
 \tilde U(t,x) = \e^{-\eps t\lambda}\e^{\eps t\Delta}\tilde u_0(x) 
 + \int_0^t \e^{-\eps (t-s)\lambda} \e^{\eps(t-s)\Delta}b(s,x) \6s\;.
\end{equation} 
The following bound is classical, and can be proved using the Fourier transform.

\begin{lemma}
\label{lem:bound_inhom_heat_scaled} 
Assume that $\tilde u_0 \in H^2$, while 
$b(t,\cdot) \in H^1$ for all $t\in[0,T]$. Then 
one has for all $t\in[0,T]$ 
\begin{equation}
 \norm{\tilde U(t,\cdot)}_{H^2}
 \leqs \e^{-\lambda t}\norm{\tilde u_0}_{H^2}
 +  \alpha(T)
 \sup_{0\leqs t\leqs T} \norm{b(t,\cdot)}_{H^1}
\end{equation} 
where 
\begin{equation}
 \alpha(T) = T + \frac{2C_1\sqrt{T}}{\sqrt{\eps}}\;.
\end{equation} 
for a constant $C_1$ independent of $T$ and $\eps$. 
\end{lemma}

\begin{proposition}
Assume $\tilde u_0 \in H^2(\R) \cap L^1(\R)$, and fix $\tilde e_0\in\R$. 
Let $R = 2(\norm{\tilde u_0}_{H^2} \vee \tilde e_0 \vee \tilde m_0)$.
Then there exists $T$, depending only on $R$, $\eps$, $\eta$, $K$ 
and $L$, such that for $\lambda = T^{-1}\log(2)$,
$\sB_{R,T}$ is invariant under $\sF_\lambda$, and 
$\sF_\lambda$ is strictly contracting on $\sB_{R,T}$. 
In addition, we have 
\begin{equation}
\label{eq:endpoint} 
 \norm{\tilde U(T,\cdot)}_{H^2} \leqs \frac{R}{2}\;, \qquad 
 \abs{\tilde E(T)} \leqs \frac{R}{2}\;, \qquad 
 \abs{\tilde M(T)} \leqs \frac{R}{2}\;. 
\end{equation} 
\end{proposition}
\begin{proof}
We start by checking that $\sB_{R,T}$ is invariant, and showing the 
end-point estimates~\eqref{eq:endpoint}. Lemma~\ref{lem:bound_inhom_heat_scaled}
yields 
\begin{equation}
 \norm{\tilde U(t,\cdot)}_{H^2} 
 \leqs \e^{-\lambda t} \frac R2 + \alpha(T) \sup_{0\leqs t\leqs T}
 \norm{b(t,\cdot)}_{H^1}\;,
\end{equation} 
where for all $t\in[0,T]$ 
\begin{align}
 \norm{b(t,\cdot)}_{H^1}
 &\leqs \e^{-\lambda t} \abs{a} \norm{\partial_x(f(n^* + \e^{\lambda t}\tilde u))}_{H^1}
 + \eps\e^{-\lambda t}\norm{\partial_{xx}n^*}_{H^1} \\
 &\leqs \e^{-\lambda t} K \norm{f(n^* + \e^{\lambda t}\tilde u)}_{H^2}
 + \eps\e^{-\lambda t}\norm{\partial_xn^*}_{H^2}\;.
\end{align}
Lipshitz continuity of $f$ yields
\begin{equation}
 \norm{f(n^* + \e^{\lambda t}\tilde u)}_{H^2} 
 \leqs  \norm{f(n^*)}_{H^2} + \e^{\lambda t}L\norm{\tilde u}_{H^2} 
 \leqs K + \e^{-\lambda t}LR\;.
\end{equation} 
This implies 
\begin{equation}
 \norm{\tilde U(t,\cdot)}_{H^2} 
 \leqs \e^{-\lambda t} \frac R2 + \alpha(T) K(LR + K + \eps)\;.
\end{equation} 
Choosing $T$ such that 
\begin{equation}
\label{eq:cond_T1} 
 \alpha(T) \leqs \frac{R}{4K(LR + K + \eps)}
\end{equation} 
ensures that $\norm{\tilde U(t,\cdot)}_{H^2} \leqs R$ for all $t\in[0,T]$, while 
the definition of $\lambda$ in terms of $T$ yields $\e^{-\lambda T} = \frac12$, 
and ensures the required end-point estimate 
$\norm{\tilde U(T,\cdot)}_{H^2} \leqs \frac R2$. 

The bounds for $\tilde E$ and $\tilde M$ are obtained in a similar way, so we only 
consider the case of $\tilde M$. We have 
\begin{equation}
 \tilde M(t) = \e^{-\lambda t} \tilde m_0 
 + \e^{-\lambda t} \int_0^t \e^{\lambda s}k(s)\6s\;. 
\end{equation} 
Here we use Morrey's inequality, which allows us to bound the H\"older $\cC^{1/2}$ 
norm in terms of the $H^1$ norm. Namely, there exists a universal constant 
$C$ such that 
\begin{equation}
 \abs{\tilde u(t,0)}, \abs{\partial_x \tilde u(t,0)} 
 \leqs C\norm{\tilde u}_{H^2} \leqs CR\;.
\end{equation} 
This yields 
\begin{align}
 \e^{\lambda s} \abs{k(s)} 
 &\leqs K \bigabs{f(n^*(0) + \e^{\lambda s}\tilde u(s,0))}
 + \eps \bigpar{\abs{\partial_x n^*(0)} + \e^{\lambda s}\abs{\partial_x \tilde u(s,0)}} \\
 &\leqs \e^{\lambda s} \bigbrak{K(K + LCR + \eps) + \eps CR}\;.
\end{align}
Therefore, imposing 
\begin{equation}
\label{eq:cond_T3} 
 T \leqs \frac{R}{4\bigbrak{K(K + LCR + \eps) + \eps CR}}
\end{equation} 
entails the required bounds for $\tilde M(t)$, $t\in[0,T]$, and 
$\tilde M(T)$. The bounds for $\tilde E$ hold under the additional 
condition 
\begin{equation}
 T \leqs \frac{\eta}{4L}\;.
\end{equation} 
We now turn to proving that $\sF_\lambda$ is contracting. 
For that, we pick $(\tilde u_1, \tilde e_1, \tilde m_1)$
and $(\tilde u_2, \tilde e_2, \tilde m_2)$ in $\sB_{R,T}$ such that 
\begin{equation}
 \norm{\tilde u_1(t,\cdot) - \tilde u_2(t,\cdot)}_{H^2} \leqs r\;, \qquad 
 \abs{\tilde e_1(t) - \tilde e_2(t)}\leqs r \;, \qquad 
 \abs{\tilde m_1(t) - \tilde m_2(t)}\leqs r 
\end{equation} 
for all $t\in[0,T]$. Then the difference $\tilde U_1 - \tilde U_2$ satisfies 
\begin{equation}
 \partial_t(\tilde U_1 - \tilde U_2) 
 = \eps \partial_{xx} (\tilde U_1 - \tilde U_2) 
 - \lambda (\tilde U_1 - \tilde U_2) 
 + b_1 - b_2\;,
\end{equation} 
where the $b_i$ (and below respectively $a_i,f_i$) are given by $b$ (respectively $a,f$) evaluated for $(\tilde u_i, \tilde e_i, \tilde m_i)$. 
Lemma~\ref{lem:bound_inhom_heat_scaled} yields 
\begin{align}
 \sup_{0\leqs t \leqs T} \norm{(\tilde U_1 - \tilde U_2)(t,\cdot)}_{H^2}
 &\leqs \alpha(T) \sup_{0\leqs t \leqs T} \norm{(b_1 - b_2)(t,\cdot)}_{H^1} \\
 &\leqs \alpha(T) \sup_{0\leqs t \leqs T} \norm{(B_1 - B_2)(t,\cdot)}_{H^2}\;,
\end{align} 
where the $B_i$ are primitives of the $b_i$ satisfying 
\begin{align}
 \norm{(B_1 - B_2)(t,\cdot)}_{H^2}
 &\leqs \e^{-\lambda t}\abs{a_1 - a_2} \norm{f_1}_{H^2}
 + \e^{-\lambda t}\abs{a_2} \norm{f_1 - f_2}_{H^2} \\
 &\leqs 2Lr \bigbrak{K + L\e^{\lambda t}R}
 + KLr \\
 &\leqs Lr\e^{\lambda t} \bigbrak{3K + 2LR}\;.
\end{align}
Therefore, imposing 
\begin{equation}
\label{eq:cond_T4} 
 2\alpha(T) L\bigbrak{3K + 2LR} < 1 
\end{equation} 
yields contraction of the $\tilde U$ component. Contraction in the $\tilde E$ 
and $\tilde M$ directions is obtained in a similar way. 
\end{proof}

\begin{corollary}
Assume $\tilde u_0\in H^2(\R)\cap L^1(\R)$ and fix $\tilde e_0$. Then the 
scaled system~\eqref{eq:u_scaled} admits a unique solution which is global in time. 
\end{corollary}
\begin{proof}
The previous theorem shows that there is a unique solution on the time interval 
$[0,T]$ for suitable $T$ and $\lambda$. Since $\tilde u(T,x)$ and $\tilde e(T)$ 
satisfy the same bounds as $\tilde u(0,x)$ and $\tilde e(0)$, 
and the bound on $\tilde M(T)$ shows that $\tilde u(T,\cdot)\in L^1$, 
we can apply the 
theorem again to obtain a unique solution on $[0,2T]$. The result then 
follows by induction. 
\end{proof}

To complete the proof of Theorem~\ref{thm:existence}, it remains to show that 
$u$ satisfies the boundary conditions, and that the integral $P$ is conserved. 
First note that since $\tilde u(t,\cdot) \in H^2$ for all $t$, the same 
holds for $u(t,\cdot)$. Now if $\hat u$ denotes the Fourier transform of 
$u$, then the Cauchy--Schwarz inequality gives 
\begin{align}
 \norm{\hat u}_{L^1} 
 &= \norm{(1+\abs{\xi})^{-1}(1+\abs{\xi})\hat u}_{L^1} \\
 &\leqs \biggpar{\int_\R\frac{\6\xi}{(1+\abs{\xi})^2}}^{1/2} 
 \bigpar{\norm{\hat u}_{L^2} + \norm{\xi \hat u}_{L^2 }} \\
 &\leqs c\norm{u}_{H^1}
 \label{eq:Fourier-L1} 
\end{align}
for a finite $c$. Therefore $\hat u\in L^1$, and the Riemann--Lebesgue 
lemma implies that $u(t,x)$ has to converge to $0$ as $x\to\pm\infty$. 
The same argument works for $\partial_x u$, which belongs to $H^1$. This 
shows that $u(t,\cdot)$ satisfies the boundary conditions~\eqref{eq:bc}. 

Since $f(0) = f(1) = 0$, it also follows that 
\begin{equation}
 \lim_{x\to\pm\infty} f(n(t,x)) = 0\;.
\end{equation} 
Therefore, 
\begin{equation}
 \lim_{L\to\infty} 
 \int_{-L}^L \partial_x \bigbrak{f(n(t,x))} \6x 
 = \lim_{L\to\infty} \bigbrak{f(n(t,L)) - f(n(t,-L))}
 = 0\;.
\end{equation} 
A similar argument shows that 
\begin{equation}
 \lim_{L\to\infty} 
 \int_{-L}^L \partial_{xx} u(t,x) \6x = 0\;. 
\end{equation} 
It follows that 
\begin{equation}
 \lim_{L\to\infty} 
 \frac{\6}{\6t} \int_{-L}^L \bigbrak{n(t,x) - n^*(x)} \6x 
 = \lim_{L\to\infty} \int_{-L}^L \partial_t n(t,x) \6x 
 = 0\;.
\end{equation} 


\tableofcontents


\bibliographystyle{abbrv}
\bibliography{BBH}

\bigskip\bigskip\noindent
{\small
Julien Barr\'e \\
Institut Denis Poisson (IDP) \\ 
Universit\'e d'Orl\'eans, Universit\'e de Tours, CNRS -- UMR 7013 \\
B\^atiment de Math\'ematiques, B.P. 6759\\
45067~Orl\'eans Cedex 2, France \\
{\it E-mail address: }
{\tt julien.barre@univ-orleans.fr}

\bigskip\bigskip\noindent
{\small
Nils Berglund \\
Institut Denis Poisson (IDP) \\ 
Universit\'e d'Orl\'eans, Universit\'e de Tours, CNRS -- UMR 7013 \\
B\^atiment de Math\'ematiques, B.P. 6759\\
45067~Orl\'eans Cedex 2, France \\
{\it E-mail address: }
{\tt nils.berglund@univ-orleans.fr}

\bigskip\bigskip\noindent
{\small
Hiroshi Horii \\
Universit\'e C\^ote d'Azur, Observatoire de la C\^ote d'Azur, CNRS \\
Laboratoire Lagrange, UMR 7293, Nice, France \\
{\it E-mail address: }
{\tt hiroshi.horii@oca.eu}

\end{document}